\documentclass[12pt]{article}
\usepackage{amsmath,amssymb,amsfonts,amsthm}
\newenvironment{myabstract}{\par\noindent
{\bf Abstract . } \small }
{\par\vskip8pt minus3pt\rm}
\newcounter{item}[section]
\newcounter{kirshr}
\newcounter{kirsha}
\newcounter{kirshb}
\newenvironment{enumroman}{\setcounter{kirshr}{1}
\begin{list}{(\roman{kirshr})}{\usecounter{kirshr}} }{\end{list}}
\newenvironment{enumarab}{\setcounter{kirshb}{1}
\begin{list}{(\arabic{kirshb})}{\usecounter{kirshb}} }{\end{list}}
\newenvironment{athm}[1]{\vskip3mm\par\noindent%\stepcounter{item}
{\bf #1 }. \slshape }
{\upshape\par\vskip10pt minus3pt}
\newtheorem{theorem}{Theorem}[section]

\newtheorem{lemma}[theorem]{Lemma}
\newtheorem{corollary}[theorem]{Corollary}
\newenvironment{demo}[1]{\noindent{\bf #1.}\upshape\mdseries}
{\nopagebreak{\hfill\rule{2mm}{2mm}\nopagebreak}\par\normalfont}
\theoremstyle{definition}
\newtheorem{remark}[theorem]{Remark}

\newtheorem{example}[theorem]{Example}
\newtheorem{definition}[theorem]{Definition}

\def\C{{\mathfrak{C}}}
\def\Fm{{\mathfrak{Fm}}}

\def\Nr{{\mathfrak{Nr}}}
\def\Fr{{\mathfrak{Fr}}}
\def\Sg{{\mathfrak{Sg}}}
\def\Zd{{\mathfrak{Z}}}
\def\Fm{{\mathfrak{Fm}}}
\def\F{{\sf F}}
\def\Co{{\sf Co}}
\def\A{{\mathfrak{A}}}
\def\B{{\mathfrak{B}}}
\def\C{{\mathfrak{C}}}
\def\D{{\mathfrak{D}}}
\def\M{{\mathfrak{M}}}

\def\Sn{{\mathfrak{Sn}}}
\def\CA{{\bf CA}}

\def\K{{\bf K}}
\def\K{{\bf K}}

\def\Rd{{\ Rd}}
\def\(R)RA{{\bf (R)RA}}

\def\P{{\mathfrak {P}}}

 \def\CA{{\sf CA}}
\def\B{{\sf B}}
\def\G{{\sf G}}

\def\K{{\sf K}}
\def\tp{{\sf tp}}
 
\def\Nr{{\mathfrak{Nr}}}

\def\Nr{{\mathfrak{Nr}}}

\def\A{{\mathfrak{A}}}
\def\B{{\mathfrak{B}}}
\def\C{{\mathfrak{C}}}
\def\D{{\mathfrak{D}}}

\def\A{{\mathfrak{A}}}
\def\B{{\mathfrak{B}}}
\def\C{{\mathfrak{C}}}
\def\D{{\mathfrak{D}}}

\def\Ig{{\mathfrak{Ig}}}
\def\Alg{{\mathfrak{Alg}}}
\def\Bl{{\mathfrak{Bl}}}
\def\L{{\mathfrak{L}}}
\def\Rl{{\mathfrak{Rl}}}

\def\L{{\mathfrak{L}}}
\def\At{{\sf At}}
\def\CA{{\bf CA}}

\def\G{{\bf G}}

\def\Fl{{\mathfrak{Fl}}}

\def\CM{{\bf CM}}

\title{Free algebras, amalgamation and omitting types in $BL$ algebars with operators}
\author{Tarek Sayed Ahmed \\
Department of Mathematics, Faculty of Science,\\ 
Cairo University, Giza, Egypt.
  }
\begin{document}
\maketitle
\begin{myabstract} Let $K$ be a class of $BL$ algebras with operators. $BL$ algebras are algebraisations of many fuzzy logics, 
they are extensions of both Boolean algebras, and $MV$ algebras, 
the latter algebraize many-valued logic. We study atomicity of free algebras, the amalgamation property, 
and the algebraic counterpart of omitting types theorem for $K$. 
%Our proof is topological based on the Baire category theorem. Our topological proof renders a set theoretic investigation possible.
%Finally a variant of Vaught's conjecture is investigated in the context of G\"odel's intuitionistic logic.
\footnote{Mathematics Subject Classification. 03G15; 06E25}
 \end{myabstract}

\section{Introduction}

A residuated lattice is an algebra
$$(L,\cup,\cap, *, \implies 0,1)$$ 
with four binary operations and two constants such that
\begin{enumroman}
\item $(L,\cup,\cap, 0,1)$ is a lattice with largest element $1$ and the least element $0$ (with respect to the lattice ordering defined the usual way: 
$a\leq b$ iff $a\cap b=a$).
\item $(L,*,1)$ is a commutative semigroup with largest element $1$, that is $*$ is commutative, associative, $1*x=x$ for all $x$.
\item Letting $\leq$ denote the usual lattice ordering, we have $*$ and $\implies $ form an adjoint pair, i.e for all $x,y,z$
$$z\leq (x\implies y)\Longleftrightarrow x*z\leq y.$$
\end{enumroman}

$BL$ algebras, introduced and studied by Hajek \cite{H}, are what is called $MTL$ algebras satisfying the identity $x*(x\implies y)=x\cap y$.
Both are residuated lattices with extra conditions. The propositional logic $MTL$ was introduced by
Esteva and Godo \cite{E}.  It has three basic connectives $\to$, $\land$ and $\&$.
We say that $\L$ is a core fuzzy logic if $\L$ expands $MTL$, $\L$ has the Local Deduction Theorem $(LDT)$, 
and $\L$ satisfies 
(*) $\phi\equiv \psi\vdash \chi(\phi)\equiv \chi(\psi)$ for all formulas $\phi,\psi,\chi$. (Here $\equiv$ is defined via $\&$ and $\implies$).
The ($LDT)$ says that for a theory $T$ and a formula 
$\phi$, whenever  $T\cup \{\phi\}\vdash \psi$, then there exists a natural number $n$ such that $T\vdash \phi^n\to \psi$.
Here $\phi^n$ is defined inductively by $\phi^1=\phi$ and $\phi^n=\phi^{n-1}\&\phi$.
Thus core fuzzy logics are axiomatic expansions of $MTL$ having $LDT$ and obeying the substitution rule (*).
The basic notions of evaluation, tautology and model  for core fuzzy logics  are defined the usual way. 
Let $\L$ be a core fuzzy logic and $I$ the set of additional connectives of $\L$. An $\L$ algebra is a structure 
$\B=(B, \cup, \cap, *, \implies, (c_B)_{c\in I},0,1)$
such that $(B,\cup, \cap, *, \implies, 0,1)$ is an $MTL$ algebra and each additional axiom of $\L$ is a tautology of $\B$. 
Throughout the paper the operations of algebras are denoted by $\cup$, $\cap$, $\implies$ 
$*$ and the corresponding logical operations by $\lor, \land, \to, \&$.

On the other hand $MV$ algebras introduced by Chang in 1958 to provide an algebraic reflection of the completeness theorem of 
the Lukasiewicz infinite valued propositional logic, are $BL$ algebras with the law of double negation. They can also be recovered from Boolean algebras by dropping 
idempotency. In recent years the range of applications of $MV$ algebras
has been enormously extended with profound interaction with other topics, ranging from  
lattice ordered abelian groups, $C^*$ algebras, to fuzzy logic.
In this paper we study $MV$ algebras in connection to fuzzy (many valued) logic. 

An $MV$ algebra, has a dual behaviour; it  can be viewed,  
in one of its  facets, as  a `non-idempotent' generalization of a Boolean algebra possesing a strong lattice structure. 
The lack of idempotency enables $MV$ algebras  to be compared to monodial structures like monoids and abelian groups. 
Indeed, the category of $MV$ algebras has been shown to be equivalent to the category of $l$ groups. At the same time the lattice structure 
of Boolean algebras can be recovered inside $MV$ algebras, by an appropriate term definability of primitive connectives.  In this 
respect, they  have a strong lattice structure 
(distributive and bounded), which make the techniques of lattice theory readily applicable to their study.
As shown in this paper, in certain contexts when we replace the notion of a Boolean algebra with an $MV$ algebra, 
the results survive such a replacement with
some non-trivial modifications, and this can be accomplished in a somewhat  unexpected manner. 

Boolean algebras work as the equivalent semantics of classical propositional logic. To sudy classical first order logic, 
Tarski \cite{HMT1}, \cite{HMT2} introduced cylindric algebras, while Halmos introduced polyadic algebras. 
Both of those can be viewed as Boolean algebras with extra operations that reflect algebraically 
existential quantifiers. 

Boolean algebras also have a neat and intuitive depiction, modulo isomorphisms; 
any Boolean algebra is an algebra of subsets of some set endowed 
with the concrete set theoretic operations of union, intersection and complements. Such a connection, a typical duality theorem, 
is today well understood. These nice properties mentioned above is formalized through the topology of 
Stone spaces that allows to select the right objects in the full power set of some set, the underlying set of the 
associated topological space. The representation theory of cylindric algebras, on the other hand, proves much more involved, 
 and lacks such a strong well understood duality theorem like that of 
Boolean algebras. However, there is an extension of Stone duality to cylindric algebras, due to Comer, where he establishes a 
dual equivalence between cylindric algebras 
and certain categories of sheaves; but such a duality does not go deeper into the analysis of representability.
There is a version of concrete  (representable) algebras for cylindric algebras, with extra operations interpreted as projections,
but this does not coincide with the abstract class of cylindric algebras. This is in sharp contrast to 
Boolean algebras. It is not the case 
that every cylindric algebra is representable in a concrete manner with the operations being set theoretic operations on relations. Not only that, 
but in fact the class of representable algebras need an infinite axiomatization in first order logic, and for any such axiomatization, 
there is an inevitable degree of complexity.
On the other hand, polyadic algebras enjoy a strong representation theorem; every polyadic algebra is representable \cite{DM}.
%Here we apply the theory of polyadic algebras to $MV$ algebras. The idea is to study transformation systems based on such algebras.

Cylindric and relation algebras were introduced by Tarski to algebraize first order logic. 
The structures of free cylindric and relation algebras are quite rich 
since they are able to capture the whole of first order logic, in a sense.
One of the first things to investigate about these free algebras is whether they 
are atomic or not, i.e. whether their boolean reduct is atomic or not.
By an atomic boolean algebra we mean an algebra for which  
below every non-zero element there is an atom, i.e. a minimal non-zero 
element. Throughout $n$ will denote a countable cardinal (i.e. $n\leq \omega$). 
More often than not, $n$ will be 
finite.
$\CA_n$ stands for the class of cylindric algebras of dimension $n$.
For a class $K$ of algebras, and a cardinal $\beta>0$, 
$\Fr_{\beta}K$ stands for the $\beta$-generated free
$K$ algebra.
In particular, $\Fr_{\beta}\CA_n$ 
denotes the $\beta$-generated free cylindric algebra
of dimension 
$n$. The following is known:
If $\beta\geq \omega$, then $\Fr_{\beta}\CA_n$ is atomless (has no atoms)
 [Pigozzi \cite{HMT1} 2.5.13].
Assume that $0<\beta<\omega$. If $n <2$ then $\Fr_{\beta}\CA_n$ is finite,
hence atomic, 
\cite{HMT1} 2.5.3(i).
$\Fr_{\beta}\CA_2$ is infinite but still atomic [Henkin, \cite{HMT1} 2.5.3(ii), 2.5.7(ii).]
If $3\leq n<\omega$, then $\Fr_{\beta}\CA_n$ has infinitely many atoms
[Tarski, \cite{HMT1} 2.5.9], and it was posed as an open question, cf \cite{HMT1} 
problem 4.14, whether it is atomic or not.
Here we prove, as a partial solution of problem 4.14 in \cite{HMT1}, and among other things, 
that $\Fr_{\beta}\CA_n$ is not atomic for $\omega>\beta>0$ and $\omega>n\geq 4.$ 
%We investigate free relation algebras, too. They are not atomic either.

In this paper we study atomicity of free $BL$ algebras with operators, we prove several amalgamation theorems for $MV$ 
algebras, obtaining several interpolation theorems for many valued logic. Intuitiionistic logic, does not belong to fuzzy logic per se, though linear Heyting algebras do.
However, here  we give a deep representation theorem for Heyting algebras, culminating in an interpolation theorem for many predicate intuitionistic 
logics. We use Sheaf theoretic duality theory as worked out by Comer for cylindric algebras, but now 
 applied to the Zarski toplogy defined on the prime spectrum of $BL$ algebras,
to obtain some results on definabality, mainly Beth definability for many valued logics.
Finally, we give a new topological  proof, using the celebrated Baire Category theorem,
 to the omitting types theorem for fuzzy logic, and we give several model 
theoretic consequence.

%Key words: basic fuzzy logic, omitting types, atomic models.}
%\end{myabstract}
\section{ Basic Fuzzy Logic}

The logics we start with arise typically from $t$ norms. 
\begin{definition}A $t$ norm is a binary operation $*$ on $[0,1]$, i.e $(t:[0,1]^2\to [0,1]$) such that
\begin{enumroman}
\item  $*$ is commutative and associative,
that is for all $x,y,z\in [0,1]$,
$$x*y=y*x$$
$$(x*y)*z=x*(y*z).$$
\item $*$ is non decreasing in both arguments, that is
$$x_1\leq x_2\implies x_1*y\leq x_2*y$$
$$y_1\leq y_2\implies x*y_1\leq x*y_2.$$
\item $1*x=x$ and $0*x=0$ for all $x\in [0,1].$
\end{enumroman}
\end{definition}
The following are the most important (known) examples of continuous $t$ norms.

\begin{enumroman}
\item Lukasiewicz $t$ norm: $x*y=max(0,x+y-1)$
\item Godel $t$ norm $x*y=min(x,y)$
\item Product $t$ norm $x*y=x.y$
\end{enumroman}
We have the following known result \cite{H} lemma 2.1.6
 
\begin{theorem} Let $*$ be a continuous $t$ norm. 
Then there is a unique operation $x\implies y$ satisfying for all $x,y,z\in [0,1]$, the condition $(x*z)\leq y$ iff $z\leq (x\implies y)$, namely 
$x\implies y=max\{z: x*z\leq y\}$
\end{theorem}
The operation $x\implies y$ is called the residuam of the $t$ norm. The residuam $\implies$ 
defines its corresponding unary operation of precomplement 
$(-)x=(x\implies 0)$.
The Godel negation satisfies $(-)0=1$, $(-)x=0$ for $x>0$.
Abstracting away from $t$ norms, we get $BL$ algebras as defined in the introduction.

The following variant of the completeness theorem for core fuzzy logics is known:

\begin{theorem} Let $\L$ be a core fuzzy logic, $\phi$ a formula and $T$ a theory. Then, the following conditions are equivalent
\begin{enumroman}
\item $T\vdash \phi$
\item $e(\phi)=1$ for each $\L$-algebra and each $\B$ model $e$ of theory $T$
\item $e(\phi)=1$ for each $\L$-chain $\B$ and each $\B$ model $e$ of theory $T$
\end{enumroman}
\end{theorem}
\begin{demo}{Proof} \cite{H}, Thm 5 p.867.
\end{demo}
Now we pass to predicate fuzzy logics, or predicate many valued logics. Let us assume from now on that $\L$ is some fixed core fuzzy logic.
A predicate language consists of non-logical symbols and logical symbols. The non-logical symbols 
consist  of a non-empty set of predicates, each together with a positive natural number - 
its arity, and a possibly empty set of constants. The logical symbols are a countable family of variables $x_1, \ldots x_n\ldots$
connectives $\&$,  $\rightarrow$, truth constants $0,1$ and quantifiers $\forall$ $\exists$. 
Terms consist of variables and constants and nothing else.
Atomic formulas have the form $P(t_1\ldots t_n)$ where $P$ is a predicate of arity $n$ and $t_1\ldots t_n$ are terms.
If $\phi,\psi$ are formulas and $x$ is a variable, then $\phi\to \psi$, $\phi\&\psi$, $(\forall x)\psi$ $(\exists x)\phi$
are formulas.
Other connectives are defined as follows:
$$\phi\land \psi\text {  is }\phi\&(\phi\rightarrow \psi),$$
$$\phi\lor \psi\text {  is }((\phi\to \psi)\to \psi)\land ((\psi\to \phi)\to \phi),$$
$$\neg \phi\text { is }\phi\to 0,$$
$$\psi\equiv \psi\text { is } (\phi\to \psi)\&(\psi\to \phi).$$

For a linearly ordered  $\L$ algebra, an $\bold L$ structure for a predicate language is $\M=(M, P_M, c_m),$ 
where $M\neq \emptyset$, for each predicate
$P$ of arity $n$, $P_M$ is an $n$-ary $\bold L$ fuzzy relation on $M,$ that is $P_M:M^n\to \bold L$,  and for each constant $c,$ $c_m\in M$. 
One then defines
for each formula $\phi$ the truth value $||\phi||_{M,v}^{\bold L}$ of $\phi$ in $\M$ determined by the $\bold L$ chain and 
evaluation $v$ of free variables the usual 
Tarskian way. 
In more detail, an $\M$ evaluation is a map from $\omega$ to $M$. For two evaluations $v$ and $v'$ and $i\in \omega$ we write 
$v\equiv_i v'$ iff $v(j)=v'(j)$ for all
$j\neq i$. 
The value of a term given by $\M,v$ is defined as follows $||x_i||_{\M,v}=v(i)$ and $||c||_{\M,v}=m_c$.
Now we define the truth value $||\phi||^{\bold L}_{\M,v}$:
$$||P(t_1,\ldots t_n)||_{\M,v}^{\bold L}=P_M(||t_1||_{\M,v}^{\bold L},\ldots ||t_n||_{M,v}^{\bold L}),$$
$$||\phi\to \psi||_{\M.v}^{\bold L}=||\phi||_{\M,v}^{\bold L}\implies ||\psi||_{\M,v}^{\bold L},$$
$$||\phi\&\psi||_{\M,v}^{\bold L}=||\phi||_{\M,v}^{\bold L}*||\psi||_{\M,v}^{\bold L},$$
$$||(\forall x_i)\phi||_{\M,v}^{\bold L}=\bigwedge\{||\phi||_{\M,v'}^{\bold L}: v'\equiv_i v\},$$ 
$$||(\exists x_i)\phi||_{\M,v}^{\bold L}=\bigvee\{||\phi||_{\M,v'}^{\bold L}: v'\equiv_i v\}.$$
The structure $\M$ is $\bold L$ safe if the needed infima and suprema exist, i.e $||\phi||_{\M,v}^{\bold L}$ is defined for all $\phi$ and $v$.
Let $\phi$ be a formula and $\M$ be a safe $\bold L$ structure. The truth value of $\phi$ in $\M$ is
$$||\phi||_{\M}^{\bold L}=\bigwedge \{||\phi||_{M,v}^{\bold L}: v \text { is an $M$ evaluation }\}.$$
For each model $(\M,L),$ let $\Alg(\M,\bold L)$ be the subalgebra of $\bold L$ with domain $\{||\phi||_{\M,v}^{\bold L}: \phi, v\}$ of 
truth degrees of all formulas $\phi$
under all $M$ evaluations $v$ of variables. $(\M,\bold L)$ is exhaustive if $\bold L=\Alg(\M,\bold L)$.
Notions of free variables, substitution of a term for a variable, are defined like in classical first order logic.
Given a safe structure  $(\M, \bold L)$, a formula $\phi(x_1\ldots x_n)$ having free variables 
among the first $n$ and $s\in {}^nM$, $s=(a_1\ldots a_n)$, say,  we write
$||\phi(a_1\ldots a_n)||_{\M, \bold L}$ or $||\phi(s)||_{\M, \bold L}$ for the value of the formula 
$\phi$ in $L$ when replacing the variables $x_1\ldots x_n$ by $a_1\ldots a_n$ respectively.

The following are logical axioms for quantifiers.

$(\forall 1) \ \ (\forall x)\psi(x)\to \psi(t)$, $t$ is substitutable for $x$ in $\psi(x),$

$(\exists 1)\ \ \psi(t)\to (\exists x) \psi(x)$, $t$ is substitutable for $x$ in $\psi(x),$

$(\forall 2) \ \ (\forall x)(\psi\to \phi)\to (\psi\to (\forall x)\phi),$ $x$ is not free in $\psi,$

$(\exists 2)\ \ (\forall x)(\psi\to \phi)\to ((\exists x)\psi\to \psi),$ $x$ is not free in $\psi,$

$(\forall 3) \ \ (\forall x)(\psi\lor \phi)\to (\psi\lor (\forall x)\phi),$ $x$ is not free in $\psi.$

Let $\L$ be a core fuzzy logic that extends the basic propositional logic $BL$.
We associate with $\L$ the corresponding predicate calculus $\L\forall$ over a given signature $S$
by taking as logical axioms 
\begin{itemize}

\item all formulas resulting from the axioms of $\L$ by substituting arbitary formulas of $S$ for propositional variables, and
the axioms $(\forall 1)$ $(\forall 2)$ $(\forall 3)$, $(\exists 1)$ $(\exists 2)$ for quantifiers 
and taking as deduction rules,

\item modus ponens (from $\phi, \phi\to \psi$ infer $\psi$) and,

\item generalization (from $\phi$ infer $(\forall x)\phi).$

\end{itemize}
Given this, the notions of proof, provability, theory, etc. are like classical logic.

\begin{definition}
Let $T$ be a theory. $T$ is linear if for every pair $\phi$, $\psi$ of sentences we have $T\vdash \phi\to \psi$ or $T\vdash \psi\to \phi$.
We say that $T$ is Henkin if for each sentence $\phi=\forall x\psi$ such that $T\nvdash \phi$, there is a constant $c$ such that
$T\nvdash \psi(c)$.

Set $[\phi]_T=\{\psi: T\vdash \psi\equiv \phi\}$ and $L_T=\{[\phi]_T: \phi \text { a formula }\}.$
The Lindenbaum algebra of the theory $T$ $(\Fm_T)$ has domain $L_T$ and operations 
$$f_{\Fm_T}([\phi_1]_T\ldots [\phi_n]_T)=[f(\phi_1\ldots \phi_n)]_T.$$
\end{definition}
Let $T$ be a Henkin Linear theory. The canonical model of theory $T$, denoted by $\CM(T),$ is the pair $(CM(T), \Fm_T)$, where $\Fm_T$
is the Lindenbaum algebra of the theory $T$, the domain  $CM(T)$ of $\CM(T)$ consists of the constants, $c_{\CM(T)}=c$ for each constant and
$P_{\CM(T)}(t_1\ldots t_n)=[P(t_1\ldots t_n)]_T$ for each predicate symbol $P$. 

\begin{lemma}Let $\L$ be a core fuzzy logic, $T$ a linear Henkin theory, and $\phi$ a formula with only one free variable $x$. Then
\begin{enumroman}
\item $\Fm_T$ is an $\L$-chain,
\item $[\forall x \phi]_T=\bigwedge [\phi(c)],$
\item $[\exists x\phi]_T=\bigvee [\phi(c)]$.
\item If $\phi$ is a sentence, then $||\phi||^{\CM(T)}=[\phi]_T$. Thus $T\vdash \phi$ iff $\CM(T)\models \phi$.
\end{enumroman}
\end{lemma}
\begin{demo}{Proof} \cite{H} lemma 6.
\end{demo}
The previous lemma gives the following completeness theorem \cite{Ha}:

\begin{theorem} Let $\L\forall$ be the predicate calculus given by a core fuzzy logic extending $BL$. 
Let $T$ be a theory over $\L\forall$ and let $\phi$ be a formula
of the language of $T$. $T$ proves $\phi$ if and only if for each linearly ordered $\L$-algebra $\bold L$ 
and every safe $\L$-model $\M$ of $\bold L$, we have  $||\phi||_{\M}^{\bold L}=1.$
\end{theorem}
The following corollary is immediate \cite{Ha}.

\begin{theorem} Let $\L$ be a core fuzzy logic, $T$ a theory and $\phi$ a formula. Then the following are equivalent:
\begin{enumroman}
\item $T\vdash \phi,$
\item $(\M,\bold L)\models \phi$ for every model $(\M,\bold L)$ of $T,$
\item $(\M,\bold L)\models \phi$ for every exhaustive model $(\M, \bold L)$ of $T.$
\end{enumroman}
\end{theorem}
In this paper we prove an omitting types theorem, but only for countable languages. 
The condition of countability of the language considered is sensible, because
it is known that the omitting types theorem fails for uncountable languages for first order logic.
Our proof is topological so we formulate and prove two known topological theorems. 
We assume familiarty with basic topological concepts such as basis, regular, compact ..etc.

\begin{definition} An $MV$ algebra is an algebra
$$\A=(A, \oplus, \odot, \neg, 0,1)$$
where $\oplus$, $\odot$ are binary operations, $\neg$ is a unary operation and $0,1\in A$, such that the following identities hold: 
\begin{enumerate}
\item $a\oplus b=b\oplus a,\ \ \  a\odot b=b\odot a.$
\item $a\oplus (b \oplus c)=(a\oplus b)\oplus c$,\  \  \ $a\odot (b \odot c)=(a\odot b)\odot c.$

\item $a\oplus 0=a$ ,\ \ \ $a\odot 1=a.$
\item $a\oplus 1=1$, \ \ \ $a\odot 0=a.$
\item  $a\oplus \neg a=1$,\ \ \  $a\odot \neg a=0.$
\item $\neg (a\oplus b)=\neg a\odot \neg b,$\ \ \ $\neg (a\odot b)=\neg a\oplus \neg b.$
\item $a=\neg \neg a$\ \ \ $\neg 0=1.$
\item $\neg(\neg a\oplus b)\oplus b=\neg(\neg b\oplus a)\oplus a.$
\end{enumerate}
\end{definition}
$MV$ algebras form a variety that is a subvariety of the variety of  $BL$ algebras intoduced by Hajek, 
in fact $MV$ algebras coincide with those $BL$ algebras satisfying double negation law, 
namely that $\neg\neg x=x$, and contains all  Boolean algebras.
\begin{example} A simple numerical example is $A=[0,1]$ with operations $x\oplus y=min(x+y, 1)$, $x\odot y=max(x+y-1, 0)$,  and $\neg x=1-x$. 
In mathematical fuzzy logic, this $MV$-algebra is called the standard $MV$ algebra, 
as it forms the standard real-valued semantics of Lukasiewicz logic. $MV$ algebras  can be obtained from Boolean algebras by dropping idempotency.
\end{example}
$MV$ algebras aso arise from the study of continous $t$ norms. 
\begin{definition}A $t$ norm is a binary operation $*$ on $[0,1]$, i.e $(t:[0,1]^2\to [0,1]$) such that
\begin{enumroman}
\item  $*$ is commutative and associative,
that is for all $x,y,z\in [0,1]$,
$$x*y=y*x,$$
$$(x*y)*z=x*(y*z).$$
\item $*$ is non decreasing in both arguments, that is
$$x_1\leq x_2\implies x_1*y\leq x_2*y,$$
$$y_1\leq y_2\implies x*y_1\leq x*y_2.$$
\item $1*x=x$ and $0*x=0$ for all $x\in [0,1].$

\end{enumroman}
\end{definition}
The following are the most important (known) examples of continuous $t$ norms.

\begin{enumroman}
\item Lukasiewicz $t$ norm: $x*y=max(0,x+y-1),$
\item Godel $t$ norm $x*y=min(x,y),$
\item Product $t$ norm $x*y=x.y$.
\end{enumroman}
We have the following known result \cite{H} lemma 2.1.6
 
\begin{theorem} Let $*$ be a continuous $t$ norm. 
Then there is a unique binary operation $x\to y$ satisfying for all $x,y,z\in [0,1]$, the condition $(x*z)\leq y$ iff $z\leq (x\to y)$, namely 
$x\to y=max\{z: x*z\leq y\}.$
\end{theorem}
The operation $x\to y$ is called the residuam of the $t$ norm. The residuam $\to$ 
defines its corresponding unary operation of precomplement 
$\neg x=(x\to 0)$.
Abstracting away from $t$ norms, we get

\begin{definition} A residuated lattice is an algebra
$$(L,\cup,\cap, *, \to 0,1)$$ 
with four binary operations and two constants such that
\begin{enumroman}
\item $(L,\cup,\cap, 0,1)$ is a lattice with largest element $1$ and the least element $0$ (with respect to the lattice ordering defined the usual way: 
$a\leq b$ iff $a\cap b=a$).
\item $(L,*,1)$ is a commutative semigroup with largest element $1$, that is $*$ is commutative, associative, $1*x=x$ for all $x$.
\item Letting $\leq$ denote the usual lattice ordering, we have $*$ and $\to $ form an adjoint pair, i.e for all $x,y,z$
$$z\leq (x\to y)\Longleftrightarrow x*z\leq y.$$
\end{enumroman}
\end{definition}
A result of Hajek, is that an $MV$ algebra is  a prelinear commutative bounded integral residuated lattice 
satisfying the additional identity $x\cup y=(x\to y)\to y.$ In case of an $MV$ algebra, $*$ is the so-called strong conjunction which we 
denote here following standard notation in the 
literature by $\odot$. $\cap$ is called weak conjunction. The other operations are defined by 
$\neg a=a\to 0$ and $a\oplus b=\neg(\neg a\odot \neg b).$ The operation $\cup$ is called weak disjunction, while $\oplus$ is called 
strong disjunction. The presence of weak and strong conjunction is a common feature of substructural logics without the rule of contraction, to which Lukasiewicz 
logic belongs.

We now turn to describing some metalogical notions, culminating in formulating our main results in logical form. However, throughout the paper, 
our investigations will be purely algebraic,
using the well develped machinery of algebraic logic. 
There are two kinds of semantics for systems of many-valued logic.
Standard logical matrices and algebraic semantics. We shall only encounter  algebraic semantics. 
From a philosophical, especially epistemological point of view the semantic aspect of logic is more basic than the syntactic one,  
because it is mainly the semantic core which determines the choice of suitable 
syntactic versions of the corresponding system of logic.

\section{Free algebras in $BL$ algebras with operators}

In this section we stdy atomicity of $BL$ algebras with extra operations. This notion is important
in cylindric algebras, and lack of atomicity has been linked to Godels incompleteness theorem.
\begin{definition}
Let $K$ be variety  of $BAO$'s. Let $\L$ be the corresponding multimodal logic.
We say that $\L$ has the {\it Godel's incompleteness property} if there exists
a formula $\phi$ that cannot be extended to a recursive complete theory.
Such formula is called incompletable.
\end{definition}
Let $\L$ be a general modal logic, and let $\Fm_{\equiv}$ be the Tarski-Lindenbaum formula algebra on
finitely many generators.
\begin{theorem}(Essentially Nemeti's) If $\L$ has $G.I$, then the algebra $\Fm_{\equiv}$
is not atomic.
\end{theorem}
\begin{proof}
Assume that $\L$ has $G.I$. Let $\phi$ be an incompletable
formula. We show that there is no atom in the Boolean algebra $\Fm$
below $\phi/\equiv.$
Note that because $\phi$ is consistent, it follows that $\phi/\equiv$ is non-zero.
Now, assume to
the contrary that there is such an atom $\tau/\equiv$ for some
formula $\tau.$
This means that .
that $(\tau\land \bar{\phi})/\equiv=\tau/\equiv$.
Then it follows that
$\vdash (\tau\land \phi)\implies \phi$, i.e.
$\vdash\tau\implies \phi$.
Let
$T=\{\tau,\phi\}$
and let
$Consq(T)=\{\psi\in Fm: T\vdash \psi\}.$
$Consq(T)$ is short for the consequences of $T$.
We show that $T$ is complete and that $Consq(T)$ is
decidable.   Let $\psi$ be an arbitrary formula in $\Fm.$
Then either $\tau/\equiv\leq \psi/\equiv$ or $\tau/\equiv\leq \neg \psi/\equiv$
because $\tau/\equiv$ is an
atom. Thus $T\vdash\psi$ or $T\vdash \neg \psi.$
Here it is the {\it exclusive or} i.e. the two cases cannot occur together.
Clearly $ConsqT$ is recursively enumerable.  By completeness of $T$ we have
$\Fm_{\equiv}\smallsetminus Consq(T)=\{\neg \psi: \psi\in Consq(T)\},$
hence the complement of $ConsqT$ is recursively enumerable as well, hence $T$
is decidable.  Here we are using the trivial fact that $\Fm$ is decidable.
This contradiction proves that $\Fm_{\equiv}$ is not atomic.
\end{proof}
%\begin{definition} An element $a\in A$ is closed, if $f_i(a)=a$ for every $i\in I$
%\end{definition}
%\begin{definition} An algebra $\A$ is hereditary atomic, if every subalgebra is atomic.
%\end{definition}
In the following theorem, we give a unified perspective 
on several classes of algebras, studied in algebraic logic. Such algebras are cousins of cylindric algebras; though
the differences, in many cases, can be subtle and big.

(1) holds for diagonal free cylindric algebras, cylindric algebras, Pinter's substitution algebras
(which are replacement algebras endowed with cylindrifiers) and quasipolydic algebras
with and without equality when the dimension is $\leq 2$. (2) holds for Boolean algebras; we do not know whether it extends any further.  
(3) holds for such algebras for all finite dimensions. 
%(4) is due to Jonsson and Tarski.

\begin{theorem} \label{free}Let $K$ be a variety of $BL$ algebras with finitely many operators.
\begin{enumarab}
\item Assume that  $K=V(Fin(K))$, and for any $\B\in K$ and $b'\in \B$, there exists a regular $b\in \B$ such that
$\Ig^{\B}\{b'\}=\Ig^{\Bl\B}\{b\}$. If $\A$ is finitely generated, then $\A$ is atomic, hence 
the finitely generated free algebras are atomic. In particular, if $K$ is a discriminator variety, with discriminator term 
$d$, then finitely generated algebras are 
atomic. (One takes $b'=d(b)$).

\item Assume that $V$ is a $BAO$ and that the condition above on principal ideals, together with the condition that
that if $b_1'$ and $b_2$'s are the generators of two given ideals happen to be a partition (of the unit), 
then $b_0, b_1$ can be chosen to be also a partition. Then
$\Fr_{\beta}K_{\alpha}\times \Fr_{\beta}K_{\alpha}\cong \Fr_{|\beta+1|}K.$ In particular if $\beta$ is infinite, and
$\A=\Fr_{\beta}K$, then $\A\times \A\cong \A$. 
\item Assume that $\beta<\omega$, and assume the above condition on principal ideals.
Suppose further that for every $k\in \omega$, there exists an algebra $\A\in K$, with at least $k$ atoms, 
that is generated by a single element. Then $\Fr_{\beta}K$ has infinitely many atoms.
\item  Assume that $K=V(Fin(K))$.
Suppose $\A$ is $K$ freely generated by a finite set $X$ and $\A=\Sg Y$ with $|Y|=|X|$. Then $\A$ is $K$ freely generated
by $Y.$
\end{enumarab}
\end{theorem}
\begin{proof}
\begin{enumarab}
\item Assume that $a\in A$ is non-zero. Let $h:\A\to \B$ be a homomorphism of $\A$ into a finite algebra $\B$ such that
$h(a)\neq 0$. Let $I=ker h.$ We claim that $I$ is a finitely generated ideal.
Let $R_I$ be the congruence relation corresponding to $I$, that is $R_I=\{(a,b)\in A\times A: h(a)=h(b)\}$.

Let $X$ be a finite set such that $X$ generates $\A$ and $h(X)=\B$. Such a set obviously exists.
Let $X'=X\cup \{x+y: x, y\in X\}\cup \{-x: x\in X\}\cup \bigcup_{f\in t}\{f(x): x\in X\}.$
Let $R=\Sg^{\A}(R_I\cap X\times X')$. Clearly $R$ is a finitely generated congruence and $R_I\subseteq R$.
We show that the converse inclusion also holds.

For this purpose we first show that $R(X)=\{a\in A: \exists x\in X (x,a)\in R\}=\A.$
Assume that $xRa$ and $yRb$, $x,y\in X$ then $x+yRa+b$, but there exists $z\in X$ such that $h(z)=h(x+y)$ and $zR(x+y)$, hence
$zR(a+b)$ , so that $a+b\in R(X)$. Similarly for all other operations. Thus $R(X)=A$.
Now assume that $a,b\in A$ such that $h(a)=h(b)$.
Then there exist $x, y\in X$ such that $xRa$ and $xRb$. Since $R\subseteq ker h$,
we have $h(x)=h(a)=h(b)=h(y)$ and so $xRy$, hence $aRb$ and $R_I\subseteq R$.
So $I=\Ig\{b'\}$ for some element $b'$.  Then there exists $b\in \A$ such that  $\Ig^{\Bl\B}\{b\}=\Ig\{b'\}.$ Since $h(b)=0$ and $h(a)\neq 0,$
we have $a.-b\neq 0$. If $a.-b=0$, then $h(a).-h(b)=0$

Now $h(\A)\cong \A/\Ig^{\Bl\B}\{b\}$ as $K$ algebras. Let $\Rl_{-b}\A=\{x: x\leq -b\}$. Let $f:\A/\Ig^{\Bl\B}\{b\}\to \Rl_{-b}\A$ be defined by
$\bar{x}\mapsto x.-b$. Then $f$ is an isomorphism of Boolean algebras (recall that the operations of $\Rl_{-b}\B$ are defined by
relativizing the Boolean operations to $-b$.)
Indeed, the map is well defined, by noting that if $x\delta y\in \Ig^{\Bl\B}\{b\}$, where $\delta$ denotes symmetric difference,
then $x.-b=y.-b$ because $x, y\leq b$.

Since $\Rl_{-b}\A$ is finite, and $a.-b\in \Rl_{-b}\A$ is non-zero, then there exists an atom $x\in \Rl_{-b}\A$ below $a$,
but clearly $\At(\Rl_{-b}\A)\subseteq \At\A$ and we are done.
%(if $d$ is the discriminator term for $V$, then one takes $b'=d(b)$.)

\item Let $(g_i:i\in \beta+1)$ be the free generators of $\A=\Fr_{\beta+1}K$.
We first show that $\Rl_{g_{\beta}}\A$ is freely generated by
$\{g_i.g_{\beta}:i<\beta\}$. Let $\B$ be in $K$ and $y\in {}^{\beta}\B$.
Then there exists a homomorphism $f:\A\to \B$ such that $f(g_i)=y_i$ for all $i<\beta$ and $f(g_{\beta})=1$.
Then $f\upharpoonright \Rl_{g_{\beta}}\A$ is a homomorphism such that $f(g_i.g_{\beta})=y_i$. Similarly
$\Rl_{-g_{\beta}}\A$ is freely generated by $\{g_i.-g_{\beta}:i<\beta\}$.
Let $\B_0=\Rl_{g_{\beta}}\A$ and $\B_1=\Rl_{g_{\beta}}\A$.
Let $t_0=g_{\beta}$ and $t_1=-g_{\beta}$. Let $x_i$ be such that $J_i=\Ig\{t_i\}=\Ig^{Bl\A}\{x_i\}$, and $x_0.x_1=0$.
Exist by assumption. Assume that $z\in J_0\cap J_1$. Then $z\leq x_i$,
for $i=0, 1$, and so  $z=0$. Thus $J_0\cap J_1=\{0\}$. Let $y\in A\times A$, and let $z=(y_0.x_0+y_1.x_1)$, then $y_i.x_i=z.x_i$ for each $i=\{0,1\}$
and so $z\in \bigcap y_0/J_0\cap y_1/J_1$. Thus $\A/J_i\cong \B_i$, and so
$\A\cong \B_0\times \B_1$.

%\end{demo}
\item Let $\A=\Fr_{\beta}K.$ Let $\B$ have $k$ atoms and generated by a single element. Then there exists a surjective
homomorphism $h:\A\to \B$. Then, as in the first item,  $\A/\Ig^{\Bl\B}\{b\}\cong \B$, and so $\Rl_{b}\B$ has $k$ atoms.
Hence $\A$ has $k$ atoms for any $k$ and we are done.

\item Let $\A=\Fr_XK$, let $\B\in Fin(K)$ and let $f:X\to \B$. Then $f$ can extended to a homomorphism $f':\A\to \B$.
Let $\bar{f}=f'\upharpoonright Y$. If $f, g\in {}^XB$ and $\bar{f}=\bar{g}$,
then $f'$ and $g'$ agree on a generating set $Y$, so $f'=g',$ hence $f=g$.
Therefore we obtain a one to one mapping from $^XB$ to $^YB$, but $|X|=|Y|,$
hence this map is surjective. In other words for each $h\in {}^YB,$ there exists a unique
$f\in {}^XB$ such that $\bar{f}=h$, then $f'$ with domain $\A$ extends $h.$
Since $\Fr_XK=\Fr_X(Fin(K))$ we are done.
\end{enumarab}
\end{proof}
For cylindric algebras, diagonal free cylindric algebras Pinter's algebras and quasipolyadic equality, though free algebras of $>2$ dimensions
contain infinitely many atoms, they are not atomic. 
(The diagonal free case of cylindric algebras is a very recent result, due to Andr\'eka and N\'emeti, that has profound
repercussions on the foundation of mathematics.)
We, next, state two theorems that hold for such algebras, in the general context of $BAO$'s. But first a definition.

\begin{definition} Let $K$ be a class of $BAO$ with operators $(f_i: i\in I)$
Let $\A\in K$. An element $b\in A$ is called {\it hereditary closed} if for all $x\leq b$, $f_i(x)=x$.
\end{definition}
In the presence of diagonal elements $d_{ij}$ and cylindrifications $c_i$ for indices $<2$, $-c_0-d_{01},$ 
is hereditory closed. 
\begin{theorem}
\begin{enumarab}
\item Let $\A=\Sg X$ and $|X|<\omega$. Let $b\in \A$ be hereditary closed. Then $\At\A\cap \Rl_{b}\A\leq 2^{n}$.
If $\A$ is freely generated by $X$, then $\At\A\cap \Rl_{b}\A= 2^{n}.$
\item If every atom of $\A$ is below $b,$ then $\A\cong \Rl_{b}\A\times \Rl_{-b}\A$, and $|\Rl_{b}\A|=2^{2^n}$.
If in addition $\A$ is infinite, then $\Rl_{-b}\A$ is atomless.
\end{enumarab}
\end{theorem}
\begin{proof} Assume that $|X|=m$. We have $|\At\A\cap \Rl_b\A|=|\{\prod Y \sim \sum(X\sim Y).b\}\sim \{0\}|\leq {}^{m}2.$
Let $\B=\Rl_b\A$. Then $\B=\Sg^{\B}\{x_i.b: i<m\}=\Sg^{Bl\B}\{x_i.b:i<\beta\}$ since $b$ is hereditary fixed.
For $\Gamma\subseteq m$, let
$$x_{\Gamma}=\prod_{i\in \Gamma}(x_i.b).\prod_{i\in m\sim \Gamma}(x_i.-b).$$
Let $\C$ be the two element algebra. Then for each $\Gamma\subseteq m$, there is a homomorphism $f:\A\to \C$ such that
$fx_i=1$ iff $i\in \Gamma$.This shows that $x_{\Gamma}\neq 0$ for every $\Gamma\subseteq m$,
while it is easily seen that $x_{\Gamma}$ and $x_{\Delta}$
are distinct for distinct $\Gamma, \Delta\subseteq m$. We show that $\A\cong \Rl_{b}\A\times \Rl_{-b}\A$.

Let $\B_0=\Rl_{b}\A$ and $\B_1=\Rl_{-b}\A$.
Let $t_0=b$ and $t_1=-b$. Let  $J_i=\Ig\{t_i\}$
Assume that $z\in J_0\cap J_1$. Then $z\leq t_i$,
for $i=0, 1$, and so  $z=0$. Thus $J_0\cap J_1=\{0\}$. Let $y\in A\times A$,
and let $z=(y_0.t_0+y_1.t_1)$, then $y_i.x_i=z.x_i$ for each $i=\{0,1\}$
and so $z\in \bigcap y_0/J_0\cap y_1/J_1$. Thus $\A/J_i\cong \B_i$, and so
$\A\cong \B_0\times \B_1$.

\end{proof}
The above theorem holds for free cylindric and quasi-polyadic equality algebras. The second part  (all atoms are zero-dimensional) 
is proved by Mad\'arasz and N\'emeti.

The following theorem holds for any class of $BAO$'s.
\begin{theorem}\label{atomless}
The free algebra on an infinite generating set is atomless.
\end{theorem}
\begin{proof} Let $X$ be the infinite freely generating set. Let $a\in A$ be non-zero. 
Then there is a finite set $Y\subseteq X$ such that $a\in \Sg^{\A} Y$. Let $y\in X\sim Y$.
Then by freeness, there exist homomorphisms $f:\A\to \B$ and $h:\A\to \B$ such that $f(\mu)=h(\mu) $ for all $\mu\in Y$ while
$f(y)=1$ and $h(y)=0$. Then $f(a)=h(a)=a$. Hence $f(a.y)=h(a.-y)=a\neq 0$ and so $a.y\neq 0$ and
$a.-y\neq 0$.
Thus $a$ cannot be an atom.
\end{proof}

\section{Amalgamation in $MV$ algebras}
%As in the lst section, we start by proving some general theorems concerning classes of $BAO$'s
Theorems \ref{supapgeneral}, \ref{CP}, \ref{weak} to come, 
give a flavour of the interconnections between the local properties of $CP$ and $SIP$ (on free algebras) and the global property of 
superamalgamation (of the entire class) . Maksimova and Mad\'arasz proved that if interpolation holds in free algebras of a variety, then the variety has the
superamalgamation property.
Using a similar argument, we prove this implication in a slightly more
general setting. But first an easy  lemma:

\begin{lemma} Let $K$ be a class of $BAO$'s. Let $\A, \B\in K$ with $\B\subseteq \A$. Let $M$ be an ideal of $\B$. We then have:
\begin{enumarab}
\item $\Ig^{\A}M=\{x\in A: x\leq b \text { for some $b\in M$}\}$
\item $M=\Ig^{\A}M\cap \B$
\item if $\C\subseteq \A$ and $N$ is an ideal of $\C$, then
$\Ig^{\A}(M\cup N)=\{x\in A: x\leq b\oplus c\ \text { for some $b\in M$ and $c\in N$}\}$
\item For every ideal $N$ of $\A$ such that $N\cap B\subseteq M$, there is an ideal $N'$ in $\A$
such that $N\subseteq N'$ and $N'\cap B=M$. Furthermore, if $M$ is a maximal ideal of $\B$, then $N'$ can be taken to be a maximal ideal of $\A$.

\end{enumarab}
\end{lemma}
\begin{demo}{Proof} Only (iv) deserves attention. The special case when $n=\{0\}$ is straightforward.
The general case follows from this one, by considering
$\A/N$, $\B/(N\cap \B)$ and $M/(N\cap \B)$, in place of $\A$, $\B$ and $M$ respectively.
\end{demo}
The previous lemma will be frequently used without being explicitly mentioned.

\begin{theorem}\label{supapgeneral} Let $K$ be a class of $BAO$'s such that $\mathbf{H}K=\mathbf{S}K=K$.
Assume that for all $\A, \B, \C\in K$, inclusions
$m:\C\to \A$, $n:\C\to \B$, there exist $\D$ with $SIP$ and $h:\D\to \C$, $h_1:\D\to \A$, $h_2:\D\to \B$
such that for $x\in h^{-1}(\C)$,
$$h_1(x)=m\circ h(x)=n\circ h(x)=h_2(x).$$
Then $K$ has $SUPAP$.
\bigskip
\bigskip
\bigskip
\bigskip
\bigskip
\bigskip
\bigskip
\bigskip
\bigskip
\bigskip
\bigskip
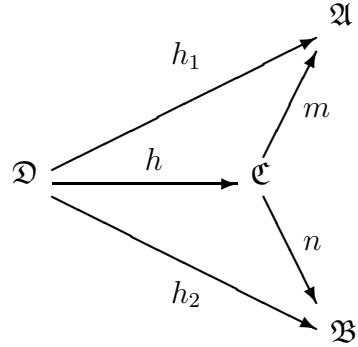
\begin{picture}(10,0)(-30,70)
\thicklines
\put (-10,0){$\D$}
\put(5,0){\vector(1,0){70}}\put(80,0){$\C$}
\put(5,5){\vector(2,1){100}}\put(110,60){$\A$}
\put(5,-5){\vector(2,-1){100}}\put(110,-60){$\B$}
\put(85,10){\vector(1,2){20}}
\put(85,-5){\vector(1,-2){20}}
\put(40,5){$h$}
\put(100,25){$m$}
\put(100,-25){$n$}
\put(50,45){$h_1$}
\put(50,-45){$h_2$}
\end{picture}
\end{theorem}
\bigskip
\bigskip
\begin{proof} Let $\D_1=h_1^{-1}(\A)$ and $\D_2=h_2^{-1}(\B)$. Then $h_1:\D_1\to \A$, and $h_2:\D_2\to \B$.

Let $M=ker h_1$ and $N=ker h_2$, and let
$\bar{h_1}:\D_1/M\to \A, \bar{h_2}:\D_2/N\to \B$ be the induced isomorphisms.

Let $l_1:h^{-1}(\C)/h^{-1}(\C)\cap M\to \C$ be defined via $\bar{x}\to h(x)$, and
$l_2:h^{-1}(\C)/h^{-1}(\C)\cap N$ to $\C$ be defined via $\bar{x}\to h(x)$.
Then those are well defined, and hence
$k^{-1}(\C)\cap M=h^{-1}(\C)\cap N$.
Then we show that $\P=\Ig(M\cup N)$ is a proper ideal and $\D/\P$ is the desired algebra.
Now let $x\in \mathfrak{Ig}(M\cup N)\cap \D_1$.
Then there exist $b\in M$ and $c\in N$ such that $x\leq b\oplus c$. Thus $x-b\leq c$.
But $x-b\in \D_1$ and $c\in \D_2$, it follows that there exists an interpolant
$d\in \D_1\cap \D_2$  such that $x-b\leq d\leq c$. We have $d\in N$
therefore $d\in M$, and since $x\leq d\oplus b$, therefore $x\in M$.
It follows that
$\mathfrak{Ig}(M\cup N)\cap \D_1=M$
and similarly
$\mathfrak{Ig}(M\cup N)\cap \D_2=N$.
In particular $P=\mathfrak{Ig}(M\cup N)$ is a proper ideal.

Let $k:\D_1/M\to \D/P$ be defined by $k(a/M)=a/P$
and $h:\D_2/N\to \D/P$ by $h(a/N)=a/P$. Then
$k\circ m$ and $h\circ n$ are one to one and
$k\circ m \circ f=h\circ n\circ g$.
We now prove that $\D/P$ is actually a
superamalgam. i.e we prove that $K$ has the superamalgamation
property. Assume that $k\circ m(a)\leq h\circ n(b)$. There exists
$x\in \D_1$ such that $x/P=k(m(a))$ and $m(a)=x/M$. Also there
exists $z\in \D_2$ such that $z/P=h(n(b))$ and $n(b)=z/N$. Now
$x/P\leq z/P$ hence $x-z\in P$. Therefore  there is an $r\in M$ and
an $s\in N$ such that $x-r\leq z\oplus a$. Now $x-r\in \D_1$ and $z\oplus a\in\D_2,$ it follows that there is an interpolant
$u\in \D_1\cap \D_2$ such that $x-r\leq u\leq z\oplus $. Let $t\in \C$ such that $m\circ
f(t)=u/M$ and $n\circ g(t)=u/N.$ We have  $x/P\leq u/P\leq z/P$. Now
$m(f(t))=u/M\geq x/M=m(a).$ Thus $f(t)\geq a$. Similarly
$n(g(t))=u/N\leq z/N=n(b)$, hence $g(t)\leq b$. By total symmetry,
we are done.
\end{proof}

The intimate relationship between $CP$ on free algebras generating a certain variety and the $AP$ for such varieties, 
has been worked out extensively by Pigozzi
for various classes of cylindric algebras. Here we prove an implication in one direction for $BAO$'s.
Notice that we do not assume that our class is a variety.
\begin{lemma} Let $L\supseteq L_{BA}$ be a functional signature, and $V$ a variety of $L-BAO$'s. Let $d(x)$ be a unary $L$ term. 
Then the following are equivalent:
\begin{enumarab}
\item $d$ is a discriminator term of $SirV$, so that $V$ is a discriminator variety.
\item all equations of the following for are valid in $V$:
\begin{enumerate}
\item $x\leq d(x)$

\item $d(d(x))\leq d(x)$
%\item $d(-d(x)\leq -d(x)$
\item $f(x)\leq d(x)$ for all $f\in L\sim L_{BA}$
\end{enumerate}
\end{enumarab}
\end{lemma}

\begin{theorem}\label{CP}
Let $K$ be such that $\mathbf{H}K=\mathbf{S}K=K$. If $K$ has the amalgamation  property, then the $V(K)$
free algebras, on any set of generators, have $CP$.
\end{theorem}

\begin{proof}
For $R\in Co\A$ and $X\subseteq  A$, by $(\A/R)^{(X)}$ we understand the subalgebra of
$\A/R$ generated by $\{x/R: x\in X\}.$ Let $\A$, $X_1$, $X_2$, $R$ and $S$ be as specified in in the definition of $CP$.
Define $$\theta: \Sg^{\A}(X_1\cap X_2)\to \Sg^{\A}(X_1)/R$$
by $$a\mapsto a/R.$$
Then $ker\theta=R\cap {}^2\Sg^{\A}(X_1\cap X_2)$ and $Im\theta=(\Sg^{\A}(X_1)/R)^{(X_1\cap X_2)}$.
It follows that $$\bar{\theta}:\Sg^{\A}(X_1\cap X_2)/R\cap {}^2\Sg^{\A}(X_1\cap X_2)\to (\Sg^{\A}(X_1)/R)^{(X_1\cap X_2)}$$
defined by
$$a/R\cap {}^{2}\Sg^{\A}(X_1\cap X_2)\mapsto a/R$$
is a well defined isomorphism.
Similarly
$$\bar{\psi}:\Sg^{\A}(X_1\cap X_2)/S\cap {}^2\Sg^{\A}(X_1\cap X_2)\to (\Sg^{\A}(X_2)/S)^{(X_1\cap X_2)}$$
defined by
$$a/S\cap {}^{2}\Sg^{\A}(X_1\cap X_2)\mapsto a/S$$
is also a well defined isomorphism.
But $$R\cap {}^2\Sg^{\A}(X_1\cap X_2)=S\cap {}^2\Sg^{\A}(X_1\cap X_2),$$
Hence
$$\phi: (\Sg^{\A}(X_1)/R)^{(X_1\cap X_2)}\to (\Sg^{\A}(X_2)/S)^{(X_1\cap X_2)}$$
defined by
$$a/R\mapsto a/S$$
is a well defined isomorphism.
Now
$(\Sg^{\A}(X_1)/R)^{(X_1\cap X_2)}$ embeds into $\Sg^{\A}(X_1)/R$ via the inclusion map; it also embeds in $\A^{(X_2)}/S$ via $i\circ \phi$ where $i$
is also the inclusion map.
For brevity let $\A_0=(\Sg^{\A}(X_1)/R)^{(X_1\cap X_2)}$, $\A_1=\Sg^{\A}(X_1)/R$ and $\A_2=\Sg^{\A}(X_2)/S$ and $j=i\circ \phi$.
Then $\A_0$ embeds in $\A_1$ and $\A_2$ via $i$ and $j$ respectively.
Then there exists $\B\in V$ and monomorphisms $f$ and $g$ from $\A_1$ and $\A_2$ respectively to
$\B$ such that
$f\circ i=g\circ j$.
Let $$\bar{f}:\Sg^{\A}(X_1)\to \B$$ be defined by $$a\mapsto f(a/R)$$ and $$\bar{g}:\Sg^{\A}(X_2)\to \B$$
be defined by $$a\mapsto g(a/R).$$
Let $\B'$ be the algebra generated by $Imf\cup Im g$.
Then $\bar{f}\cup \bar{g}\upharpoonright X_1\cup X_2\to \B'$ is a function since $\bar{f}$ and $\bar{g}$ coincide on $X_1\cap X_2$.
By freeness of $\A$, there exists $h:\A\to \B'$ such that $h\upharpoonright_{X_1\cup X_2}=\bar{f}\cup \bar{g}$.
Let $T=kerh $. Then it is not hard to check that
$$T\cap {}^2 \Sg^{\A}(X_1)=R \text { and } T\cap {}^2\Sg^{\A}(X_2)=S.$$
\end{proof}
Finally we show that $CP$ implies a weak form of interpolation.
\begin{theorem}\label{weak}
If an algebra $\A$ has $CP,$  then for $X_1, X_2\subseteq \A$, if $x\in \Sg^{\A}X_1$ and $z\in \Sg^{\A}X_2$ are such that
$x\leq z$, then there exists $y\in \Sg^{\A}(X_1\cap X_2),$ $n\in \omega$ and a term $\tau$ such that $x\leq y\leq \tau(z^n)$.
If $Ig^{\Bl\A}\{z\}=\Ig^{\A}\{z\},$ then $\tau$ can be chosen to be the
identity term. In particular, if $z$ is closed, 
or $\A$ comes from a discriminator variety, then the latter case occurs.
\end{theorem}

\begin{proof}
Now let $x\in \Sg^{\A}(X_1)$, $z\in \Sg^{\A}(X_2)$ and assume that $x\leq z$.
Then $$x\in (\Ig^{\A}\{z\})\cap \Sg^{\A}(X_1).$$
Let $$M=\Ig^{\A^{(X_1)}}\{z\}\text { and } N=\Ig^{\Sg^{\A}(X_2)}(M\cap \Sg^{\A}(X_1\cap X_2)).$$
Then $$M\cap \Sg^{\A}(X_1\cap X_2)=N\cap \Sg^{\A}(X_1\cap X_2).$$
By identifying ideals with congruences, and using the congruence extension property,
there is a an ideal $P$ of $\A$
such that $$P\cap \Sg^{\A}(X_1)=N\text { and }P\cap \Sg^{\A}(X_2)=M.$$
It follows that
$$\Ig^{\A}(N\cup M)\cap \Sg^{\A}(X_1)\subseteq P\cap \Sg^{\A}(X_1)=N.$$
Hence
$$(\Ig^{(\A)}\{z\})\cap A^{(X_1)}\subseteq N.$$
and we have
$$x\in \Ig^{\Sg^{\A}X_1}[\Ig^{\Sg^{\A}(X_2)}\{z\}\cap \Sg^{\A}(X_1\cap X_2).]$$
This implies that there is an element $y$ such that
$$x\leq y\in \Sg^{\A}(X_1\cap X_2)$$
and $y\in \Ig^{Sg^{\A}X}\{z\}$, hence the first required. The second required follows
follows, also immediately, since $y\leq z$, because $\Ig^{\A}\{z\}=\Rl_z\A$.
\end{proof}
t of forming dilations and neat reducts (which are, in fact, dual operations.)

\section{An application to Heyting algebras.}

For an algebra $\A$, $End(\A)$ denotes the set of endomorphisms of $\A$ (homomorphisms of $\A$ into itself), 
which is a semigroup under the operation $\circ$ of composition of maps.

\begin{definition} A transformation system is a quadruple $(\A, I, G, {\sf S})$ where 
$\A$ is an algebra, $I$ is a set, $G$ is a subsemigroup of $(^II,\circ)$ 
and ${\sf S}$ is a homomorphism from
$G$ into $End(\A).$
\end{definition}
Throughout the paper, $\A$ will always be a Heyting algebra.
If we want to study predicate intuitionistic logic, then we are naturally led to expansions of Heyting algebras allowing quantification.
But we do not have negation in the classical sense, so we have to deal with existential and universal quantifiers each separately.

\begin{definition} Let $\A=(A, \lor, \land,\rightarrow,0)$ be a Heyting algebra. An existential quantifier $\exists$ on $A$ is a mapping
$\exists:\A\to \A$ such that the following hold for all $p,q\in A$: 
\begin{enumarab}
\item $\exists(0)=0,$
\item $p\leq \exists p,$
\item $\exists(p\land \exists q)=\exists p\land \exists q,$
\item $\exists(\exists p\rightarrow \exists q)=\exists p\rightarrow \exists q,$
\item $\exists(\exists p\lor \exists q)=\exists p\lor \exists q,$
\item $\exists\exists p=\exists p.$
\end{enumarab}
\end{definition}
\begin{definition}  Let $\A=(A, \lor, \land,\rightarrow,0)$ be a Heyting algebra. A universal quantifier  $\forall$ on $A$ is a mapping
$\forall:\A\to \A$ such that the following hold for all $p,q\in A$: 
\begin{enumarab}
\item $\forall 1=1,$
\item $\forall p\leq p,$
\item $\forall(p\rightarrow q)\leq \forall p\rightarrow \forall q,$
\item $\forall \forall p=\forall p.$
\end{enumarab}
\end{definition}

Now we define our algebras. Their similarity type depends on a fixed in advance semigroup. 
We write $X\subseteq_{\omega} Y$ to denote that $X$ is a finite subset 
of $Y$.
\begin{definition} Let $\alpha$ be an infinite set. Let $G\subseteq {}^{\alpha}\alpha$ be a semigroup under the operation of composition of maps. 
An $\alpha$ dimensional polyadic Heyting $G$ algebra, a $GPHA_{\alpha}$ for short, is an algebra of the following
form
$$(A,\lor,\land,\rightarrow, 0, {\sf s}_{\tau}, {\sf c}_{(J)}, {\sf q}_{(J)})_{\tau\in G, J\subseteq_{\omega} \alpha}$$
where $(A,\lor,\land, \rightarrow, 0)$ is a Heyting algebra, ${\sf s}_{\tau}:\A\to \A$ is an endomorphism of Heyting algebras,
${\sf c}_{(J)}$ is an existential quantifier, ${\sf q}_{(J)}$ is a universal quantifier, such that the following hold for all 
$p\in A$, $\sigma, \tau\in [G]$ and $J,J'\subseteq_{\omega} \alpha:$
\begin{enumarab}
\item ${\sf s}_{Id}p=p.$
\item ${\sf s}_{\sigma\circ \tau}p={\sf s}_{\sigma}{\sf s}_{\tau}p$ (so that ${\sf s}:\tau\mapsto {\sf s}_{\tau}$ defines a homomorphism from $G$ to $End(\A)$; 
that is $(A, \lor, \land, \to, 0, G, {\sf s})$ is a transformation system).
\item ${\sf c}_{(J\cup J')}p={\sf c}_{(J)}{\sf c}_{(J')}p , \ \  {\sf q}_{(J\cup J')}p={\sf q}_{(J)}{\sf c}_{(J')}p.$
\item ${\sf c}_{(J)}{\sf q}_{(J)}p={\sf q}_{(J)}p , \ \  {\sf q}_{(J)}{\sf c}_{(J)}p={\sf c}_{(J)}p.$
\item If $\sigma\upharpoonright \alpha\sim J=\tau\upharpoonright \alpha\sim J$, then
${\sf s}_{\sigma}{\sf c}_{(J)}p={\sf s}_{\tau}{\sf c}_{(J)}p$ and ${\sf s}_{\sigma}{\sf q}_{(J)}p={\sf s}_{\tau}{\sf q}_{(J)}p.$
\item If $\sigma\upharpoonright \sigma^{-1}(J)$ is injective, then
${\sf c}_{(J)}{\sf s}_{\sigma}p={\sf s}_{\sigma}{\sf c}_{\sigma^{-1}(J)}p$
and ${\sf q}_{(J)}{\sf s}_{\sigma}p={\sf s}_{\sigma}{\sf q}_{\sigma^{-1}(J)}p.$
\end{enumarab}
\end{definition}

\begin{definition} Let $\alpha$ and $G$ be as in the prevoius definition. 
By a $G$ polyadic equality algebra, a $GPHAE_{\alpha}$ for short, 
we understand an algebra  of the form  
$$(A,\lor,\land,\rightarrow, 0, {\sf s}_{\tau}, {\sf c}_{(J)}, {\sf q}_{(J)},  {\sf d}_{ij})_{\tau\in G, J\subseteq_{\omega} \alpha, i,j\in \alpha}$$
where $(A,\lor,\land,\rightarrow, 0, {\sf s}_{\tau}, {\sf c}_{(J)}, {\sf q}_{(J)})_{\tau\in G\subseteq {}^{\alpha}\alpha, J\subseteq_{\omega} \alpha}$
is a $GPHA_{\alpha}$ and ${\sf d}_{ij}\in A$ for each $i,j\in \alpha,$ such that
the following identities hold for all $k,l\in \alpha$
and all $\tau\in G:$
\begin{enumarab}
\item ${\sf d}_{kk}=1$

\item ${\sf s}_{\tau}{\sf d}_{kl}={\sf d}_{\tau(k), \tau(l)}.$

\item $x\cdot {\sf d}_{kl}\leq {\sf s}_{[k|l]}x$

\end{enumarab}
\end{definition}

Here $[k|l]$ is the replacement that sends $k$ to $l$ and otherwise is the identity.
In our definition of algebras, we depart from \cite{HMT2} by defining polyadic algebras on sets rather than on ordinals. 
In this manner, we follow the tradition of Halmos.
We refer to $\alpha$ as the dimension of $\A$ and we write $\alpha=dim\A$.
Borrowing terminology from cylindric algebras, we refer to ${\sf c}_{(\{i\})}$ by ${\sf c}_i$ and ${\sf q}_{(\{i\})}$ by ${\sf q}_i.$
However, we will have occasion to impose a well order on dimensions thereby dealing with ordinals.

\begin{remark}

When $G$ consists of all finite transformations, 
then any algebra with a Boolean reduct satisfying the above identities relating cylindrifications, diagonal elements and substitutions,
will be a quasipolyadic equality algebra of infinite dimension.
\end{remark}
Next, we collect some properties of $G$ algebras that are more handy to use in our subsequent work.
In what follows, we will be writing $GPHA$ ($GPHAE$)  for all algebras considered.
\begin{theorem}\label{axioms} Let $\alpha$ be an infinite set and $\A\in GPHA_{\alpha}$. 
Then $\A$ satisfies the following identities for $\tau,\sigma\in G$
and all $i,j,k\in \alpha$.
\begin{enumerate}

\item $x\leq {\sf c}_ix={\sf c}_i{\sf c}_ix,\ {\sf c}_i(x\lor y)={\sf c}_ix\lor {\sf c}_iy,\ {\sf c}_i{\sf c}_jx={\sf c}_j{\sf c}_ix$.

That is  ${\sf c}_i$ is an additive operator (a modality)  and ${\sf c}_i,{\sf c}_j$ commute.

\item ${\sf s}_{\tau}$ is a Heyting algebra  endomorphism.

\item  ${\sf s}_{\tau}{\sf s}_{\sigma}x={\sf s}_{\tau\circ \sigma}x$
and ${\sf s}_{Id}x=x$.

\item ${\sf s}_{\tau}{\sf c}_ix={\sf s}_{\tau[i|j]}{\sf c}_ix$.

Recall that $\tau[i|j]$ is the transformation that agrees with $\tau$ on 
$\alpha\smallsetminus\{i\}$ and $\tau[i|j](i)=j$.

\item ${\sf s}_{\tau}{\sf c}_ix={\sf c}_j{\sf s}_{\tau}x$ if $\tau^{-1}(j)=\{i\}$, 
${\sf s}_{\tau}{\sf q}_ix={\sf q}_j{\sf s}_{\tau}x$ 
if $\tau^{-1}(j)=\{i\}$.  

\item  ${\sf c}_i{\sf s}_{[i|j]}x={\sf s}_{[i|j]}x$,\ \ ${\sf q}_i{\sf s}_{[i|j]}x={\sf s}_{[i|j]}x$

\item  ${\sf s}_{[i|j]}{\sf c}_ix={\sf c}_ix$, \ \ ${\sf s}_{[i|j]}{\sf q}_ix={\sf q}_ix$.

\item ${\sf s}_{[i|j]}{\sf c}_kx={\sf c}_k{\sf s}_{[i|j]}x$,\ \ ${\sf s}_{[i|j]}{\sf q}_kx={\sf q}_k{\sf s}_{[i|j]}x$
whenever $k\notin \{i,j\}$.

\item  ${\sf c}_i{\sf s}_{[j|i]}x={\sf c}_j{\sf s}_{[i|j]}x$,\ \  ${\sf q}_i{\sf s}_{[j|i]}x={\sf q}_j{\sf s}_{[i|j]}x$.
\end{enumerate}
\end{theorem}
\begin{demo}{Proof} The proof is tedious but fairly straighforward.
\end{demo}
Obviously the previous equations hold in $GPHAE_{\alpha}$.
Following cylindric algebra tradition and terminology, we will be often writing ${\sf s}_j^i$ for ${\sf s}_{[i|j]}$.

\begin{remark} For $GPHA_{\alpha}$ when $G$ is rich or $G$ consists only of finite transformation it is enough to restrict our attenstion to replacements. Other substitutions are 
definable from those.
\end{remark}

\subsection{Neat reducts and dilations}

Now we recall the important notion of neat reducts, a central concept in cylindric algebra theory, strongly related to representation theorems.
This concept also occurs in polyadic algebras, but unfortunately under a different name, that of compressions.

Forming dilations of an algebra, is basically  an algebraic reflection of a Henkin 
construction; in fact, the dilation of an algebra is another algebra that has an infinite number of new dimensions (constants) 
that potentially eliminate cylindrifications (quantifiers). Forming neat reducts has to do with 
restricting or compressing dimensions (number of variables) rather than increasing them.
(Here the duality has a precise categorical sense which will be formulated in the part 3 of this paper as an adjoint situation).

\begin{definition} 
\begin{enumarab}
\item  Let $ \alpha\subseteq \beta$ be infinite sets. Let $G_{\beta}$ be a semigroup of transformations on $\beta$, 
and let $G_{\alpha}$ be a semigroup of transformations on $\alpha$ such that for all $\tau\in G_{\alpha}$, one has $\bar{\tau}=\tau\cup Id\in G_{\beta}$.
Let $\B=(B, \lor, \land, \to, 0,  {\sf c}_i, {\sf s}_{\tau})_{i\in \beta, \tau\in G_{\beta}}$ be a $G_{\beta}$ algebra.
\begin{enumroman}

\item  We denote by $\Rd_{\alpha}\B$ the $G_{\alpha}$ algebra obtained by dicarding operations in $\beta\sim \alpha$. That is 
$\Rd_{\alpha}\B=(B, \lor,  \land, \to, 0, {\sf c}_i, {\sf s}_{\bar{\tau}})_{i\in \alpha, \tau\in G_{\alpha}}$. Here ${\sf s}_{\bar{\tau}}$ is evaluated 
in $\B$.
\item For $x\in B$,  then $\Delta x,$ the dimension set of $x$, 
is defined by $\Delta x=\{i\in \beta: {\sf c}_ix\neq x\}.$
Let $A=\{x\in B: \Delta x\subseteq \alpha\}$. If $A$ is
a subuniverse of $\Rd_{\alpha}\B$, then $\A$ (the algebra with universe $A$) is a subreduct of $\B$, it is called the {\it neat $\alpha$ 
reduct} of $\B$ and is denoted by $\Nr_{\alpha}\B$.
\end{enumroman}
\item If $\A\subseteq \Nr_{\alpha}\B$, then $\B$ is called a {\it dilation} of $\A$, and we say that $\A$ {\it neatly embeds} in $\B$.
if $A$ generates $\B$ (using all operations of $\B$), then $\B$ is called a {\it minimal dilation} of $\A$.
\end{enumarab}
\end{definition}
The above definition applies equally well to $GPHAE_{\alpha}$. 

\begin{remark}
In certain contexts minimal dilations may not be unique (up to isomorphism), but what we show next is that in all the cases 
we study, they are unique, so for a given algebra $\A$, we may safely say {\it the} minimal dilation of $\A$.
\end{remark}

For an algebra $\A$, and $X\subseteq \A$, $\Sg^{\A}X$ or simply $\Sg X$, when $\A$ is clear from context,
denotes the subalgebra of $\A$ generated by $X.$
The next theorems apply equally well to $GPHAE_{\alpha}$ with easy modifications which we state as we go along. 

\begin{lemma}\label{dl} 

Let $G_{I}$ be the semigroup of finite transformations on $I$.
Let $\A\in G_{\alpha}PHA_{\alpha}$ be such that $\alpha\sim \Delta x$ is infinite for every $x\in A$. 
Then for any set $\beta$, such that $\alpha\subseteq \beta$, there exists $\B\in G_{\beta}PHA_{\beta},$ 
such that $\A\subseteq \Nr_{\alpha}\B$. 

\end{lemma}

\begin{demo}{Proof}
Let $\alpha\subseteq \beta$. We assume, loss of generality, that $\alpha$ and $\beta$ are ordinals with $\alpha<\beta$.
The proof is a direct adaptation of the proof of Theorem 2.6.49(i)
in \cite{HMT1}. First we show that there exists $\B\in G_{\alpha+1}PHA_{\alpha+1}$ 
such that $\A$ embeds into $\Nr_{\alpha}\B,$ then we proceed inductively. 
Let $$R = Id\upharpoonright (\alpha\times A)
 \cup \{  ((k,x), (\lambda, y)) : k, \lambda <
\alpha, x, y \in A, \lambda \notin \Delta x, y = {\mathsf s}_{[k|\lambda]} x \}.$$
It is easy to see  that $R$ is an equivalence relation on $\alpha
\times A$.
Define the following operations on $(\alpha\times A)/R$ with $\mu, i, k\in \alpha$ and $x,y\in A$ :
\begin{equation*}\label{l5}
\begin{split}
(\mu, x)/R \lor (\mu, y)/R = (\mu, x \lor y)/R, 
\end{split}
\end{equation*}
\begin{equation*}\label{l6}
\begin{split}
(\mu, x)/R\land  (\mu, y)/R = (\mu, x\land  y)/R, 
\end{split}
\end{equation*}
 \begin{equation*}\label{l7}
\begin{split}
(\mu, x)/R\to  (\mu, y)/R = (\mu, x\to  y)/R, 
\end{split}
\end{equation*}
\begin{equation*}\label{l10}
\begin{split}
{\mathsf c}_i ((\mu, x)/R)  = (\mu, {\mathsf c}_i x )/R, \quad
\mu \in \alpha \smallsetminus
\{i\},
\end{split}
\end{equation*}
\begin{equation*}\label{l11}
\begin{split}
{\mathsf s}_{[j|i]} ((\mu, x)/R)  = (\mu, {\mathsf s}_{[j|i]} x )/R, \quad \mu \in \alpha
\smallsetminus \{i, j\}.
\end{split}
\end{equation*}
It can be checked that these operations are well defined.
Let $$\C=((\alpha\times A)/R, \lor, \land, \to, 0, {\sf c_i}, {\sf s}_{i|j]})_{i,j\in \alpha},$$
and
let $$h=\{(x, (\mu,x)/R): x\in A, \mu\in \alpha\sim \Delta x\}.$$
Then $h$ is an isomorphism from $\A$ into $\C$. 
Now to show that $\A$ neatly embeds into $\alpha+1$ extra dimensions, we define the operations ${\sf c}_{\alpha}, {\sf s}_{[i|\alpha]}$
and ${\sf s}_{[\alpha|i]}$ on $\C$ as follows:
$${\mathsf c}_\alpha = \{ ((\mu, x)/R, (\mu, {\mathsf c}_\mu x)/R) :
\mu \in \alpha, x \in B \},$$
$${\mathsf s}_{[i|\alpha]} = \{ ((\mu, x)/R, (\mu, {\mathsf s}_{[i|\mu]}
x)/R) : \mu \in \alpha \smallsetminus \{i\}, x \in B \},$$
$${\mathsf s}_{[\alpha|i]} = \{ ((\mu, x)/R, (\mu, {\mathsf s}_{[\mu|i]}
x)/R) : \mu \in \alpha \smallsetminus \{i\}, x \in B \}.$$
Let $$\B=((\alpha\times A)/R, \lor,\land, \to, {\sf c}_i, {\sf s}_{[i|j]})_{i,j\leq \alpha}.$$
Then $$\B\in G_{\alpha+1}PA_{\alpha+1}\text{ and }h(\A)\subseteq \Nr_{\alpha}\B.$$
It is not hard to check that the defined operations are as desired. We have our result when $G$ consists only of replacements.
But since $\alpha\sim \Delta x$ is infinite one can show that substitutions corresponding to all finite transformations are term definable.
For a finite transformation $\tau\in {}^{\alpha}\alpha$ we write $[u_0|v_0, u_1|v_1,\ldots,
u_{k-1}|v_{k-1}]$ if $sup\tau=\{u_0,\ldots ,u_{k-1}\}$, $u_0<u_1
\ldots <u_{k-1}$ and $\tau(u_i)=v_i$ for $i<k$.
Let $\A\in GPHA_{\alpha}$ be such that $\alpha\sim \Delta x$ is
infinite for every $x\in A$. If $\tau=[u_0|v_0, u_1|v_1,\ldots,
u_{k-1}|v_{k-1}]$ is a finite transformation, if $x\in A$ and if
$\pi_0,\ldots ,\pi_{k-1}$ are in this order the first $k$ ordinals
in $\alpha\sim (\Delta x\cup Rg(u)\cup Rg(v))$, then
$${\mathsf s}_{\tau}x={\mathsf s}_{v_0}^{\pi_0}\ldots
{\mathsf s}_{v_{k-1}}^{\pi_{k-1}}{\mathsf s}_{\pi_0}^{u_0}\ldots
{\mathsf s}_{\pi_{k-1}}^{u_{k-1}}x.$$
The ${\sf s}_{\tau}$'s so defined satisfy the polyadic axioms, cf \cite{HMT1} Theorem 1.11.11.
Then one proceeds by a simple induction to show that for all $n\in \omega$ there exists $\B\in G_{\alpha+n}PHA_{\alpha+n}$ 
such that $\A\subseteq \Nr_{\alpha}\B.$ For the transfinite, one uses ultraproducts \cite{HMT1} theorem 2.6.34. 
\end{demo}

\begin{lemma}\label{cylindrify} With $\A$ and $\B$ as in the previous lemmand for all $X\subseteq 
A$, one has $\Sg^{\A}X=\Nr_{\alpha}\Sg^{\B}X.$
\end{lemma}

\begin{proof}let $\A\subseteq \Nr_{\alpha}\B$ and $A$ generates $\B$ then $\B$ consists of all elements ${\sf s}_{\sigma}^{\B}x$ such that 
$x\in A$ and $\sigma$ is a finite transformation on $\beta$ such that
$\sigma\upharpoonright \alpha$ is one to one \cite{HMT1} lemma 2.6.66. 
Now suppose $x\in \Nr_{\alpha}\Sg^{\B}X$ and $\Delta x\subseteq
\alpha$, then there exist $y\in \Sg^{\A}X$ and a finite transformation $\sigma$
of $\beta$ such that $\sigma\upharpoonright \alpha$ is one to one
and $x={\sf s}_{\sigma}^{\B}y.$  
Let $\tau$ be a finite
transformation of $\beta$ such that $\tau\upharpoonright  \alpha=Id
\text { and } (\tau\circ \sigma) \alpha\subseteq \alpha.$ Then
$x={\sf s}_{\tau}^{\B}x={\sf s}_{\tau}^{\B}{\sf s}_{\sigma}y=
{\sf s}_{\tau\circ \sigma}^{\B}y={\sf s}_{\tau\circ
\sigma\upharpoonright \alpha}^{\A}y.$
In the presence of diagonal elements, one defines them in the bigger algebra (the dilation) precisely  as in \cite{HMT1}, theorem 2.6.49(i).
The next lemma formulated only for $GPHA_{\alpha}$ will be used in proving our main (algebraic) result.
The proof works without any modifications when we add diagonal elements. The lemma says, roughly, that if we have an $\alpha$ dimensional  
algebra $\A$, and a set $\beta$ containing $\alpha$, then we can find an extension $\B$ of $\A$ in $\beta$ dimensions, specified by a 
carefully chosen  subsemigroup of $^{\beta}\beta$, such that $\A=\Nr_{\alpha}\B$ and  for all $b\in B$, $|\Delta b\sim \alpha|<\omega$.
$\B$ is not necessarily the minimal dilation of $\A$, because the large subsemigroup chosen maybe 
smaller than the semigroup used to form 
the unique dilation. It can happen that this extension is the minimal dilation, 
but in the case we consider all transformations, the constructed algebra 
is only a proper subreduct of the dilation obtained basically by discarding 
those elements $b$ in the original dilation for which $\Delta b\sim \alpha$ 
is infinite.
\end{proof}

\subsection{Algebraic Proofs of main theorems}
Our work in this section is closely related to that in 
\cite{Hung}. Our main theorem is a typical representabilty result, where we start with an abstract (free) algebra, 
and we find a non-trivial homomorphism from this algebra
to a concrete algebra based on Kripke systems (an algebraic version of Kripke frames).

The idea (at least for the equality-free case) is that we start with a theory 
(which is defined as a pair of sets of formulas, as is the case with classical intuitionistic logic), extend 
it to a saturated one in enough spare dimensions, or an appropraite dilation (lemma \ref{t2}),  
and then iterate this process countably many times forming  consecutive (countably many) dilations
in enough spare dimensions, using pairs of pairs (theories), cf. lemma \ref{t3}; finally forming an
 extension that will be used  to construct desired Kripke models (theorem \ref{main}). 
The extensions constructed are essentially conservative extensions, 
and they will actually constitute the set of worlds of our desired Kripke model.

The iteration is done by a subtle 
zig-zag process, a technique due to 
Gabbay \cite{b}. When we have diagonal elements (equality), constructing desired Kripke model, 
is substantialy different, and much more intricate.

All definitions and results up to lemma \ref{main1}, 
though formulated only for the diagonal-free case, applies equally well to the case when there are diagonal elements, 
with absolutely no modifications. (The case when diagonal elements are present will be dealt with in part 2).

\begin{definition} Let $\A\in GPHA_{\alpha}$. 
\begin{enumarab}
\item  A theory in $\A$ is a pair $(\Gamma, \Delta)$ such that $\Gamma, \Delta\subseteq \A$.
\item A theory $(\Gamma, \Delta)$ is consistent if there are no $a_1,\ldots a_n\in \Gamma$ and 
$b_1,\ldots b_m\in \Delta$ ($m,n\in \omega$) such that
$$a_1\land\ldots a_n\leq b_1\lor\ldots b_m.$$
Not that in this case, we have $\Gamma\cap \Delta=\emptyset$. Also if $F$ is a filter 
(has the finite intersection property), then it is always the case that
$(F, \{0\})$ is consistent.
\item A theory $(\Gamma, \Delta)$ is complete if for all $a\in A,$ either $a\in \Gamma$ or $a\in \Delta$.
\item A theory $(\Gamma, \Delta)$ is saturated if for all $a\in A$ and $j\in \alpha$, 
if ${\sf c}_ja\in \Gamma$,
then there exists $k\in \alpha\sim \Delta a$, such that ${\sf s}^j_ka\in \Gamma$.
Note that a saturated theory depends only on $\Gamma$.
%\item A model of $(\Gamma,\Delta)$ is a pair $(h,\mathfrak{F})$ where $h:\A\to \mathfrak{F}$ is a homorphism such that all elements in $\Gamma$
%are true in $w$ and alll elements of $\Delta$ are false in $w$.
\end{enumarab}
\end{definition}

\begin{lemma} \label{t1}Let $\A\in GPHA_{\alpha}$ and $(\Gamma,\Delta)$ be a consistent theory. 
\begin{enumroman}
\item For any $a\in A,$ either $(\Gamma\cup \{a\}, \Delta)$ or $(\Gamma, \Delta\cup\{a\})$ is consistent.
\item $(\Gamma,\Delta)$ can be extended to a complete theory in $\A.$
\end{enumroman}
\end{lemma}
\begin{demo}{Proof} \begin{enumroman}
\item Cf.  \cite{Hung}. Suppose for contradiction that both theories are inconsistent.
Then we have $\mu_1\land a\leq \delta_1$ and $\mu_2\leq a\land \delta_2$ where 
$\mu_1$ and $\mu_2$ are some conjunction of elements of $\Gamma$
and $\delta_1$, $\delta_2$ are some disjunction of elements of $\Delta$.
But from 
$(\mu_1\land a\to \delta_1)\land (\mu_2\to a\lor \delta_2)\leq (\mu_1\land \mu_2\to \delta_1\lor \delta_2),$
we get
$\mu_1\land \mu_2\leq \delta_1\lor \delta_2,$ which contradicts the consistency of $(\Gamma, \Delta)$.

\item Cf. \cite{Hung}. Assume that $|A|=\kappa$. Enumerate the elements of $\A$ as $(a_i:i<\kappa)$. 
Then we can extend $(\Gamma, \Delta)$ consecutively by adding $a_i$ either
to $\Gamma$ or $\Delta$ while preserving consistency. In more detail, we define by transfinite induction a sequence of 
theories $(\Gamma_i,\Delta_i)$ for 
$i\in \kappa$ as follows. Set $\Gamma_0=\Gamma$ and $\Delta_0=\Delta$. If $\Gamma_i,\Delta_i$ are defined for all $i<\mu$ 
where $\mu$ is a limit ordinal, let $\Gamma_{\mu}=(\bigcup_{i\in \mu} \Gamma_i, \bigcup_{i\in \mu} \Delta_i)$. Now for successor ordinals. 
Assume that $(\Gamma_i, \Delta_i)$ are defined.
Set $\Gamma_{i+1}=\Gamma_i\cup \{a_i\}, \Delta_{i+1}=\Delta_i$ in case this is consistent, else 
set $\Gamma_{i+1}=\Gamma_i$ and $\Delta_{i+1}=\Delta_i\cup \{a_i\}$.
Let $T=\bigcup_{i\in \kappa}T_i$ and  $F= \bigcup_{i\in \kappa} F_i$, then $(T, F)$ is as desired.
\end{enumroman}
\end{demo}

\begin{lemma}\label{t2} Let $\A\in GPHA_{\alpha}$ and $(\Gamma,\Delta)$  be a consistent theory of $\A$. 
Let $I$ be a set such that $\alpha\subseteq I$  and  let $\beta=|I\sim \alpha|=\max(|A|, |\alpha|).$ Then there exists a minimal dilation
$\B$ of $\A$ of dimension $I$, and a theory $(T,F)$ in $\B$, 
extending $(\Gamma,\Delta)$ such that $(T,F)$ is saturated and complete.
\end{lemma}

\begin{demo}{Proof}
Let $I$ be provided as in the statement of the lemma. By lemma \ref{dl}, 
there exists $\B\in GPHA_I$ such that $\A\subseteq \Nr_{\alpha}\B$ and $\A$ generates $\B$.
We also have for all $X\subseteq \A$, $\Sg^{\A}X=\Nr_{\alpha}\Sg^{\B}X$.
Let $\{b_i:i<\kappa\}$ be an enumeration of the elements of $\B$; here $\kappa=|B|.$ 
Define by transfinite recursion a sequence $(T_i, F_i)$ for  $i<\kappa$ of
theories as follows. Set $T_0=\Gamma$ and $F_0=\Delta$. We assume inductively that 
$$|\beta\sim \bigcup_{x\in T_i} \Delta x\cup \bigcup_{x\in F_i}\Delta x|\geq \omega.$$
This is clearly satisfied for $F_0$ and $T_0$. Now we need to worry only about successor ordinals.
Assume that $T_i$ and $F_i$ are defined. 
We distinguish between two cases:
\begin{enumerate}
\item $(T_i, F_i\cup \{b_i\})$ is consistent.
Then set $T_{i+1}=T_i$ and $F_{i+1}=F_i\cup \{b_i\}.$
\item  If not, that is if $(T_i, F_i\cup \{b_i\})$ is inconsistent.
In this case, we distinguish between two subcases:

(a) $b_i$ is not of the form ${\sf c}_jp.$
Then set
$T_{i+1}=T_i\cup \{b_i\}$ and $F_{i+1}=F_i$.

(b) $b_i={\sf c}_jp$ for some $j\in I$.
Then set
$T_{i+1}=T_i\cup \{{\sf c}_jp, {\sf s}_u^jp\}$ where $u\notin \Delta p\cup \bigcup_{x\in T_i}\cup \bigcup_{x\in F_i}\Delta x$ and $F_{i+1}=F_i$.
\end{enumerate}
Such a $u$ exists by the inductive assumption. Now we check by induction that each $(T_i, F_i)$ is consistent.
The only part that needs checking, in view of the previous lemma,  is subcase (b).
So assume that $(T_i,F_i)$ is consistent and $b_i={\sf c}_jp.$
If $(T_{i+1}, F_{i+1})$ is inconsistent, then we would have for some $a\in T_i$ and some $\delta\in F_i$ that
$a\land {\sf c}_jp\land {\sf s}_u^jp\leq \delta.$ From this we get
$a\land {\sf c}_jp\leq \delta,$ because ${\sf s}_u^jp\leq {\sf c}_jp.$
But this contradicts the consistency of 
$(T_i\cup \{{\sf c}_jp\}, F_i)$.
Let $T=\bigcup_{i\in \kappa}T_i$ and $F=\bigcup_{i\in \kappa} F_i$, then $(T,F)$ is consistent. 
We show that it is saturated. If ${\sf c}_jp\in T$, then ${\sf c}_jp\in T_{i+1}$ for some $i$,
hence ${\sf s}_u^jp\in T_{i+1}\subseteq T$ and $u\notin \Delta p$.
Now by lemma \ref{t1}, we can extend $(T,F)$ is $\B$ to a complete theory, 
and this will not affect saturation, since the process of completion does not take us out of 
$\B$.
\end{demo}

The next lemma constitutes the core of our construction; involving a zig-zag Gabbay construction, 
it will be used repeatedly, to construct our desired representation
via a set algebra based on a Kripke system defined in  \ref{Kripke}

\begin{lemma}\label{t3} Let $\A\in GPHA_{\alpha}$ be generated by $X$ and let $X=X_1\cup X_2$.  
Let $(\Delta_0, \Gamma_0)$, $(\Theta_0, \Gamma_0^*)$ be two consistent theories in $\Sg^{\A}X_1$ and $\Sg^{\A}X_2,$ respectively
such that $\Gamma_0\subseteq \Sg^{\A}(X_1\cap X_2)$, $\Gamma_0\subseteq \Gamma_0^*$. Assume further that 
$(\Delta_0\cap \Theta_0\cap \Sg^{\A}X_1\cap \Sg^{\A}X_2, \Gamma_0)$ is complete in $\Sg^{\A}X_1\cap \Sg^{\A}X_2$. 
Suppose that $I$ is a set such that $\alpha\subseteq I$ and $|I\sim \alpha|=max (|A|,|\alpha|)$. 
Then there exist a dilation $\B\in GPHA_I$ of $\A$, and theories $T_1=(\Delta_{\omega}, \Gamma_{\omega})$, 
$T_2=(\Theta_{\omega}, \Gamma_{\omega}^*)$ extending 
$(\Delta_0, \Gamma_0)$, $(\Theta_0, \Gamma_0^*)$, such that $T_1$ and $T_2$ are consistent and saturated in 
$\Sg^{\B}X_1$ and $\Sg^{\B}X_2,$ respectively,  
$(\Delta_{\omega}\cap \Theta_{\omega}, \Gamma_{\omega})$ is complete in $\Sg^{\B}X_1\cap \Sg^{\B}X_2,$ and
$\Gamma_{\omega}\subseteq \Gamma_{\omega}^*$.
\end{lemma}

\begin{demo}{Proof}  Like the corresponding proof in \cite{Hung}, we will build the desired theories 
in a step-by-step zig-zag manner in a large enough dilation whose dimension is specified 
by $I$. The spare dimensions play a role of added witnesses, that will allow us to eliminate quantifiers, in a sense.
Let $\A=\A_0\in GPHA_{\alpha}$. The proof consists of an iteration of lemmata \ref{t1} and \ref{t2}.
Let $\beta=max(|A|, |\alpha|)$, and let $I$ be such that $|I\sim \alpha|=\beta$.

We distinguish between two cases:

%\begin{enumarab}

Assume  that $G$ is strongly rich or $G$ contains consists of all finite transformations. In this case we only deal with minimal dilations.
We can write  
$\beta = I\sim \alpha$ as $\bigcup_{n=1}^{\infty}C_n$ where $C_i\cap C_j=\emptyset$ for distinct $i$ and $j$ and $|C_i|=\beta$
for all $i$. Then iterate first two items in lemma \ref{dl}.
Let $\A_1=\A(C_1)\in G_{\alpha\cup C_1}PHA_{\alpha\cup C_1}$ be a minimal dilation of $\A$, so that $\A=\Nr_{\alpha}\A_1$.
Let $\A_2=\A(C_1)(C_2)$ be a minimal dilation of $\A_1$ so that $\A_1=\Nr_{\alpha\cup C_1}\A_2$. 
Generally, we define inductively $\A_n=\A(C_1)(C_2)\ldots (C_n)$ to be  a minimal dilation of $\A_{n-1}$, so that 
$\A_{n-1}=\Nr_{\alpha\cup C_1\cup \ldots C_{n-1}}\A_n$.
Notice that for $k<n$, $\A_n$ is a minimal dilation of $\A_k$.
So we have a sequence of algebras 
$\A_0\subseteq \A_1\subseteq \A_2\ldots.$ Each element in the sequence is the minimal dilation of its preceding one.

%\item $G$ contains all transformations. Here we shall have to use minimal extensions at the start, i.e at the first step of the iteration.  
%We iterate lemma \ref{dl}, using items (3) and (4) in lemma \ref{cylindrify} by taking $|C_1|=\beta$, 
%and $|C_i|=\omega$ for all $i\geq 2$; this will yield the desired sequence of extensions.

%\end{enumarab}
Now that we have a sequence of extensions $\A_0\subseteq \A_1\ldots$ in different increasing dimensions,  
we now form a limit of this sequence
in $I$ dimensions. We can use ultraproducts, but instead we use products, and quotient algebras. First form the Heyting algebra, that is the product of the Heyting reducts of the constructed algebras, that is take   
$\C=\prod_{n=0}^{\infty}\Rd A_n$, where $\Rd \A_n$ denotes the Heyting reduct of $\A_n$ obtained by discarding substitutions and cylindrifiers. 
Let 
$$M=\{f\in C: (\exists n\in \omega)(\forall k\geq n) f_{k}=0\}.$$
Then $M$ is a Heyting ideal of $\C$. Now form the quotient Heyting  algebra $\D=\C/M.$
We want to expand this Heyting algebra algebra by cylindrifiers and substitutions, i.e to an algebra in $GPHA_{I}$.
Towards this aim, for $\tau\in {}G,$ define
$\phi({\tau})\in {} ^CC$ as follows:
$$(\phi(\tau)f)_n={\sf s}_{\tau\upharpoonright dim \A_n}^{\A_n}f_n$$
if $\tau(dim(\A_n))\subseteq dim (\A_n)$.
Otherwise $$(\phi(\tau)f)_n=f_n.$$
For $j\in I$,  define
$${\sf c}_jf_n={\sf c}_{(dim \A_n\cap \{j\})}^{\A_n}f_n,$$
and
$${\sf q}_jf _n={\sf q}_{(dim \A_n\cap \{j\})}^{\A_n}f_n.$$
Then for $\tau\in G$ and $j\in I$,  set
$${\sf s}_{\tau}(f/M)=\phi({\tau})f/M,$$
$${\sf c}_{j}(f/M)=({\sf c}_j f)/M,$$
and
$${\sf q}_{j}(f/M)=({\sf q}_j f)/M.$$
Then, it can be easily checked that,  $\A_{\infty}=(\D, {\sf s}_{\tau}, {\sf c}_{j}, {\sf q}_{j})$ is a $GPHA_I$,
in which every $\A_n$ neatly embeds. We can and will assume that $\A_n=\Nr_{\alpha\cup C_1\ldots \cup C_n}\A_{\infty}$. 
Also $\A_{\infty}$ is a minimal dilation of $\A_n$ for all $n$.
During our 'zig-zagging' 
we shall be extensively using lemma \ref{cylindrify}.

From now on, fix $\A$ to be as in the statement of lemma \ref{t3} for some time to come. 
So $\A\in GPHA_{\alpha}$ is generated by $X$ and $X=X_1\cup X_2$.  
$(\Delta_0, \Gamma_0)$, $(\Theta_0, \Gamma_0^*)$ are two consistent theories in $\Sg^{\A}X_1$ and $\Sg^{\A}X_2,$ respectively
such that $\Gamma_0\subseteq \Sg^{\A}(X_1\cap X_2)$, $\Gamma_0\subseteq \Gamma_0^*$. Finally  
$(\Delta_0\cap \Theta_0\cap \Sg^{\A}X_1\cap \Sg^{\A}X_2, \Gamma_0)$ is complete in $\Sg^{\A}X_1\cap \Sg^{\A}X_2.$ 
Now we have:
$$\Delta_0\subseteq \Sg^{\A}X_1\subseteq \Sg^{\A(C_1)}X_1\subseteq \Sg^{\A(C_1)(C_2)}X_1\subseteq \Sg^{\A(C_1)(C_2)(C_3)}X_1 \ldots\subseteq 
\Sg^{\A_{\infty}}X_1.$$ 
$$\Theta_0\subseteq \Sg^{\A}X_2\subseteq \Sg^{\A(C_1)}X_2\subseteq 
\Sg^{\A(C_1)(C_2)}X_2\subseteq \Sg^{\A(C_1)(C_2)(C_3)}X_2 \ldots\subseteq \Sg^{\A_{\infty}}X_2. $$
In view of lemmata  \ref{t1}, \ref{t2}, extend $(\Delta_0, \Gamma_0)$ 
to a complete and saturated theory $(\Delta_1, \Gamma_1')$ in $\Sg^{\A(C_1)}X_1$. Consider $(\Delta_1, \Gamma_0)$. Zig-zagging away,
we extend our theories in a step by step manner. The proofs of the coming Claims, 1, 2 and 3, 
are very similar to the proofs of the corresponding claims in \cite{Hung}, which are in turn an algebraic version of lemmata 4.18-19-20 in \cite {b},
with one major difference from the former. 
In our present situation, we can cylindrify on only finitely many indices, so we have to be 
careful, when talking about dimension sets, and in forming neat reducts (or compressions). 
Our proof then becomes substantially more involved. In the course of our proof we use extensively lemmata 
\ref{dl} and \ref{cylindrify} which are not formulated
in \cite{Hung} because we simply did not need them when we had cylindrifications on possibly 
infinite sets.

\begin{athm}{Claim 1} The theory $T_1=(\Theta_0\cup (\Delta_1\cap \Sg^{\A(C_1)}X_2), \Gamma_0^*)$ is consistent in $\Sg^{\A(C_1)}X_2.$
\end{athm}
\begin{demo}{Proof of Claim 1}
Assume that $T_1$ is inconsistent. Then 
for some conjunction $\theta_0$ of elements in $\Theta_0$, some $E_1\in \Delta_1\cap \Sg^{\A(C_1)}X_2,$
and some 
disjunction $\mu_0^*$ in $\Gamma_0^*,$ we have
$\theta_0\land E_1\leq \mu_0^*,$
and so 
$E_1\leq \theta_0\rightarrow \mu_0^*.$
Since $\theta_0\in \Theta_0\subseteq \Sg^{\A}X_2$ 
and $\mu_0^*\in \Gamma_0^*\subseteq \Sg^{\A}X_2\subseteq \Nr_{\alpha}^{\A(C_1)}\A$, 
therefore, for any finite 
set $D\subseteq C_1\sim \alpha$, we have ${\sf c}_{(D)}\theta_0=\theta_0$ and ${\sf c}_{(D)}\mu_0^*=\mu_0^*$.
Also for any finite set $D\subseteq C_1\sim \alpha,$ we have 
${\sf c}_{(D)}E_1\leq {\sf c}_{(D)}(\theta_0\to \mu_0^*)=\theta_0\to \mu^*.$
Now $E_1\in \Delta_1$, hence $E_1\in \Sg^{\A(C_1)}X_1$. 
By definition, we also have $E_1\in \Sg^{\A(C_1)}X_2.$
By lemma \ref{cylindrify} there exist finite sets $D_1$ and $D_2$ contained in $C_1\sim \alpha,$ such that
$${\sf c}_{(D_1)}E_1\in \Nr_{\alpha}\Sg^{\A(C_1)}X_1$$ and 
$${\sf c}_{(D_2)}E_1\in \Nr_{\alpha}\Sg^{\A(C_1)}X_2.$$
Le $D=D_1\cup D_2$. Then $D\subseteq  C_1\sim \alpha$ and we have: 
$${\sf c}_{(D)}E_1\in \Nr_{\alpha}\Sg^{\A(C_1)}X_1=\Sg^{\Nr_{\alpha}\A(C_1)}X_1=\Sg^{\A}X_1$$ and
$${\sf c}_{(D)}E_1\in \Nr_{\alpha}\Sg^{\A(C_1)}X_2=\Sg^{\Nr_{\alpha}\A(C_1)}X_2=\Sg^{\A}X_2,$$ 
that is to say
$${\sf c}_{(D)}E_1\in \Sg^{\A}X_1\cap \Sg^{\A}X_2.$$
Since  
$(\Delta_0\cap \Theta_0\cap \Sg^{\A}X_1\cap \Sg^{\A}X_2, \Gamma_0)$ is complete in $\Sg^{\A}X_1\cap \Sg^{\A}X_2,$
we get that ${\sf c}_{(D)}E_1$ is either in $\Delta_0\cap \Theta_0$ or $\Gamma_0$.
We show that either way leads to a contradiction, by which we will be done. Suppose it is in $\Gamma_0$. 
Recall that we extended $(\Delta_0, \Gamma_0)$ to a complete saturated extension 
$(\Delta, \Gamma')$ in $\Sg^{\A(C_1)}X_1$. Since $\Gamma_0\subseteq \Gamma_1',$ we get that ${\sf c}_{(D)}E_1\in \Gamma_1'$ hence
${\sf c}_{(D)}E_1\notin \Delta_1$  because $(\Delta_1,\Gamma_1')$ is saturated and consistent. But this 
contradicts that $E_1\in \Delta_1$ because $E_1\leq {\sf c}_{(D)}E_1.$
Thus, we can infer that  ${\sf c}_{(D)}E_1\in \Delta_0\cap \Theta_0$. In particular, it is in $\Theta_0,$ and so 
$\theta_0\rightarrow \mu_0^*\in \Theta_0$. 
But again  this contradicts the consistency of $(\Theta_0, \Gamma_0^*)$.
\end{demo}

Now we extend $T_1$ to a complete and saturated theory $(\Theta_2, \Gamma_2^*)$ in $\Sg^{\A(C_1)(C_2)}X_2$. Let 
$\Gamma_2=\Gamma_2^*\cap \Sg^{\A(C_1)(C_2)}X_1$. 

\begin{athm}{Claim 2} The theory $T_2=(\Delta_1\cup (\Theta_2\cap \Sg^{\A(C_1)(C_2))}X_1), \Gamma_2)$ 
is consistent in $\Sg^{\A(C_1)(C_2)}X_1$.
\end{athm}
\begin{demo}{Proof of Claim 2}  If the Claim fails to hold, then we would have some 
$\delta_1\in \Delta_1$, $E_2\in \Theta_2\cap \Sg^{\A(C_1)(C_2)}X_1,$  and a disjunction $\mu_2\in \Gamma_2$  such that  
$\delta_1\land E_2\rightarrow \mu_2,$
and so 
$\delta_1\leq (E_2\rightarrow \mu_2)$
since  $\delta_1\in \Delta_1\subseteq \Sg^{\A(C_1)}X_1$. But $\Sg^{\A(C_1)}X_1\subseteq \Nr_{\alpha\cup C_1}^{\A(C_1)(C_2)}X_1$, 
therefore for any finite set $D\subseteq C_2\sim C_1,$ 
we have ${\sf q}_{(D)}\delta_1=\delta_1.$
The following holds for any finite set $D\subseteq C_2\sim C_1,$
$$\delta_1\leq {\sf q}_{(D)}(E_2\rightarrow \mu_2).$$
Now, by lemma \ref{cylindrify}, there is a finite set $D\subseteq C_2\sim C_1,$ satisfying 
\begin{equation*}
\begin{split}
\delta_1\to {\sf q}_{(D)}(E_2\rightarrow \mu_2)
&\in \Nr_{\alpha\cup C_1}\Sg^{\A(C_1)(C_2)}X_2,\\
&=\Sg^{\Nr_{\alpha\cup C_1}\A(C_1)(\A(C_2)}X_2,\\
&=\Sg^{\A(C_1)}X_2.\\
\end{split}
\end{equation*}
Since $\delta_1\in \Delta_1$, and $\delta_1\leq {\sf q}_{(D)}(E_2\to \mu_2)$,  
we get that ${\sf q}_{(D)}(E_2\rightarrow \mu_2)$ is in $\Delta_1\cap \Sg^{\A(C_1)}X_2$. 
We proceed as in the previous claim replacing $\Theta_0$ by $\Theta_2$ and the existental quantifier by the universal one.
Let $E_1= {\sf q}_{(D)}(E_2\to \mu_2)$. Then $E_1\in \Sg^{\A(C_1)}X_1\cap \Sg^{\A(C_2)}X_2$.
By lemma \ref{cylindrify} there exist finite sets $D_1$ and $D_2$ contained in $C_1\sim \alpha$ such that
$${\sf q}_{(D_1)}E_1\in \Nr_{\alpha}\Sg^{\A(C_1)}X_1,$$ and 
$${\sf q}_{(D_2)}E_1\in \Nr_{\alpha}\Sg^{\A(C_1)}X_2.$$
Le $J=D_1\cup D_2$. Then $J\subseteq  C_1\sim \alpha,$ and we have: 
$${\sf q}_{(J)}E_1\in \Nr_{\alpha}\Sg^{\A(C_1)}X_1=\Sg^{\Nr_{\alpha}\A(C_1)}X_1=\Sg^{\A}X_1$$ and
$${\sf q}_{(J)}E_1\in \Nr_{\alpha}\Sg^{\A(C_1)}X_2=\Sg^{\Nr_{\alpha}\A(C_1)}X_2=\Sg^{\A}X_2.$$ 
That is to say,
$${\sf q}_{(J)}E_1\in \Sg^{\A}X_1\cap \Sg^{\A}X_2.$$
Now $(\Delta_0\cap \Theta_2\cap \Sg^{\A}X_1\cap \Sg^{\A}X_2, \Gamma_0)$ is complete in $\Sg^{\A}X_1\cap \Sg^{\A}X_2,$
we get that ${\sf q}_{(J)}E_1$ is either in $\Delta_0\cap \Theta_2$ or $\Gamma_0$.
Suppose it is in $\Gamma_0$. Since $\Gamma_0\subseteq \Gamma_1'$, we get that ${\sf q}_{(J)}E_1\in \Gamma_1'$, hence
${\sf q}_{(J)}E_1\notin \Delta_1,$  because $(\Delta_1,\Gamma_1')$ is saturated and consistent. 
Here, recall that,  $(\Delta, \Gamma')$ is a saturated complete extension of 
$(\Gamma, \Delta)$.  But this 
contradicts that $E_1\in \Delta_1$.
Thus, we can infer that  ${\sf q}_{(J)}E_1\in \Delta_0\cap \Theta_2$.  In particular, it is in $\Theta_2$. 
Hence ${\sf q}_{(D\cup J)}(E_2\to \mu_2)\in \Theta_2$, and so $E_2\to \mu_2\in \Theta_2$ since 
${\sf q}_{(D\cup J)}(E_2\to \mu_2)\leq E_1\to \mu_2$. But this is a contradiction,
since $E_2\in \Theta_2$, $\mu_2\in \Gamma_2^*$ and $(\Theta_2,\Gamma_2^*)$ is consistent. 
\end{demo}
Extend $T_2$ to a complete and saturated theory $(\Delta_3, \Gamma_3')$ in $\Sg^{\A(C_1)(C_2)(C_3)}X_1$
such that $\Gamma_2\subseteq \Gamma_3'$. Again we are interested only in $(\Delta_3, \Gamma_2)$.

\begin{athm}{Claim 3 } The theory $T_3=(\Theta_2\cup \Delta_3\cap \Sg^{\A(C_1)(C_2)(C_3)}X_2, \Gamma_2^*)$ is consistent in 
$\Sg^{\A(C_1)(C_2)(C_3)}X_2.$
\end{athm}
\begin{demo}{Proof of Claim 3}  Seeking a contradiction, assume that the Claim does not hold. Then we would get 
for some $\theta_2\in \Theta_2$, $E_3\in \Delta_3\cap \Sg^{\A(C_1)(C_2)(C_3)}X_2$ and some disjunction $\mu_2^*\in \Gamma_2^*,$  that
$\theta_2\land E_3\leq \mu_2^*.$
Hence $E_3\leq \theta_2\rightarrow \mu_2^*.$
For any finite set $D\subseteq C_3\sim (C_1\cup C_2),$ we have ${\sf c}_{(D)}E_3\leq \theta_2\rightarrow \mu_2^*$.
By lemma \ref{cylindrify},  there is a  finite set $D_3\subseteq C_3\sim (C_1\cup C_2),$ satisfying  
\begin{equation*}
\begin{split}
{\sf c}_{(D_3)}E_3
&\in \Nr_{\alpha\cup C_1\cup C_2}\Sg^{\A(C_1)(C_2)(C_3)}X_1\\
&=\Sg^{\Nr_{\alpha\cup C_1\cup C_2}\A(C_1)C_2)(C_3)}X_1\\
&= \Sg^{\A(C_1)(C_2)}X_1.\\ 
\end{split}
\end{equation*}
If ${\sf c}_{(D_3)}E_3\in \Gamma_2^*$, then it in $\Gamma_2$, 
and since $\Gamma_2\subseteq \Gamma_3'$, it cannot be in $\Delta_3$.
But this contradicts that $E_3\in \Delta_3$. So ${\sf c}_{(D_3)}E_3\in \Theta_2,$ because $E_3\leq {\sf c}_{(D_3)}E_3,$ 
and so $(\theta_2\rightarrow \mu_2^*)\in \Theta_2,$ which contradicts 
the consistency of $(\Theta_2, \Gamma_2^*).$
\end{demo}
 Likewise, now extend $T_3$ to a complete and saturated theory $(\Delta_4, \Gamma_4')$ in $\Sg^{\A(C_1)(C_2)(C_3)(C_4)}X_2$
such that $\Gamma_3\subseteq \Gamma_4'.$ As before the theory 
$(\Delta_3, \Theta_4\cap \Sg^{\A(C_1)(C_2)(C_3)(C_4)}X_1, \Gamma_4)$ is consistent in
$\Sg^{\A(C_1)(C_2)(C_3)(C_4)}X_1$.
Continue, inductively,  to construct $(\Delta_5, \Gamma_5')$, $(\Delta_5, \Gamma_4)$
and so on.
We obtain, zigzaging along, the following sequences: 
$$(\Delta_0, \Gamma_0), (\Delta_1, \Gamma_0), (\Delta_3, \Gamma_2)\ldots$$
$$(\Theta_0, \Gamma_0^*), (\Theta_2, \Gamma_2^*), (\Theta_4, \Gamma_4^*)\ldots $$
such that
\begin{enumarab}
\item $(\theta_{2n}, \Gamma_{2n}^*)$ is complete and saturated in $\Sg^{\A(C_1)\ldots (C_{2n})}X_2,$
\item $(\Delta_{2n+1}, \Gamma_{2n})$ is a saturated theory in $\Sg^{\A(C_1)\ldots (C_{2n+1})}X_1,$
\item $\Theta_{2n}\subseteq \Theta_{2n+2}$, $\Gamma_{2n}^*\subseteq \Gamma_{2n+2}^*$ and 
$\Gamma_{2n}=\Gamma_{2n}^*\cap \Sg^{\A(C_1)\ldots \A(C_{2n})}X_1,$
\item $\Delta_0\subseteq \Delta_1\subseteq \Delta_3\subseteq \ldots .$
\end{enumarab}
Now let $\Delta_{\omega}=\bigcup_{n}\Delta_n$, $\Gamma_{\omega}=\bigcup_{n}\Gamma_n$, $\Gamma_{\omega}^*=\bigcup_{n}\Gamma_n^*$ and
$\Theta_{\omega}=\bigcup_n\Theta_n$. Then we have
 $T_1=(\Delta_{\omega}, \Gamma_{\omega})$, 
$T_2=(\Theta_{\omega}, \Gamma_{\omega}^*)$ extend 
$(\Delta, \Gamma)$, $(\Theta, \Gamma^*)$, such that $T_1$ and $T_2$ are consistent and saturated in $\Sg^{\B}X_1$ and $\Sg^{\B}X_2,$ respectively, 
$\Delta_{\omega}\cap \Theta_{\omega}$ is complete in $\Sg^{\B}X_1\cap \Sg^{\B}X_2,$ and
$\Gamma_{\omega}\subseteq \Gamma_{\omega}^*$.
We check that $(\Delta_{\omega}\cap \Theta_{\omega},\Gamma_{\omega})$
is complete in $\Sg^{\B}X_1\cap \Sg^{\B}X_2$.
Let $a\in \Sg^{\B}X_1\cap \Sg^{\B}X_2$. Then there exists $n$ such that
$a\in \Sg^{\A_n}X_1\cap \Sg^{\A_n}X_2$. Now $(\Theta_{2n}, \Gamma_{2n}^*)$ is complete and so either $a\in \Theta_{2n}$ or $a\in \Gamma_{2n}^*$.
If $a\in \Theta_{2n}$ it will be in $\Delta_{2n+1}$ and if $a\in \Gamma_{2n}^*$ it will be in $\Gamma_{2n}$. In either case,
$a\in \Delta_{\omega}\cap \Theta_{\omega}$ or $a\in \Gamma_{\omega}$. 
\end{demo}
\begin{definition}
\begin{enumarab}
\item  Let $\A$ be an algebra generated by $X$ and assume that $X=X_1\cup X_2$. A pair $((\Delta,\Gamma)$  $(T,F))$ of theories in 
$\Sg^{\A}X_1$ and $\Sg^{\A}X_2$ is a matched pair of theories if 
$(\Delta\cap T\cap \Sg^{\A}X_1\cap \Sg^{\A}X_2, \Gamma\cap F\cap \Sg^{\A}X_1\cap \Sg^{\A}X_2)$
is complete in $\Sg^{\A}X_1\cap \Sg^{\A}X_2$.
\item A theory $(T, F)$ extends a theory $(\Delta, \Gamma)$ if $\Delta\subseteq T$ and $\Gamma\subseteq F$.
\item A pair $(T_1, T_2)$ of theories extend another pair  $(\Delta_1, \Delta_2)$ if $T_1$ 
extends $\Delta_1$ and $T_2$ extends $\Delta_2.$ 
\end{enumarab}
\end{definition}
The following Corollary follows directly from the proof of lemma \ref{t3}.

\begin{corollary}\label{main1} Let $\A\in GPHA_{\alpha}$ be generated by $X$ and let $X=X_1\cup X_2$.  
Let $((\Delta_0, \Gamma_0)$, $(\Theta_0, \Gamma_0^*))$ be a matched pair in $\Sg^{\A}X_1$ and $\Sg^{\A}X_2,$ respectively.
Let $I$ be  a set such that $\alpha \subseteq I$, and $|I\sim \alpha|=max(|A|, |\alpha|)$.
Then there exists a dilation $\B\in GPHA_I$ of $\A$, and a matched pair, $(T_1, T_2)$ extending 
$((\Delta_0, \Gamma_0)$, $(\Theta_0, \Gamma_0^*))$, such that $T_1$ and $T_2$ are saturated in $\Sg^{\B}X_1$ and $\Sg^{\B}X_2$, respectively.
\end{corollary}
We next define set algebras based on Kripke systems. We stipulate that ubdirect products (in the univerasl algebraic sense) 
are the representable algebras, which the abstract axioms aspire to capture.
Here Kripke systems (a direct generalization of Kripke frames) are defined differently 
than those defined in \cite{Hung}, because we allow {\it relativized} semantics. In the clasical case, such algebras reduce to products of set algebras.
\footnote{The idea of relativization, similar to Henkin's semantics for second order logic, has proved a very 
fruitful idea in the theory of cylindric algebras.} 

\begin{definition}
Let $\alpha$ be an infinite set. 
A Kripke system of dimension $\alpha$ is a quadruple $\mathfrak{K}=(K, \leq \{X_k\}_{k\in K}, \{V_k\}_{k\in K}),$ such that
$V_k\subseteq {}^{\alpha}X_k,$
and
\begin{enumarab}
\item $(K,\leq)$ is preordered set,
\item For any $k\in K$, $X_k$ is a non-empty set such that
$$k\leq k'\implies X_k\subseteq  X_{k'}\text { and } V_k\subseteq V_{k'}.$$
\end{enumarab}
\end{definition}\label{Kripke} 
Let $\mathfrak{O}$ be the Boolean algebra $\{0,1\}$. Now Kripke systems define concrete polyadic Heyting algebras as follows.
Let $\alpha$ be an infinite set and $G$ be a 
semigroup of transformations on $\alpha$. Let $\mathfrak{K}=(K,\leq \{X_k\}_{k\in K}, \{V_k\}_{k\in K})$ 
be a Kripke system.
Consider the set
$$\mathfrak{F}_{\mathfrak{K}}=\{(f_k:k\in K); f_k:V_k\to \mathfrak{O}, k\leq k'\implies f_k\leq f_{k'}\}.$$
If $x,y\in {}^{\alpha}X_k$ and $j\in \alpha$ we write $x\equiv_jy$ if $x(i)=y(i)$ for all $i\neq j$.
We write $(f_k)$ instead of $(f_k:k\in K)$.
In $\mathfrak{F}_{\mathfrak{K}}$ we introduce the following operations:
$$(f_k)\lor (g_k)=(f_k\lor g_k)$$
$$(f_k)\land (g_k)=(f_k\land g_k.)$$
For any $(f_k)$ and $(g_k)\in \mathfrak{F}$, define
$$(f_k)\rightarrow (g_k)=(h_k),$$
where
$(h_k)$ is given for $x\in V_k$ by $h_k(x)=1$ if and only if for any $k'\geq k$ if $f_{k'}(x)=1$ then $g_{k'}(x)=1$.
For any $\tau\in G,$  define
$${\sf s}_{\tau}:\mathfrak{F}\to \mathfrak{F}$$ by
$${\sf s}_{\tau}(f_k)=(g_k)$$
where
$$g_k(x)=f_k(x\circ \tau)\text { for any }k\in K\text { and }x\in V_k.$$
For any $j\in \alpha$ and $(f_k)\in \mathfrak{F},$
define
$${\sf c}_{j}(f_k)=(g_k),$$
where for $x\in V_k$
$$g_k(x)=\bigvee\{f_k(y): y\in V_k,\  y\equiv_j x\}.$$
Finally, set
$${\sf q}_{j}(f_k)=(g_k)$$
where for $x\in V_k,$
$$g_k(x)=\bigwedge\{f_l(y): k\leq l, \ y\in V_k, y\equiv_j x\}.$$

The diagonal element ${\sf d}_{ij}$ is defined to be the tuple $(f_k:k\in K)$ where for $x\in V_k$, $f_k(x)=1$ iff $x_i=x_j.$ 
 
The algebra $\F_{\bold K}$ is called the set algebra based on the Kripke system $\bold K$.

\subsection{Diagonal Free case}
%\begin{theorem}\label{main}

We now deal with the case when $G$ is the semigroup of all finite transformations on $\alpha$. In this case, 
we stipulate that $\alpha\sim \Delta x$ is infinite for all $x$ in algebras considered.
To deal with such a case, we need to define certain free algebras, called dimension restricted. 
Those algebras were introduced by Henkin, 
Monk and Tarski. 
The free algebras defined the usual way, will have
the dimensions sets of their elements equal to their dimension, but we do not want that. 
For a class $K$, ${\bf S}$ 
stands for the operation of forming subalgebras of $K$, ${\bf P}K$ that of forming direct products, 
and ${\bf H}K$ stands for the operation of taking homomorphic images.
In particular, for a class $K$, ${\bf HSP}K$ stands for the variety generated by $K$.    

Our dimension restricted free algebbras, are an instance of certain independently generated algebras, 
obtained by an appropriate relativization of the universal algebraic 
concept of free algebras. For an algebra $\A,$ we write
$R\in Con\A$ if $R$ is a congruence relation on $\A.$

\begin{definition} Assume that $K$ is a class of algebras of similarity $t$ and $S$ is any set of ordered pairs of words of $\Fr_{\alpha}^t,$
the absolutely free algebra of type $t$. Let
$$Cr_{\alpha}^{(S)}K=\cap \{R\in Con \Fr_{\alpha}^t, \Fr_{\alpha}^t/R\in SK, S\subseteq R\}$$
and let
$$\Fr_{\alpha}^{(S)}K=\Fr_{\alpha}^t/Cr_{\alpha}^{(S)}K.$$
$\Fr_{\alpha}^{(S)}K$ is called the free algebra over $K$ with $\alpha$ generators subject to the defining relations $S$.
\end{definition}
As a special case, we obtain dimension restricted free algebra, defined next.
\begin{definition}
\begin{enumarab}
\item Let $\delta$ be a cardinal. Let $\alpha$ be an ordinal, and let $G$ be the semigroup of finite transformations on $\alpha$.
Let$_{\alpha} \Fr_{\delta}$ be the absolutely free algebra on $\delta$
generators and of type $GPHA_{\alpha}$.  Let $\rho\in
{}^{\delta}\wp(\alpha)$. Let $L$ be a class having the same
similarity type as $GPHA_{\alpha}.$ Let
$$Cr_{\delta}^{(\rho)}L=\bigcap\{R: R\in Con_{\alpha}\Fr_{\delta},
{}_{\alpha}\Fr_{\delta}/R\in \mathbf{SP}L, {\mathsf
c}_k^{_{\alpha}\Fr_{\delta}}{\eta}/R=\eta/R \text { for each }$$
$$\eta<\delta \text
{ and each }k\in \alpha\smallsetminus \rho(\eta)\}$$ and
$$\Fr_{\delta}^{\rho}L={}_{\alpha}\Fr_{\beta}/Cr_{\delta}^{(\rho)}L.$$

The ordinal $\alpha$ does not figure out in $Cr_{\delta}^{(\rho)}L$
and $\Fr_{\delta}^{(\rho)}L$ though it is involved in their
definition. However, $\alpha$ will be clear from context so that no
confusion is likely to ensue.

%In the following  $Hom(\A,\B)$ is the set of all homomorphisms from $\A$ to $\B$.

\item Assume that $\delta$ is a cardinal, $L\subseteq GPHA_{\alpha}$, $\A\in L$,
$x=\langle x_{\eta}:\eta<\beta\rangle\in {}^{\delta}A$ and $\rho\in
{}^{\delta}\wp(\alpha)$. We say that the sequence $x$ $L$-freely
generates $\A$ under the dimension restricting function $\rho$, or
simply $x$ freely generates $\A$ under $\rho,$ if the following two
conditions hold:
\begin{enumroman}
\item $\A=\Sg^{\A}Rg(x)$ and $\Delta^{\A} x_{\eta}\subseteq \rho(\eta)$ for all $\eta<\delta$.
\item Whenever $\B\in L$, $y=\langle y_{\eta}, \eta<\delta\rangle\in
{}^{\delta}\B$ and $\Delta^{\B}y_{\eta}\subseteq \rho(\eta)$ for
every $\eta<\delta$, then there is a unique homomorphism from $\A$ to $\B$, such
that $h\circ x=y$.
\end{enumroman}
\end{enumarab}
\end{definition}
The second item says that dimension restricted free algebras has the universal property of free algebras with respect to algebras whose dimensions 
are also restricted. The following theorem can be easily distilled from the literature of cylindic algebra.

\begin{theorem}
Assume that $\delta$ is a cardinal, $L\subseteq GPHA_{\alpha}$,
$\A\in L$, $x=\langle x_{\eta}:\eta<\delta\rangle\in {}^{\delta}A$
and $\rho\in {}^{\delta}\wp(\alpha).$ Then the following hold:
\end{theorem}
\begin{enumroman}
\item $\Fr_{\delta}^{\rho}L\in GPHA_{\alpha}$
and $x=\langle \eta/Cr_{\delta}^{\rho}L: \eta<\delta \rangle$
$\mathbf{SP}L$- freely generates $\A$ under $\rho$.
%\item $\Delta \eta/Cr_{\delta}^{\rho}CA_{\alpha}=\rho\eta$
\item In order that
$\A\cong \Fr_{\delta}^{\rho}L$ it is necessary and sufficient that
there exists a sequence $x\in {}^{\delta}A$ which $L$ freely
generates $\A$ under $\rho$.
\end{enumroman}
\begin{demo}{Proof} \cite{HMT1} theorems 2.5.35, 2.5.36, 2.5.37.
\end{demo}
Note that when $\rho(i)=\alpha$ for all $i$ then $\rho$ is not restricting the dimension, and we recover the notion of ordinary free algebras.
That is for such a $\rho$, we have $\Fr_{\beta}^{\rho}GPHA_{\alpha}\cong \Fr_{\beta}GPHA_{\alpha}.$

Now we formulate the analogue of  theorem \ref{main} for dimension restricted agebras, 
which adresses infinitely many cases, 
because we have infinitely many dimension restricted free algebras having the same number of generators.

Now we formulate the analogue of  theorem \ref{main} for dimension restricted agebras, 
which adresses infinitely many cases, 
because we have infinitely many dimension restricted free algebras having the same number of generators.

\begin{theorem}\label{main2}
Let $G$ be the semigroup of finite transformations on an infinite set 
$\alpha$ and let $\delta$ be a cardinal $>0$. Let $\rho\in {}^{\delta}\wp(\alpha)$ be such that
$\alpha\sim \rho(i)$ is infinite for every 
$i\in \delta$. Let $\A$ be the free  $G$ algebra generated by $X$ restristed by $\rho$;
 that is $\A=\Fr_{\delta}^{\rho}GPHA_{\alpha},$
and suppose that $X=X_1\cup X_2$.
Let $(\Delta_0, \Gamma_0)$, $(\Theta_0, \Gamma_0^*)$ be two consistent theories in $\Sg^{\A}X_1$ and $\Sg^{\A}X_2,$ respectively.
Assume that $\Gamma_0\subseteq \Sg^{\A}(X_1\cap X_2)$ and $\Gamma_0\subseteq \Gamma_0^*$.
Assume, further, that  
$(\Delta_0\cap \Theta_0\cap \Sg^{\A}X_1\cap \Sg^{\A}X_2, \Gamma_0)$ is complete in $\Sg^{\A}X_1\cap \Sg^{\A}X_2$. 
Then there exist a Kripke system $\mathfrak{K}=(K,\leq \{X_k\}_{k\in K}\{V_k\}_{k\in K}),$ a homomorphism $\psi:\A\to \mathfrak{F}_K,$
$k_0\in K$, and $x\in V_{k_0}$,  such that for all $p\in \Delta_0\cup \Theta_0$ if $\psi(p)=(f_k)$, then $f_{k_0}(x)=1$
and for all $p\in \Gamma_0^*$ if $\psi(p)=(f_k)$, then $f_{k_0}(x)=0$.
\end{theorem}
\begin{demo}{Proof} We use lemma \ref{t3}, extensively. Assume that $\alpha$, $G$, $\A$ and $X_1$, $X_2$ and everything 
else in the hypothesis
are given.  Let $I$ be  a set containing $\alpha$ such that 
$\beta=|I\sim \alpha|=max(|A|, |\alpha|).$
If $G$ is strongly rich, let $(K_n:n\in \omega)$ be a family of pairwise disjoint sets such that $|K_n|=\beta.$
Define a sequence of algebras
$\A=\A_0\subseteq \A_1\subseteq \A_2\subseteq \A_2\ldots \subseteq \A_n\ldots,$
such that
$\A_{n+1}$ is a minimal dilation of $\A_n$ and $dim(\A_{n+1})=\dim\A_n\cup K_n$.

If $G={}^{\alpha}\alpha$, then let $(K_n:n\in \omega\}$ be a family of pairwise disjoint sets, such that $|K_1|=\beta$ 
and $|K_n|=\omega$ for $n\geq 1$, and define a sequence of algebras
$\A=\A_0\subseteq \A_1\subseteq \A_2\subseteq \A_2\ldots \subseteq \A_n\ldots,$
such that $\A_1$ is a minimal extension of $\A$, and 
$\A_{n+1}$ is a minimal dilation of $\A_n$ for $n\geq 2$, with  $dim(\A_{n+1})=\dim\A_n\cup K_n$.

We denote $dim(\A_n)$ by $I_n$ for $n\geq 1$. Recall that $dim(\A_0)=\dim\A=\alpha$.

We interrupt the main stream of the proof by two consecutive claims. Not to digress, it might be useful that the reader 
at first reading, only memorize their statements, skip their proofs, go on with the main proof, and then get back to them.
The proofs of Claims 1 and 2 to follow are completely analogous to the corresponding claims in \cite{Hung}. 
The only difference is that we deal with only finite 
cylindrifiers, and in this respect they are closer to the proofs of lemmata 4.22-23 in \cite{b}. Those two claims are essential 
in showing that the maps that will be defined shortly into concrete set algebras based on appropriate 
Kripke systems, defined via pairs of theories, in increasing extensions (dimensions),
are actually homomorphisms. In fact, they have to do with the preservation of 
the operations of implication and universal quantification. The two claims use lemma \ref{t3}.

\begin{athm}{Claim 1} Let $n\in \omega$. If $((\Delta, \Gamma), (T,F))$ 
is a matched pair of saturated theories in $\Sg^{\A_n}X_1$ and $\Sg^{\A_n}X_2$, then the following hold.
For any $a,b\in \Sg^{\A_n}X_1$ if $a\rightarrow b\notin \Delta$, then there is a matched pair $((\Delta',\Gamma'), (T', F'))$ of saturated theories in
$\Sg^{\A_{n+1}}X_1$ and $\Sg^{\A_{n+1}}X_2,$ respectively, such that $\Delta\subseteq \Delta'$, $T\subseteq T',$  
$a\in \Delta'$ and $b\notin \Delta'$.
\end{athm}
\begin{demo}{Proof of Claim 1}  Since $a\rightarrow b\notin \Delta,$
 we have $(\Delta\cup\{a\}, b)$ is consistent in $\Sg^{\A_n}X_1$.
Then by lemma \ref{t1}, it can be extended to a complete theory $(\Delta', T')$ in $\Sg^{\A_n}X_1$.
Take $$\Phi=\Delta'\cap \Sg^{\A_n}X_1\cap \Sg^{\A_n}X_2,$$ 
and $$\Psi=T'\cap \Sg^{\A_n}X_1\cap \Sg^{\A_n}X_2.$$
Then $(\Phi, \Psi)$ is complete in $\Sg^{\A_n}X_1\cap \Sg^{\A_n}X_2.$
We shall now show that $(T\cup \Phi, \Psi)$ is consistent in $\Sg^{\A_n}X_2$. If not, then there is $\theta\in T$, 
$\phi\in \Phi$ and $\psi\in \Psi$ such that
$\theta\land \phi\leq \psi$. So $\theta\leq \phi\rightarrow \psi$. Since $T$ is saturated, we get that $\phi\rightarrow \psi$ is in $T$.
Now $\phi\rightarrow \psi\in \Delta\cap \Sg^{\A_n}X_1\cap \Sg^{\A_n}X_2\subseteq \Delta'\cap \Sg^{\A}X_1\cap \Sg^{\A}X_2=\Phi$.
Since $\phi\in \Phi$ and $\phi\rightarrow \psi\in \Phi,$ we get that $\psi\in \Phi\cap \Psi$. But this means that 
$(\Phi, \Psi)$ is inconsistent which is impossible. Thus $(T\cup \Phi, \Psi)$ is consistent.
Now the pair $((\Delta', T') (T\cup \Phi, \Psi))$ satisfy the conditions of lemma \ref{t3}. 
Hence this pair can be extended to a matched pair of 
saturated theories in $\Sg^{\A_{n+1}}X_1$ and $\Sg^{\A_{n+1}}X_2$. 
This pair is as required by the conclusion of lemma \ref{t3}.
\end{demo}

\begin{athm}{Claim 2} Let $n\in \omega$. If $((\Delta, \Gamma), (T,F))$ is a matched pair of saturated theories in $\Sg^{\A_n}X_1$ and $\Sg^{\A_n}X_2$, 
then the following hold.
For $x\in \Sg^{\A_n}X_1$ and $j\in I_n=dim\A_n$,  if ${\sf q}_{j}x\notin  \Delta$, then there is a matched pair $((\Delta',\Gamma'), (T', F'))$ 
of saturated theories in
$\Sg^{\A_{n+2}}X_1$ and $\Sg^{\A_{n+2}}X_2$ respectively,  $u\in I_{n+2}$
 such that $\Delta\subseteq \Delta'$, $T\subseteq T'$ and ${\sf s}_u^j x\notin \Delta'$. 
\end{athm}
\begin{demo}{Proof}
Assume that $x\in \Sg^{\A_n}X_1$ and $j\in I_n$ such that ${\sf q}_{j}x\notin  \Sg^{\A_n}X_1$. Then there exists 
$u\in I_{n+1}\sim I_n$ such that $(\Delta, {\sf s}_u^j x)$ is consistent in $\Sg^{\A_{n+1}}X_1$.
So $(\Delta, {\sf s}_u^jx)$ can be extended to a complete theory $(\Delta', T')$ in $\Sg^{\A_{n+1}}X_1$.
Take $$\Phi=\Delta'\cap \Sg^{\A_{n+1}}X_1\cap \Sg^{\A_{n+1}}X_2,$$
and
$$\Psi=T'\cap \Sg^{\A_{n+1}}X_1\cap \Sg^{\A_{n+1}}X_2.$$
Then $(\Phi,\Psi)$ is complete in $\Sg^{\A_{n+1}}X_1\cap \Sg^{\A_{n+1}}X_2$.
We shall show that $(T\cup \Phi,\Psi)$ is consistent in $\Sg^{\A_{n+1}}X_2$. If not, then there exist $\theta\in T,$ 
$\phi\in \Phi$ and $\psi\in \Psi,$ such that
$\theta\land \phi\leq \psi$. Hence, $\theta\leq \phi\rightarrow \psi$. Now
$$\theta={\sf q}_j(\theta)\leq  {\sf q}_{j}(\phi\rightarrow \psi).$$ 
Since $(T,F)$ is saturated in $\Sg^{\A_n}X_2,$ it thus follows that 
$${\sf q}_{j}(\phi\rightarrow \psi) \in T\cap \Sg^{\A_n}X_1\cap \Sg^{\A_n}X_2=\Delta\cap \Sg^{\A_n}X_1\cap \Sg^{\A_n}X_2.$$
So ${\sf q}_{j}(\phi\rightarrow \psi)\in \Delta'$ and consequently we get ${\sf q}_{j}(\phi\rightarrow \psi)\in \Phi$. Also, we have, $\phi\in \Phi$.
But $(\Phi, \Psi)$ is complete, we get $\psi\in \Phi$ and this contradicts that $\psi\in \Psi$.
Now the pair $((\Delta', \Gamma'), (T\cup \Phi, \Psi))$ satisfies the hypothesis of lemma \ref{t3} applied to 
$\Sg^{\A_{n+1}}X_1, \Sg^{\A_{n+1}}X_2$. 
The required now follows from the concusion of lemma \ref{t3}.
\end{demo}

Now that we have proved our claims, we go on with the proof.
We prove the theorem when $G$ is a strongly rich semigroup, because in this case we deal with relativized semantics,
and during the proof we state the necessary modifications for the case when $G$ is the semigroup of all transformations. 
Let $$K=\{((\Delta, \Gamma), (T,F)): \exists n\in \omega \text { such that } (\Delta, \Gamma), (T,F)$$
$$\text { is a a matched pair of saturated theories in }
\Sg^{\A_n}X_1, \Sg^{\A_n}X_2\}.$$
We have $((\Delta_0, \Gamma_0)$, $(\Theta_0, \Gamma_0^*))$ is a matched pair but the theories are not saturated. But by lemma \ref{t3}
there are $T_1=(\Delta_{\omega}, \Gamma_{\omega})$, 
$T_2=(\Theta_{\omega}, \Gamma_{\omega}^*)$ extending 
$(\Delta_0, \Gamma_0)$, $(\Theta_0, \Gamma_0^*)$, such that $T_1$ and $T_2$ are saturated in $\Sg^{\A_1}X_1$ and $\Sg^{\A_1}X_2,$ respectively.
Let $k_0=((\Delta_{\omega}, \Gamma_{\omega}), (\Theta_{\omega}, \Gamma_{\omega}^*)).$ Then $k_0\in K.$

If $i=((\Delta, \Gamma), (T,F))$ is a matched pair of saturated theories in $\Sg^{\A_n}X_1$ and $\Sg^{\A_n}X_2$, let $M_i=dim \A_n$, 
where $n$ is the least such number, so $n$ is unique to $i$.
Before going on we introduce a piece of notation. For a set $M$ and a sequence $p\in {}^{\alpha}M$, $^{\alpha}M^{(p)}$ is the following set
$$\{s\in {}^{\alpha}M: |\{i\in \alpha: s_i\neq p_i\}|<\omega\}.$$
Let $${\mathfrak{K}}=(K, \leq, \{M_i\}, \{V_i\})_{i\in \mathfrak{K}}$$
where $V_i=\bigcup_{p\in G_n}{}^{\alpha}M_i^{(p)}$, and $G_n$ is the strongly rich semigroup determining the similarity type of $\A_n$, with $n$ 
the least number such $i$ is a saturated matched pair in $\A_n$.
The order $\leq $ is defined as follows:
If $i_1=((\Delta_1, \Gamma_1)), (T_1, F_1))$ and $i_2=((\Delta_2, \Gamma_2), (T_2,F_2))$ are in $\mathfrak{K}$, then define
$$i_1\leq i_2\Longleftrightarrow  M_{i_1}\subseteq M_{i_2}, \Delta_1\subseteq \Delta_2, T_1\subseteq T_2.$$
This is, indeed as easily checked,  a preorder on $K$.

We  define two maps on $\A_1=\Sg^{\A}X_1$ and $\A_2=\Sg^{\A}X_2$ respectively, then those will be pasted using the freeness of $\A$
to give the required single homomorphism, by noticing that they agree on the common part, that is on $\Sg^{\A}(X_1\cap X_2).$

Set $\psi_1: \Sg^{\A}X_1\to \mathfrak{F}_{\mathfrak K}$ by
$\psi_1(p)=(f_k)$ such that if $k=((\Delta, \Gamma), (T,F))\in K$ is a matched pair of saturated theories in 
$\Sg^{\A_n}X_1$ and $\Sg^{\A_n}X_2$,
and $M_k=dim \A_n$, then for $x\in V_k=\bigcup_{p\in G_n}{}^{\alpha}M_k^{(p)}$,
$$f_k(x)=1\Longleftrightarrow {\sf s}_{x\cup (Id_{M_k\sim \alpha)}}^{\A_n}p\in \Delta\cup T.$$
To avoid tiresome notation, we shall denote the map $x\cup Id_{M_k\sim \alpha}$ simply by $\bar{x}$ when $M_k$ is clear from context.
It is easily verifiable that $\bar{x}$ is in the semigroup determining the similarity type of $\A_n$ hence the map is well defined.
More concisely, we we write $$f_k(x)=1\Longleftrightarrow {\sf s}_{\bar{x}}^{\A_n}p\in \Delta\cup T.$$
The map $\psi_2:\Sg^{\A}X_2\to \mathfrak{F}_{\mathfrak K}$ is defined in exactly the same way.
Since the theories are matched pairs, $\psi_1$ and $\psi_2$ agree on the common part, i.e. on $\Sg^{\A}(X_1\cap X_2).$
Here we also make the tacit assumption that if $k\leq k'$ then $V_k\subseteq V_{k'}$ 
via the embedding $\tau\mapsto \tau\cup Id$.

When $G$ is the semigroup of all transformations, with no restrictions on cardinalities, 
we need not relativize since $\bar{\tau}$ is in the big semigroup. In more detail, 
in this case, we take for $k=((\Delta,\Gamma), (T,F))$ a matched pair of saturated theories in 
$\Sg^{\A_n}X_1,\Sg^{\A_n}X_2$,
$M_k=dim\A_n$ and $V_k={}^{\alpha}M_k$
and for $x\in {}^{\alpha}M_k$, we set 
$$f_k(x)=1\Longleftrightarrow {\sf s}_{x\cup (Id_{M_k\sim \alpha)}}^{\A_n}p\in \Delta\cup T.$$

Before proving that $\psi$ is a homomorphism, we show that 
$$k_0=((\Delta_{\omega},\Gamma_{\omega}), (\Theta_{\omega}, \Gamma^*_{\omega}))$$ 
is as desired. Let $x\in V_{k_0}$ be the identity map. Let $p\in \Delta_0\cup \Theta_0$, then 
${\sf s}_xp=p\in \Delta_{\omega}\cup \Theta_{\omega},$ 
and so if $\psi(p)=(f_k)$ then $f_{k_0}(x)=1$.
On the other hand if $p\in \Gamma_0^*$, then $p\notin \Delta_{\omega}\cup \Theta_{\omega}$, and so $f_{k_0}(x)=0$.
Then the union $\psi$ of $\psi_1$ and $\psi_2$, $k_0$ and $Id$ are as required, modulo proving that 
$\psi$ is a homomorphism from $\A$, to the set algebra based on the above defined
Kripke system, which we proceed to show. We start by $\psi_1$.
Abusing notation, we denote $\psi_1$ by $\psi$, and we write a matched pair in $\A_n$ 
instead of a matched pair of saturated theories in $\Sg^{\A_n}X_1$,
$\Sg^{\A_n}X_2$, since $X_1$ and $X_2$ are fixed. The proof that the postulated map is a homomorphism is 
similar to the proof in \cite{Hung} baring in mind that it is far from being identical because 
cylindrifiers and their duals are only finite.
\begin{enumroman}
\item We prove that $\psi$ preserves $\land$. Let $p,q\in A$. Assume that
$\psi(p)=(f_k)$ and $\psi(q)=(g_k)$. Then $\psi(p)\land \psi(q)=(f_k\land g_k)$.
We now compute $\psi(p\land q)=(h_k)$
Assume that  $x\in V_k$, where $k=((\Delta,\Gamma), (T, F))$ is a matched pair in $\A_n$ and  $M_k=dim\A_n$.
Then
$$h_k(x)=1\Longleftrightarrow {\sf s}_{\bar{x}}^{\A_n}(p\land q)\in \Delta\cup T$$
$$\Longleftrightarrow {\sf s}_{\bar{x}}^{\A_n}p\land {\sf s}_{\bar{x}}^{\A_n}q\in \Delta\cup T$$
$$\Longleftrightarrow {\sf s}_{\bar{x}}^{\A_n}p\in T\cup \Delta\text { and }{\sf s}_{\bar{x}}^{\A_n}q\in \Delta\cup T$$
$$\Longleftrightarrow f_k(x)=1 \text { and } g_k(x)=1$$
$$\Longleftrightarrow (f_k\land g_k)(x)=1$$
$$\Longleftrightarrow (\psi(p)\land \psi(q))(x)=1.$$

\item $\psi$ preserves $\rightarrow.$ (Here we use Claim 1).  Let $p,q\in A$. Let $\psi(p)=(f_k)$ and $\psi(q)=(g_k)$.
Let $\psi(p\rightarrow q)=(h_k)$ and $\psi(p)\rightarrow \psi(q)=(h'_k)$.
We shall prove that for any $k\in \mathfrak{K}$ and any $x\in V_k$, we have
$$h_k(x)=1\Longleftrightarrow h'_k(x)=1.$$
Let $x\in V_k$. Then $k=((\Delta,\Gamma),(T,F))$ is a matched pair in $\A_n$ and $M_k=dim\A_n$.
Assume that $h_k(x)=1$. Then we have $${\sf s}_{\bar{x}}^{\A_n}(p\rightarrow q)\in \Delta\cup T,$$
from which we get that $$(*) \ \ \ {\sf s}_{\bar{x}}^{\A_n}p\rightarrow {\sf s}_{\bar{x}}^{\A_n}q\in \Delta\cup T.$$
Let $k'\in K$ such that $k\leq k'$. Then $k'=((\Delta', \Gamma'), (T', F'))$ is a matched pair in $\A_m$ with $m\geq n$.
Assume that $f_{k'}(x)=1$. Then, by definition we have (**)
$${\sf s}_{\bar{x}}^{\A_m}p\in \Delta'\cup T'.$$
But $\A_m$ is a dilation of $\A_n$ and so 
$${\sf s}_{\bar{x}}^{\A_m}p={\sf s}_{\bar{x}}^{\A_n}p\text { and } {\sf s}_{\bar{x}}^{\A_m}q={\sf s}_{\bar{x}}^{\A_n}q.$$
From (*) we get that, 
$$ {\sf s}_{\bar{x}}^{\A_m}p\rightarrow {\sf s}_{\bar{x}}^{\A_m}q\in \Delta'\cup T'.$$
But, on the other hand,  from (**), we have ${\sf s}_{\bar{x}}^{\A_m}q\in \Delta'\cup T',$
so $$f_{k'}(x)=1\Longrightarrow g_{k'}(x)=1.$$
That is to say, we have  $h_{k'}(x)=1$.
Conversely, assume that $h_k(x)\neq 1,$
then
$$ {\sf s}_{\bar{x}}^{\A_n}p\rightarrow {\sf s}_{\bar{x}}^{\A_n}q\notin \Delta\cup T,$$
and consequently
$$ {\sf s}_{\bar{x}}^{\A_n}p\rightarrow {\sf s}_{\bar{x}}^{\A_n}q\notin \Delta.$$
From Claim 1, we get that there exists a matched pair $k'=((\Delta',\Gamma')((T',F'))$ in $\A_{n+2},$ such that
$$ {\sf s}_{\bar{x}}^{\A_{n+2}}p\in \Delta'\text { and } {\sf s}_{\bar{x}}^{\A_{n+2}}q\notin \Delta'.$$
We claim that ${\sf s}_{\bar{x}}^{\A_{n+2}}q\notin T'$, for otherwise, if it is in $T'$, then we would get that
$${\sf s}_{\bar{x}}^{\A_{n+2}}q\in \Sg^{\A_{n+2}}X_1\cap \Sg^{\A_{n+2}}X_2.$$
But $$(\Delta'\cap T'\cap \Sg^{\A_{n+2}}X_1\cap \Sg^{\A_{n+2}}X_2, \Gamma'\cap F'\cap\Sg^{\A_{n+2}}X_1\cap \Sg^{\A_{n+2}}X_2)$$ is complete 
in $\Sg^{\A_{n+2}}X_1\cap \Sg^{\A_{n+2}}X_2,$
and ${\sf s}_{\bar{x}}^{\A_{n+2}}q\notin \Delta'\cap T'$, hence it must be the case that 
 $${\sf s}_{\bar{x}}^{\A_{n+2}}q\in \Gamma'\cap F'.$$
In particular, we have
$${\sf s}_{\bar{x}}^{\A_{n+2}}q\in F',$$
which contradicts the consistency of $(T', F'),$ since by assumption ${\sf s}_x^{\A_{n+2}}q\in T'$.
Now we have
$${\sf s}_{\bar{x}}^{\A_{n+2}}q\notin \Delta'\cup T',$$
and
$${\sf s}_{\bar{x}}^{\A_{n+2}}p\in \Delta'\cup T'.$$
Since $\Delta'\cup T'$ extends $\Delta\cup T$, we get that $h_k'(x)\neq 1$.
\item $\psi$ preserves substitutions. Let $p\in \A$. Let $\sigma\in {}G$.
Assume that $\psi(p)=(f_k)$ and $\psi({\sf s}_{\sigma}p)=(g_k).$ 
Assume that $M_k=\dim\A_n$ where $k=((\Delta,\Gamma),(T,F))$ is a matched pair 
in $\A_n$.
Then, for $x\in V_k$, we have 
$$g_k(x)=1\Longleftrightarrow {\sf s}_{\bar{x}}^{\A_n}{\sf s}_{\sigma}^{\A}p\in \Delta\cup T$$
$$\Longleftrightarrow  {\sf s}_{\bar{x}}^{\A_n}{\sf s}_{\bar{\sigma}}^{\A_n}p\in \Delta\cup T$$
$$\Longleftrightarrow  {\sf s}_{\bar{x}\circ {\bar{\sigma}}}^{\A_n}p\in \Delta\cup T$$
$$\Longleftrightarrow {\sf s}_{\overline{x\circ \sigma}}^{\A_n}p\in \Delta\cup T$$
$$\Longleftrightarrow f_k(x\circ \sigma)=1.$$

\item $\psi$ preserves cylindrifications. Let $p\in A.$ 
Assume that $m\in I$  and assume that  $\psi({\sf c}_{m}p)=(f_k)$ and ${\sf c}_m\psi(p)=(g_k)$. 
Assume that $k=((\Delta,\Gamma),(T,F))$ is a matched pair in 
$\A_n$ and that $M_k=dim\A_n$.  Let $x\in V_k$. Then 
$$f_k(x)=1\Longleftrightarrow {\sf s}_{\bar{x}}^{\A_n}{\sf c}_{m}p\in \Delta\cup T.$$
We can assume that  $${\sf s}_{\bar{x}}^{\A_n}{\sf c}_{m}p\in \Delta.$$
For if not, that is if  $${\sf s}_{\bar{x}}^{\A_n}{\sf c}_{m}p\notin \Delta\text { and }  {\sf s}_{\bar{x}}^{\A_n}{\sf c}_{(m)}p\in T,$$
then  $${\sf s}_{\bar{x}}^{\A_n}{\sf c}_{m}p\in \Sg^{\A_n}X_1\cap \Sg^{\A_n}X_2, $$
but  $$(\Delta\cap T\cap \Sg^{\A_n}X_1\cap \Sg^{\A_n}X_2, \Gamma\cap F\cap \Sg^{\A_n}X_1\cap \Sg^{\A_n}X_2)$$ is complete in 
$\Sg^{\A_n}X_1\cap \Sg^{\A_n}X_2$, 
and $${\sf s}_{\bar{x}}^{\A_n}{\sf c}_{m}p\notin \Delta\cap T,$$ it must be the case that
$${\sf s}_{\bar{x}}^{\A_n}{\sf c}_{m}p\in \Gamma\cap F.$$
In particular, $${\sf s}_{\bar{x}}^{\A_n}{\sf c}_{m}p\in F.$$
But this contradicts the consistency of $(T,F)$.

Assuming that ${\sf s}_x{\sf c}_mp\in \Delta,$ we proceed as follows.
Let  $$\lambda\in \{\eta\in I_n: x^{-1}\{\eta\}=\eta\}\sim \Delta p.$$
Let $$\tau=x\upharpoonright I_n\sim\{m, \lambda\}\cup \{(m,\lambda)(\lambda, m)\}.$$
Then, by item (5) in theorem \ref{axioms}, we have
$${\sf c}_{\lambda}{\sf s}^{\A_n}_{\bar{\tau}}p={\sf s}_{\bar{\tau}}^{\A_n}{\sf c}_{m}p={\sf s}_{\bar{x}}^{\A_n}{\sf c}_mp\in \Delta.$$
We introduce a piece of helpful notation. For a function $f$, let $f(m\to u)$ is the function that agrees with $f$ except at $m$ 
where its value is $u$.
Since $\Delta$ is saturated, there exists $u\notin \Delta x$ such that ${\sf s}_u^{\lambda}{\sf s}_xp\in \Delta$, and so ${\sf s}_{(x(m\to u))}
p\in \Delta$. 
This implies that $x\in {\sf c}_mf(p)$ and so $g_k(x)=1$. Conversely, assume that $g_k(x)=1$ with 
$k=((\Gamma,\Delta))$, $(T,F))$ a matched pair in $\A_n$.
Let $y\in V_k$ such that $y\equiv_m x$ and  $\psi(p)y=1$. Then ${\sf s}_{\bar{y}}p\in \Delta\cup T$. 
Hence ${\sf s}_{\bar{y}}{\sf c}_mp\in \Delta\cup T$ and so ${\sf s}_{\bar{x}}{\sf c}_mp\in \Delta\cup T$, 
thus  $f_k(x)=1$ and we are done.

\item $\psi$ preserves universal quantifiers. (Here we use Claim 2). Let $p\in A$ and $m\in I$. 
Let $\psi(p)=(f_k)$, ${\sf q}_{m}\psi(p)=(g_k)$ and $\psi({\sf q}_{m}p)=(h_k).$
Assume that $h_k(x)=1$. We have $k=((\Delta,\Gamma), (T,F))$ is a matched pair in $\A_n$ and $x\in V_k$.
Then $${\sf s}_{\bar{x}}^{\A_n}{\sf q}_{m}p\in \Delta\cup T,$$
and so 
$${\sf s}_{\bar{y}}^{\A_n}{\sf q}_{m}p\in \Delta\cup T \text{ for all } y\in {}^IM_k, y\equiv_m x.$$ 
Let $k'\geq k$. Then $k'=((\Delta',\Gamma'), (T',F'))$ is a matched pair in $\A_l$ $l\geq n$, 
$\Delta\subseteq \Delta'$ and $T\subseteq T'.$ 
Since $p\geq {\sf q}_{m}p$ it follows that
$${\sf s}_{\bar{y}}^{\A_n}p\in \Delta'\cup T' \text{ for all } y\in {}^IM_k, y\equiv_mx.$$ 
Thus $g_k(x)=1$.
Now conversely, assume that $h_k(x)=0$, $k=((\Delta,\Gamma), (T,F))$ is a matched pair in $\A_n,$
then, we have
$${\sf s}_{\bar{x}}^{\A_n}{\sf q}_{m}p\notin \Delta\cup T,$$
and so $${\sf s}_{\bar{x}}^{\A_n}{\sf q}_{m}p\notin \Delta.$$
Let  $$\lambda\in \{\eta\in I_n: x^{-1}\{\eta\}=\eta\}\sim \Delta p.$$
Let $$\tau=x\upharpoonright I_n\sim\{m, \lambda\}\cup \{(m,\lambda)(\lambda, m)\}.$$
Then, like in the existential case, using polyadic axioms, 
we get
$${\sf q}_{\lambda}{\sf s}_{\tau}p={\sf s}_{\tau}{\sf q}_{m}p={\sf s}_{x}{\sf q}_mp\notin \Delta$$
Then there exists $u$ such that ${\sf s}_u^{\lambda}{\sf s}_xp\notin \Delta.$
So ${\sf s}_u^{\lambda}{\sf s}_xp\notin T$, for if it is, then by the previous reasoning since it is an element of 
$\Sg^{\A_{n+2}}X_1\cap \Sg^{\A_{n+2}}X_2$ and by completeness
of $(\Delta\cap T, \Gamma\cap F)$ we would reach a contradiction.
The we get that ${\sf s}_{(x(m\to u))}p\notin \Delta\cup T$
which means that $g_k(x)=0,$
and we are done. 
\end{enumroman}
\end{demo}

\begin{theorem}\label{main2}
Let $G$ be the semigroup of finite transformations on an infinite set 
$\alpha$ and let $\delta$ be a cardinal $>0$. Let $\rho\in {}^{\delta}\wp(\alpha)$ be such that
$\alpha\sim \rho(i)$ is infinite for every 
$i\in \delta$. Let $\A$ be the free  $G$ algebra generated by $X$ restristed by $\rho$;
 that is $\A=\Fr_{\delta}^{\rho}GPHA_{\alpha},$
and suppose that $X=X_1\cup X_2$.
Let $(\Delta_0, \Gamma_0)$, $(\Theta_0, \Gamma_0^*)$ be two consistent theories in $\Sg^{\A}X_1$ and $\Sg^{\A}X_2,$ respectively.
Assume that $\Gamma_0\subseteq \Sg^{\A}(X_1\cap X_2)$ and $\Gamma_0\subseteq \Gamma_0^*$.
Assume, further, that  
$(\Delta_0\cap \Theta_0\cap \Sg^{\A}X_1\cap \Sg^{\A}X_2, \Gamma_0)$ is complete in $\Sg^{\A}X_1\cap \Sg^{\A}X_2$. 
Then there exist a Kripke system $\mathfrak{K}=(K,\leq \{X_k\}_{k\in K}\{V_k\}_{k\in K}),$ a homomorphism $\psi:\A\to \mathfrak{F}_K,$
$k_0\in K$, and $x\in V_{k_0}$,  such that for all $p\in \Delta_0\cup \Theta_0$ if $\psi(p)=(f_k)$, then $f_{k_0}(x)=1$
and for all $p\in \Gamma_0^*$ if $\psi(p)=(f_k)$, then $f_{k_0}(x)=0$.
\end{theorem}
\begin{demo}{Proof} We state the modifications in the above proof of theorem \ref{main}.
Form the sequence of minimal dilations $(\A_n:n\in \omega)$ built on the sequence $(K_n:n\in \omega)$, with $|K_n|=\beta$,
$\beta=|I\sim \alpha|=max(|A|, \alpha)$ with $I$ is a superset of $\alpha.$ 
If $i=((\Delta, \Gamma), (T,F))$ is a matched pair of saturated theories in $\Sg^{\A_n}X_1$ and $\Sg^{\A_n}X_2$, let $M_i=dim \A_n$, 
where $n$ is the least such number, so $n$ is unique to $i$. Define $K$ as in in the proof of theorem \ref{main}, that is, 
let $$K=\{((\Delta, \Gamma), (T,F)): \exists n\in \omega \text { such that } (\Delta, \Gamma), (T,F)$$
$$\text { is a a matched pair of saturated theories in }
\Sg^{\A_n}X_1, \Sg^{\A_n}X_2\}.$$
Let $${\mathfrak{K}}=(K, \leq, \{M_i\}, \{V_i\})_{i\in \mathfrak{K}},$$
where now $V_i={}^{\alpha}M_i^{(Id)}=\{s\in {}^{\alpha}M: |\{i\in \alpha: s_i\neq i\}|<\omega\},$ and the order $\leq $ is defined by:
If $i_1=((\Delta_1, \Gamma_1)), (T_1, F_1))$ and $i_2=((\Delta_2, \Gamma_2), (T_2,F_2))$ are in $\mathfrak{K}$, then
$$i_1\leq i_2\Longleftrightarrow  M_{i_1}\subseteq M_{i_2}, \Delta_1\subseteq \Delta_2, T_1\subseteq T_2.$$
This is a preorder on $K$.
Set $\psi_1: \Sg^{\A}X_1\to \mathfrak{F}_{\mathfrak K}$ by
$\psi_1(p)=(f_k)$ such that if $k=((\Delta, \Gamma), (T,F))\in \mathfrak{K}$ 
is a matched pair of saturated theories in $\Sg^{\A_n}X_1$ and $\Sg^{\A_n}X_2$,
and $M_k=dim \A_n$, then for $x\in V_k={}^{\alpha}M_k^{(Id)}$,
$$f_k(x)=1\Longleftrightarrow {\sf s}_{x\cup (Id_{M_k\sim \alpha)}}^{\A_n}p\in \Delta\cup T.$$
Define $\psi_2$ analogously. The rest of the proof is identical to the previous one.
\end{demo}
It is known that the condition $\Gamma\subseteq \Gamma^*$ cannot be omitted.
On the other hand, to prove our completeness theorem, we need the following weaker 
version of theorem \ref{main}, with a slight modification in the proof, which is still a step-by-step technique,
though, we do not `zig-zag'.  

%We do not know whether the next lemma holds for $GPHAE_{\alpha}$.

\begin{lemma}\label{rep}
Let $\A\in GPHA_{\alpha}$. Let $(\Delta_0, \Gamma_0)$ be consistent. 
Suppose that $I$ is a set such that $\alpha\subseteq I$  and $|I\sim \alpha|=max (|A|,|\alpha|)$. 
\begin{enumarab}
\item Then there exists a dilation $\B\in GPHA_I$ of $\A$, and theory $T=(\Delta_{\omega}, \Gamma_{\omega})$, 
extending  $(\Delta_0, \Gamma_0)$, such that $T$  is consistent and saturated in $\B$. 
\item There exists $\mathfrak{K}=(K,\leq \{X_k\}_{k\in K}\{V_k\}_{k\in K}),$ a homomorphism $\psi:\A\to \mathfrak{F}_K,$
$k_0\in K$, and $x\in V_{k_0}$,  such that for all $p\in \Delta_0$ if $\psi(p)=(f_k)$, then $f_{k_0}(x)=1$ and for all $p\in \Gamma_0$
if $\psi(p)=(g_k)$, then $g_{k_0}(x)=0.$
\end{enumarab}
\end{lemma}
\begin{demo}{Proof} We deal only with the case when $G$ is strongly rich. The other cases can be dealt with in a similar manner 
by undergoing the obvious modifications, as indicated above. 
As opposed to theorem \ref{main}, we use theories rather than pairs of theories, since we are not dealing with two subalgebras simultaneously.
(i) follows from \ref{t2}. Now we prove (ii). The proof is a simpler version of the proof of \ref{main}.
Let $I$ be  a set such that 
$\beta=|I\sim \alpha|=max(|A|, |\alpha|).$
Let $(K_n:n\in \omega)$ be a family of pairwise disjoint sets such that $|K_n|=\beta.$
Define a sequence of algebras
$\A=\A_0\subseteq \A_1\subseteq \A_2\subseteq \A_2\ldots \subseteq \A_n\ldots$
such that
$\A_{n+1}$ is a minimal dilation of $\A_n$ and $dim(\A_{n+1})=\dim\A_n\cup K_n$.
We denote $dim(\A_n)$ by $I_n$ for $n\geq 1$. 
If $(\Delta, \Gamma)$ 
is saturated in $\A_n$ then the following analogues of Claims 1 and 2 in theorem \ref{main} hold:
For any $a,b\in \A_n$ if $a\rightarrow b\notin \Delta$, then there is a saturated  theory $(\Delta',\Gamma')$ in $\A_{n+1}$ 
such that $\Delta\subseteq \Delta'$ $a\in \Delta'$ and $b\notin \Delta'$.
If $(\Delta, \Gamma)$ is saturated in $\A_n$ then for all  $x\in \A_n$ and $j\in I_n$,  if ${\sf q}_{j}x\notin  \Delta,$ 
then there $(\Delta',\Gamma')$ 
of saturated theories in
$\A_{n+2}$, $u\in I_{n+2}$
such that $\Delta\subseteq \Delta'$, and ${\sf s}_j^u x\notin \Delta'$. Now let
$$K=\{(\Delta, \Gamma): \exists n\in \omega \text { such that } (\Delta,\Gamma) \text { is saturated in }\A_n.\}$$
If $i=(\Delta, \Gamma)$ is a saturated theory in $\A_n$, let $M_i=dim \A_n$, 
where $n$ is the least such number, so $n$ is unique to $i$.
If $i_1=(\Delta_1, \Gamma_1)$ and $i_2=(\Delta_2, \Gamma_2)$ are in $K$, then set
$$i_1\leq i_2\Longleftrightarrow  M_{i_1}\subseteq M_{i_2}, \Delta_1\subseteq \Delta_2. $$
This is a preorder on $K$; define the kripke system ${\mathfrak K}$ based on the set of worlds $K$ as before.
Set $\psi: \A\to \mathfrak{F}_{\mathfrak K}$ by
$\psi_1(p)=(f_k)$ such that if $k=(\Delta, \Gamma)\in \mathfrak{K}$ is  saturated in $\A_n$,
and $M_k=dim \A_n$, then for $x\in V_k=\bigcup_{p\in G_n}{}^{\alpha}M_k^{(p)}$,
$$f_k(x)=1\Longleftrightarrow {\sf s}_{x\cup (Id_{M_k\sim \alpha)}}^{\A_n}p\in \Delta.$$
Let  $k_0=(\Delta_{\omega}, \Gamma_{\omega})$ be defined as a complete saturated extension of $(\Delta_0, \Gamma_0)$
in $\A_1$, then $\psi,$ $k_0$ and $Id$ are as desired. 
The analogues of Claims 1 and 2 in theorem \ref{main} 
are used to show that $\psi$ so defined preserves implication and 
universal quantifiers.
\end{demo}

\section{Presence of diagonal elements}

All results, in Part 1, up to the previous theorem,  are proved in the absence of diagonal elements.
Now lets see how far we can go if we have diagonal elements. 
Considering diagonal elements, as we shall see, turn out to be problematic but not hopeless.

Our representation theorem has to respect diagonal elements, 
and this seems to be an impossible task with the presence of infinitary substitutions, 
unless we make a compromise that is, from our point of view, acceptable.
The interaction of substitutions based on infinitary transformations, 
together with the existence of diagonal elements tends to make matters `blow up'; indeed this even happens in the classical case,
when the class of (ordinary) set algebras ceases to be closed under ultraproducts.
The natural thing to do is to avoid those infinitary substitutions at the start, while finding the interpolant possibly using such substitutions.
We shall also show that in some cases the interpolant has to use infinitary substitutions, even if the original implication uses only finite transformations.

So for an algebra $\A$, we let $\Rd\A$ denote its reduct when we discard infinitary substitutions. $\Rd\A$ satisfies 
cylindric algebra axioms.

\begin{theorem}\label{main4}
Let $G$ be the semigroup of finite transformations on an infinite set 
$\alpha$ and let $\delta$ be a cardinal $>0$. Let $\rho\in {}^{\delta}\wp(\alpha)$ be such that
$\alpha\sim \rho(i)$ is infinite for every 
$i\in \delta$. Let $\A$ be the free  $G$ algebra with equality generated by $X$ restristed by $\rho$;
 that is $\A=\Fr_{\delta}^{\rho}GPHAE_{\alpha},$
and suppose that $X=X_1\cup X_2$.
Let $(\Delta_0, \Gamma_0)$, $(\Theta_0, \Gamma_0^*)$ be two consistent theories in $\Sg^{\A}X_1$ and $\Sg^{\A}X_2,$ respectively.
Assume that $\Gamma_0\subseteq \Sg^{\A}(X_1\cap X_2)$ and $\Gamma_0\subseteq \Gamma_0^*$.
Assume, further, that  
$(\Delta_0\cap \Theta_0\cap \Sg^{\A}X_1\cap \Sg^{\A}X_2, \Gamma_0)$ is complete in $\Sg^{\A}X_1\cap \Sg^{\A}X_2$. 
Then there exist a Kripke system $\mathfrak{K}=(K,\leq \{X_k\}_{k\in K}\{V_k\}_{k\in K}),$ a homomorphism $\psi:\A\to \mathfrak{F}_K,$
$k_0\in K$, and $x\in V_{k_0}$,  such that for all $p\in \Delta_0\cup \Theta_0$ if $\psi(p)=(f_k)$, then $f_{k_0}(x)=1$
and for all $p\in \Gamma_0^*$ if $\psi(p)=(f_k)$, then $f_{k_0}(x)=0$.
\end{theorem}
\begin{demo}{Proof}
The first half of the proof is almost identical to that of  lemma \ref{main}. We highlight the main steps, 
for the convenience of the reader, except that we only deal with the case
when $G$ is strongly rich.
Assume, as usual, that $\alpha$, $G$, $\A$ and $X_1$, $X_2$, and everything else in the hypothesis are given.
Let $I$ be  a set such that  $\beta=|I\sim \alpha|=max(|A|, |\alpha|).$
Let $(K_n:n\in \omega)$ be a family of pairwise disjoint sets such that $|K_n|=\beta.$
Define a sequence of algebras
$\A=\A_0\subseteq \A_1\subseteq \A_2\subseteq \A_2\ldots \subseteq \A_n\ldots$
such that
$\A_{n+1}$ is a minimal dilation of $\A_n$ and $dim(\A_{n+1})=\dim\A_n\cup K_n$.We denote $dim(\A_n)$ by $I_n$ for $n\geq 1$. 
The proofs of Claims 1 and 2 in the proof of \ref{main} are the same.

Now we prove the theorem when $G$ is a strongly rich semigroup.
Let $$K=\{((\Delta, \Gamma), (T,F)): \exists n\in \omega \text { such that } (\Delta, \Gamma), (T,F)$$
$$\text { is a a matched pair of saturated theories in }
\Sg^{\Rd\A_n}X_1, \Sg^{\Rd\A_n}X_2\}.$$
We have $((\Delta_0, \Gamma_0)$, $(\Theta_0, \Gamma_0^*))$ is a matched pair but the theories are not saturated. But by lemma \ref{t3}
there are $T_1=(\Delta_{\omega}, \Gamma_{\omega})$, 
$T_2=(\Theta_{\omega}, \Gamma_{\omega}^*)$ extending 
$(\Delta_0, \Gamma_0)$, $(\Theta_0, \Gamma_0^*)$, such that $T_1$ and $T_2$ are saturated in $\Sg^{\Rd\A_1}X_1$ and $\Sg^{\Rd\A_1}X_2,$ 
respectively. Let $k_0=((\Delta_{\omega}, \Gamma_{\omega}), (\Theta_{\omega}, \Gamma_{\omega}^*)).$ Then $k_0\in K,$
and   $k_0$ will be the desired world and $x$ will be specified later; in fact $x$ will be the identity map on some specified 
domain.

If $i=((\Delta, \Gamma), (T,F))$ is a matched pair of saturated theories in $\Sg^{\Rd\A_n}X_1$ and $\Sg^{\Rd\A_n}X_2$, let $M_i=dim \A_n$, 
where $n$ is the least such number, so $n$ is unique to $i$.
Let $${\bf K}=(K, \leq, \{M_i\}, \{V_i\})_{i\in \mathfrak{K}},$$
where $V_i=\bigcup_{p\in G_n, p\text { a finitary transformation }}{}^{\alpha}M_i^{(p)}$ 
(here we are considering only substitutions that move only finitely many points), 
and $G_n$ 
is the strongly rich semigroup determining the similarity type of $\A_n$, with $n$ 
the least number such $i$ is a saturated matched pair in $\A_n$, and $\leq $ is defined as follows: 
If $i_1=((\Delta_1, \Gamma_1)), (T_1, F_1))$ and $i_2=((\Delta_2, \Gamma_2), (T_2,F_2))$ are in $\bold K$, then set
$$i_1\leq i_2\Longleftrightarrow  M_{i_1}\subseteq M_{i_2}, \Delta_1\subseteq \Delta_2, T_1\subseteq T_2.$$ 
We are not yet there, to preserve diagonal elements we have to factor out $\bold K$ 
by an infinite family equivalence relations, each defined on the dimension of $\A_n$, for some $n$, which will actually turn out to be 
a congruence in an exact sense. 
As usual, using freeness of $\A$, we will  define two maps on $\A_1=\Sg^{\Rd\A}X_1$ and $\A_2=\Sg^{\Rd\A}X_2$, respectively;
then those will be pasted 
to give the required single homomorphism.

Let $i=((\Delta, \Gamma), (T,F))$ be a matched pair of saturated theories in $\Sg^{\Rd\A_n}X_1$ and $\Sg^{\Rd\A_n}X_2$, let $M_i=dim \A_n$, 
where $n$ is the least such number, so $n$ is unique to $i$.
For $k,l\in dim\A_n=I_n$, set $k\sim_i l$ iff ${\sf d}_{kl}^{\A_n}\in \Delta\cup T$. This is well defined since $\Delta\cup T\subseteq \A_n$.
We omit the superscript $\A_n$.
These are infinitely many relations, one for each $i$, defined on $I_n$, with $n$ depending uniquely on $i$, 
we denote them uniformly by $\sim$ to 
avoid complicated unnecessary notation.
We hope that no confusion is likely to ensue. We claim that $\sim$ is an equivalence relation on $I_n.$
Indeed,  $\sim$ is reflexive because ${\sf d}_{ii}=1$ and symmetric 
because ${\sf d}_{ij}={\sf d}_{ji};$
finally $E$ is transitive because for  $k,l,u<\alpha$, with $l\notin \{k,u\}$, 
we have 
$${\sf d}_{kl}\cdot {\sf d}_{lu}\leq {\sf c}_l({\sf d}_{kl}\cdot {\sf d}_{lu})={\sf d}_{ku},$$
and we can assume that $T\cup \Delta$ is closed upwards.
For $\sigma,\tau \in V_k,$ define $\sigma\sim \tau$ iff $\sigma(i)\sim \tau(i)$ for all $i\in \alpha$.
Then clearly $\sigma$ is an equivalence relation on $V_k$. 

Let $W_k=V_k/\sim$, and $\mathfrak{K}=(K, \leq, M_k, W_k)_{k\in K}$, with $\leq$ defined on $K$ as above.
We write $h=[x]$ for $x\in V_k$ if $x(i)/\sim =h(i)$ for all $i\in \alpha$; of course $X$ may not be unique, but this will not matter.
Let $\F_{\mathfrak K}$ be the set algebra based on the new Kripke system ${\mathfrak K}$ obtained by factoring out $\bold K$.

Set $\psi_1: \Sg^{\Rd\A}X_1\to \mathfrak{F}_{\mathfrak K}$ by
$\psi_1(p)=(f_k)$ such that if $k=((\Delta, \Gamma), (T,F))\in K$ 
is a matched pair of saturated theories in $\Sg^{\Rd\A_n}X_1$ and $\Sg^{\Rd\A_n}X_2$,
and $M_k=dim \A_n$, with $n$ unique to $k$, then for $x\in W_k$
$$f_k([x])=1\Longleftrightarrow {\sf s}_{x\cup (Id_{M_k\sim \alpha)}}^{\A_n}p\in \Delta\cup T,$$
with $x\in V_k$ and $[x]\in W_k$ is define as above. 

To avoid cumbersome notation, we 
write ${\sf s}_{x}^{\A_n}p$, or even simply ${\sf s}_xp,$ for 
${\sf s}_{x\cup (Id_{M_k\sim \alpha)}}^{\A_n}p$.  No ambiguity should arise because the dimension $n$ will be clear from context.

We need to check that $\psi_1$ is well defined. 
It suffices to show that if $\sigma, \tau\in V_k$ if $\sigma \sim \tau$ and $p\in \A_n$,  
with $n$ unique to $k$, 
then $${\sf s}_{\tau}p\in \Delta\cup T\text { iff } {\sf s}_{\sigma}p\in \Delta\cup T.$$

This can be proved by induction on the cardinality of 
$J=\{i\in I_n: \sigma i\neq \tau i\}$, which is finite since we are only taking finite substitutions.
If $J$ is empty, the result is obvious. 
Otherwise assume that $k\in J$. We recall the following piece of notation.
For $\eta\in V_k$ and $k,l<\alpha$, write  
$\eta(k\mapsto l)$ for the $\eta'\in V$ that is the same as $\eta$ except
that $\eta'(k)=l.$ 
Now take any 
$$\lambda\in \{\eta\in I_n: \sigma^{-1}\{\eta\}= \tau^{-1}\{\eta\}=\{\eta\}\}\smallsetminus \Delta x.$$
This $\lambda$ exists, because $\sigma$ and $\tau$ are finite transformations and $\A_n$ is a dilation with enough spare dimensions.
We have by cylindric axioms (a)
$${\sf s}_{\sigma}x={\sf s}_{\sigma k}^{\lambda}{\sf s}_{\sigma (k\mapsto \lambda)}p.$$
We also have (b)
$${\sf s}_{\tau k}^{\lambda}({\sf d}_{\lambda, \sigma k}\land {\sf s}_{\sigma} p)
={\sf d}_{\tau k, \sigma k} {\sf s}_{\sigma} p,$$
and (c)
$${\sf s}_{\tau k}^{\lambda}({\sf d}_{\lambda, \sigma k}\land {\sf s}_{\sigma(k\mapsto \lambda)}p)$$
$$= {\sf d}_{\tau k,  \sigma k}\land {\sf s}_{\sigma(k\mapsto \tau k)}p.$$
and (d)
$${\sf d}_{\lambda, \sigma k}\land {\sf s}_{\sigma k}^{\lambda}{\sf s}_{{\sigma}(k\mapsto \lambda)}p=
{\sf d}_{\lambda, \sigma k}\land {\sf s}_{{\sigma}(k\mapsto \lambda)}p$$
Then by (b), (a), (d) and (c), we get,
$${\sf d}_{\tau k, \sigma k}\land {\sf s}_{\sigma} p= 
{\sf s}_{\tau k}^{\lambda}({\sf d}_{\lambda,\sigma k}\cdot {\sf s}_{\sigma}p)$$
$$={\sf s}_{\tau k}^{\lambda}({\sf d}_{\lambda, \sigma k}\land {\sf s}_{\sigma k}^{\lambda}
{\sf s}_{{\sigma}(k\mapsto \lambda)}p)$$
$$={\sf s}_{\tau k}^{\lambda}({\sf d}_{\lambda, \sigma k}\land {\sf s}_{{\sigma}(k\mapsto \lambda)}p)$$
$$= {\sf d}_{\tau k,  \sigma k}\land {\sf s}_{\sigma(k\mapsto \tau k)}p.$$
The conclusion follows from the induction hypothesis.
Now $\psi_1$ respects all quasipolyadic equality operations, that is finite substitutions (with the proof as before; 
recall that we only have finite substitutions since we are considering 
$\Sg^{\Rd\A}X_1$)  except possibly for diagonal elements. 
We check those:

Recall that for a concrete Kripke frame $\F_{\bold W}$ based on ${\bold W}=(W,\leq ,V_k, W_k),$ we have
the concrete diagonal element ${\sf d}_{ij}$ is given by the tuple $(g_k: k\in K)$ such that for $y\in V_k$, $g_k(y)=1$ iff $y(i)=y(j)$.

Now for the abstract diagonal element in $\A$, we have $\psi_1({\sf d}_{ij})=(f_k:k\in K)$, such that if $k=((\Delta, \Gamma), (T,F))$ 
is a matched pair of saturated theories in $\Sg^{\Rd\A_n}X_1$, $\Sg^{\Rd\A_n}X_2$, with $n$ unique to $i$, 
we have $f_k([x])=1$ iff ${\sf s}_{x}{\sf d}_{ij}\in \Delta \cup T$ (this is well defined $\Delta\cup T\subseteq \A_n).$
 
But the latter is equivalent to ${\sf d}_{x(i), x(j)}\in \Delta\cup T$, which in turn is equivalent to $x(i)\sim x(j)$, that is 
$[x](i)=[x](j),$ and so  $(f_k)\in {\sf d}_{ij}^{\F_{\mathfrak K}}$.  
The reverse implication is the same.

We can safely assume that $X_1\cup X_2=X$ generates $\A$.
Let $\psi=\psi_1\cup \psi_2\upharpoonright X$. Then $\psi$ is a function since, by definition, $\psi_1$ and $\psi_2$ 
agree on $X_1\cap X_2$. Now by freeness $\psi$ extends to a homomorphism, 
which we denote also by $\psi$ from $\A$ into $\F_{\mathfrak K}$.
And we are done, as usual, by $\psi$, $k_0$ and $Id\in V_{k_0}$.
\end{demo}

Theorem \ref{main2}, generalizes as is, to the expanded structures by diagonal elements. That is to say, we have:

\begin{theorem} The free dimension restricted free algebras have the interpolation property
\end{theorem}

\begin{demo}Assume that $\theta_1\in \Sg^{\A}X_1$ and $\theta_2\in \Sg^{\A}X_2$ such that $\theta_1\leq \theta_2$.
Let
$\Delta_0=\{\theta\in \Sg^{\A}(X_1\cap X_2): \theta_1\leq \theta\}.$
If for some $\theta\in \Delta_0$ we have $\theta\leq \theta_2$, then we are done.
Else $(\Delta_0, \{\theta_2\})$ is consistent. Extend this to a complete theory $(\Delta_2, \Gamma_2)$ in $\Sg^{\A}X_2$.
Consider $(\Delta, \Gamma)=(\Delta_2\cap \Sg^{\A}(X_1\cap X_2), \Gamma_2\cap \Sg^{\A}(X_1\cap X_2))$.
Then $(\Delta\cup \{\theta_1\}), \Gamma)$ is consistent. For otherwise, for some $F\in \Delta, \mu\in \Gamma,$ 
we would have $(F\land \theta_1)\to \mu$ and $\theta_1\to (F\to \mu)$, so $(F\to \mu)\in \Delta_0\subseteq \Delta_2$ which is impossible.
Now $(\Delta\cup \{\theta_1\}, \Gamma)$ $(\Delta_2,\Gamma_2)$ are consistent with $\Gamma\subseteq \Gamma_2$ and $(\Delta,\Gamma)$
complete in $\Sg^{\A}X_1\cap \Sg^{\A}X_2$. So by theorem \ref{main2}, $(\Delta_2\cup \{\theta_1\}, \Gamma_2)$ 
is satisfiable at some world in some set algbra based on a Kripke 
system, hence consistent. 
But this contradicts that
$\theta_2\in \Gamma_2, $ and we are done.
\end{demo}

\section{Sheaf Duality, epimorphisms and omitting types}

\begin{definition} Let $\B$ be an algebra. A filter of $\B$ is a nonempty subset $F\subseteq A$ such that for all $a,b\in B$,
\begin{enumroman}
\item $a,b\in F$ implies $a*b\in F.$
\item $a\in F$ and $a\leq b$ imply $b\in F.$
\end{enumroman}
\end{definition}
It easy to check that if $F$ is a filter on $A$ then $1\in F$ and whenever $a, a\implies b\in F$ then $b\in F$.
Also $a*b\in F$ if and only if $a\cap b\in F$ iff $a\in F$ and $b\in F$. A filter $F$ is proper if $F\neq A$ and it is easy to see that a filter $F$ is proper
iff $0\notin F$. 

\begin{definition} A filter $P$ of $A$ is prime provided that it is a prime filter of the underlying lattice $L(\B)$ of $\B$, that is
$a\cup b\in P$ implies $a\in P$ or $b\in P$. This is equivalent to the statement that for all $a,b\in \B,$ $a\implies b\in P$ or $b\implies a\in P$.
A proper filter $F$ is maximal if it is not properly contained in any other proper filter. 
\end{definition}
We let $Max(\B)$ denote the set of maximal filters
and $Spec(\B)$ the family of prime filters. Then it is not hard to actually show that $Max(\B)\subseteq Spec(\B)$ \cite{spec}.
For a set $X\subseteq \B$, $\Fl^{\B}X$ denotes the filter generated by $X$. 
A filter $F$ is called principal, if $F=\Fl\{a\}=\{x\in B: x\geq a\}$.
The following notions are taken from \cite{spec}. Proofs are also found in \cite{spec}.
Let $\B$ be a non-trivial algebra. For each $X\subseteq \B$, we set
$$V(X)=\{P\in Spec(X): X\subseteq P\}.$$ Then the family $\{V(X)\}_{X\subseteq \B}$ of subsets of $spec(\B)$ 
satisfies the axioms for closed sets in a topological space. The resulting topology is called the 
Zariski topology, and the resulting topological space is called the prime spectrum of $\B$. We write $V(a)$ for the more cumbersome $V(\{a\})$.
For any $X\subseteq \B$, let
$$D(X)=\{P\in Spec(X): X\nsubseteq P\}$$
Then $\{D(X)\}_{X\subseteq A}$ is the family of open sets of the Zariski topology.
We write $D(a)$ for $D(\{a\})$.
The minimal spectrum of $\B$ is the topology induced by the Zariski topology on $Max(\B)$.
For $X\subseteq \B$ and $a\in \B$, let
$$V_M(X)=V(X)\cap Max(\B)\text { and } D_M(X)=D(X)\cap Max(\B).$$
$$V_M(a)=V(a)\cap Max(\B),\text { and } D_M(a)=D(a)\cap Max(\B).$$
In other words,
$$V_M(a)=\{F\in Max(\B): a\in F\}$$ 
and $$D_M(a)=\{F\in Max(\B): a\notin F\}.$$

\begin{lemma} Let $\B$ be an algebra. Let $a,b\in \B$.
Then the following hold:
\begin{enumroman}
\item $D_M(a)\cap D_{M}(b)=D_M(a\cup b).$
\item  $D_M(a)\cup D_M(b)=D_M(a\cap b)=D_M(a*b).$
\item $D_M(X)=Max(\B)$  iff $\Fl^{\B}X=\B.$
\item  $D_M(\bigcup_{i\in I}X_i)=\bigcup_{i\in I}D_M(X_i).$
\item $V_M(a)\cap V_M(b)=V_M(a\cap b).$
\item $a\leq b$ if and only if $V_M(a)\subseteq V_M(b).$
\end{enumroman}
\end{lemma}
\begin{demo}{Proof} \cite{spec} proposition 2.8. We only prove one side of the last item, since it is not mentioned in \cite{spec}.
Assume that $V_a\subseteq V_b$. If it is not the case that $a\leq b$, then we may assume that $a\cap (b\implies 0)$ is not $0$.
Hence there is a proper maximal filter $F$, such that $a\cap (b\implies  0)\in F.$ Hence $a\in F$ and $b\to 0$ is in $F$.
But this implies that $b\notin F$ lest $0\in F$. Hence $F\in V_a$ and $F\notin V_b.$ This is a contradiction, and the required is proved.
\end{demo}
\begin{theorem}\label{d}Let $\B$ be an algebra. 
\begin{enumroman}
\item $\{D_M(a)\}_{a\in \B}$ is a basis for a compact Hausdorff topology on $Max(\B)$
\item Furthermore if $a=\bigvee a_i$, then $V_M(a)\sim \bigcup V_M(a_i)$ is a nowhere dense subset of $Max(\B)$. 
Similarly if $a=\bigwedge a_i$, then $\bigcap V_M(a_i)\sim V_M(a).$
is nowhere dense. 
\item If $\B$ is countable, then $Max(\B)$ is a Polish space.
\end{enumroman}
\end{theorem} 
\begin{demo}{Proof} 
\begin{enumroman}
\item We include the proof for self completeness and also because the `nowhere density' part is completely new, 
and as we shall see in a while it will play a pivotal role in the proof of the omitting types theorem.
That $Max(\B)$ is compact and Hausdorff  is proved in \cite{spec}, theorem 2.9, the proof goes as follows:
Assume that $$Max(\B)=\bigcup_{i\in I}D_M(a_i)=D_M(\bigcup_{i\in I}a_i).$$ Then $\B=\Fl\{\bigcup_{i\in I} a_i\}$, hence $0\in \Fl\{\bigcup_{i\in I}a_i\}$.
There is an $n\geq 1$ and $i_1,\ldots i_n\in I$ such that $a_{i_1}*\ldots a_{i_n}=0$. But 
$$Max(\B)=D_{M}(0)=D_M(a_{i_1}*\ldots a_{i_n})=D_M(a_{i_1})\cup\ldots D_M(a_{i_n}).$$
Hence every cover is reducible to a finite subcover. Hence the space is compact.
Now we show that it is Hausdorff. Let $M$, $N$ be distinct maximal filters. 
Let $x\in M\sim N$ and $y\in N\sim M$. Let $a=x\implies  y$ and $b=y\implies  x$. Then $a\notin M$ and $b\notin N$. Hence 
$M\in D_M(a)$ and $N\in D_M(b)$. Also $D_M(a)\cap D_M(b)=D_M(a\lor b)=D_M(1)=\emptyset$.
 We have proved that the space is Hausdorff.

\item Now assume that $a=\bigvee a_i$ and $V_M(a)\sim \bigcup V_M(a_i)$ is not nowhere dense. 
Then there exists $d$ such that $D_M(d)\subseteq V_M(a)\sim V_M(a_i)$
Hence $$V_M(a_i)\subseteq V_M(a)\sim D_M(d)=V_M(a)\cap V_M(d)=V_{M}(a\cap d).$$ It follows that
$a\cap d=a$ so $a\leq d$. Then $D_M(d)\subseteq D_M(a)$. So we have, 
$D_M(d)\subseteq D_M(a)\cap V_M(a)=\emptyset$ contradiction.
Conversely assume that $a=\bigwedge{a_i}$ 
and assume that $$D_M(d)\subseteq \bigcap V_M(a_i)\sim V_M(a).$$
Let $e=d\to 0$. Then $V_M(e)=D_M(d)$.
Now we have
$$V_M(e)\subseteq \bigcap V_M(a_i)\sim V_M(a).$$
Taking complements twice, we get
$$V_M(e)\subseteq D_M(a)\sim \bigcup D_M(a_i)$$
Then
$V_M(e)\subseteq D_M(a)\sim D_M(a_i)$. So $$D_M(a_i)\subseteq D_M(a)\sim V_M(e)=D_M(a)\cap D_M(e)=D_M(a\cup e).$$
Hence $V_M(a\cup e)\subseteq V_M(a_i)$. So $a\cup e\leq a_i$ for each $i$. Thus $a\cup e=a$ from 
which we get that $e\leq a$. Hence $V_M(e)\subseteq V_M(a)$.
But $V_M(e)\subseteq D_M(a)$ it follows that $V_M(e)=\emptyset$. But $V_M(e)=D_M(d)$ and we are done.
\item If $\B$ is countable, then $Max{\B}$ is second countable, so the required follows.
\end{enumroman}
\end{demo}

We start by a concrete example addressing variants and extension first order logics. 
The following discussion applies to $L_n$ (first order logic with $n$ variables), $L_{\omega,\omega}$ (usual first order logic), 
rich logics, Keislers logics with and without equality, finitray logics of infinitary
relations; the latter three logics are infinitary extensions of first order logic, 
though the former and the latter have a finitary flavour, because quantification is taken only
on finitely many variables. These logics have an extensive literature in algebraic logic.
%It also applies to non classical logics, whose Stone space is the Zarski topology.
Let us start with the concrete example of usual first order logic. $\L_n$ denotes a relational first order language (we have no function symbols)
with $n$ constants, $n\leq \omega,$
and as usual a sequence of variables of order type $\omega$.

\begin{example}

Let $\Sn_{\L_n}$ denote the set of all $\L_n$ sentences, and  fix an enumeration $(c_i: i<n)$ of the constant symbols.
We assume that $T\subseteq  \Sn_{\L_0}$. 
Let $X_T=\{\Delta\subseteq \Sn_{\L_0}: \Delta \text{ is complete }\}$.
This is simply the underlying set of the Priestly  space, equivalently  the Stone space, 
of the Boolean algebra $\Sn_{L_0}/T$. For each $\Delta\in X_T,$ let $\Sn_{\L_n}/{\Delta}$
be the corresponding Tarski-Lindenbaum quotient algebra, which is a (representable) cylindric algebra of dimension $n$. 
The $i$th cylindrifier $c_i$ is defined by 
$c_i\phi/{\Delta}=\exists \phi(c_i|x)$, where the latter is the formula
obtained by replacing the $i$th constant if present by the first variable $x$ not occurring in $\phi$, 
and then applying the existential quantifier $\exists x$.
Let $\delta T$ be the following disjoint union
$\bigcup_{\Delta\in X_T}\{\Delta\}\times Sn_{\L_n}/{\Delta}.$ 
Define the following topologies, on $X_T$ and $\delta T$, respectively.
On $X_{T}$ the Priestly (Stone) topology, and on $\delta_{\Gamma}$ 
the topology with base $B_{\psi,\phi}=\{\Delta, [\phi]_{\Delta}, \psi\in \Delta, \Delta\in \Delta_{\Gamma}\}.$
Then $(X_T, \delta T)$ is a {\it sheaf}, and its dual consisting of the continuous sections, 
$\Gamma(T,\Delta)$, with operations defined pointwise, is actually isomorphic to $\Sn_{\L_n}/T$. 

\end{example}

\begin{example}
By the same token, let $\L$ be the  predicate language for $BL$ algebras, $\Fm$ denotes the set of $L$ formulas, and $\Sn$ denotes the 
set of all sentences (formulas with no free variables). 
This for example includes $MV$ algebras; that are, in turn, algebraisations of many 
valued logics.
Let $X_T$ be the Zarski (equivalently the Priestly) topology on $\Sn/T$ based on $\{\Delta\in Spec(\Sn): a\notin \Delta\}$. 
Let $\delta T=\bigcup _{\Delta\in X_T}\{\Delta\}\times \Fm_{\Delta}$. 
Then again, we have $(X_T, \delta T)$ is a {\it sheaf}, and its dual consisting of the continuous sections with operations defined pointwise, 
$\Gamma(T,\Delta)$ is actually isomorphic to $\Fm_T$. 
\end{example}

This situation is very similar to the one in algebraic geometry of desribing the ring associated with the affine variety 
in terms of the local rings given at at point of the variety.

This needs further clarification. Let us formalize the above concrete examples in an abstract more general setting, that allows further applications.
Let $\A$ be a bounded distributive lattice with extra operations $(f_i: i\in I)$. $\Zd\A$ denotes the distributive 
bounded lattice $\Zd\A=\{x\in \A: f_ix=x,\  \forall i\in I\}$, where the operations are the natural restrictions.(Idempotency of the $f_i$s guarantees 
that this is well defined). 
If $\A$ is a locally finite  algebra of formulas of first order logic 
or predicate modal logic or intiutionistic logic, or any predicate logic where the $f_i$s are interpreted as the existential 
quantifiers, then $\Zd\A$ is the Boolean 
algebra of sentences.

Let $\K$ be class of bounded distributive lattices with extra operations $(f_i: i\in I)$.
We describe a functor that associates to each $\A\in \K$, and $J\subseteq I$, a pair of topological spaces
$(X(\A,J), \delta(\A))=\A^d$, where $\delta(\A)$ has an algebraic structure, as well; in fact it is a subdirect product
of distributive lattices, that turn out to be simple (have no proper congruences)
under favourable circumstances, in which case $\delta(\A)$ is a semi-simple lattice carrying 
a product topology. This pair is  called the dual space of $\A$.
For $J\subseteq I$, let $\Nr_J\A=\{x\in A: f_ix=x \forall i\notin j\}$, with operations $f_i: i\in J$.
$X(\A,J)$ is the usual dual space of $\Nr_J\A$, that is, the set of all prime ideals of the lattice $\Nr_J\A$, 
this becomes a Priestly space (compact, Hausdorff and totally disconnected), when we take the collection of all sets
$N_a=\{x\in X(\A,J): a\notin x\}$, and their complements, as a base for the topology. 

For a set $X$ of an algebra $\A$ we let $\Co^{\A} X$ denote the congruence relation generated by $X$ (in the universal algebraic sense).
This is defined as the intersection of all congruence relations that have $X$ as an equivalence class.
%It is worthy of note, at this point that ideals and congruences 
Now we turn to defining the second component; this is more involved. 
For $x\in X(\A,J)$, let $\G_x=\A/\Co^{\A}x$  and 
$\delta(\A)=\bigcup\{\G_x: x\in X(\A)\}.$
This is clearly a disjoint union, and hence 
it can also be looked upon as the following product $\prod_{x\in \A} \G_x$ of algebras. 
This is not semi-simple, because $x$ is only prime,
least maximal in $\Nr_J\A$. 
But the semi-simple case will deserve special attention.

The projection $\pi:\delta(\A)\to X(\A)$ is defined for $s\in \G_x$ by $\pi(s)=x$.Here $\G_x=\pi^{-1}x$ is the stalk over $x$. For $a\in A$, 
we define a function
$\sigma_a: X(\A)\to \delta(\A)$ by $\sigma_a(x)=a/\Ig^{\A}x\in \G_x$. 

Now we define the topology on 
$\delta(\A)$. It is the smallest topology  for which all these functions are open, so $\delta(\A)$ 
has both an algebraic structure and a topological one, and they are compatible.
%Indeed, $\A^d=(X(\A), \delta(\A))$ is a reduced space.  

We can turn the glass around. Having such a space we associate a bounded distributive lattice in $\K$.
Let $\pi:\G\to X$ denote the projection associated with the space $(X,\G)$, built on $\A$.
A function  $\sigma:X\to \G$ is a section of $(X,\G)$ if $\pi\circ \sigma$ is the identity 
on $X$. 

Dually, the inverse construction  uses the sectional functor.
The set $\Gamma(X,\G)$ of all continuous sections of 
$(X,\G)$ becomes a $BLO$ by defining the operations pointwise, recall that $\G=\prod \G_x$ is a product of bounded distributive lattices.

The mapping $\eta:\A\to \Gamma(X(\A,J), \delta(\A))$ defined by $\eta(a)=\sigma_a$ 
is as easily checked  an isomorphism. 
Note that under this map an element in $\Nr_J\A$ corresponds with the characteristic 
function $\sigma_N\in \Gamma(X, \delta)$ 
of the basic set $N_a$.

To complete the definition of the contravariant 
functor we need to define the dual of morphisms. 

Given two spaces $(Y,\G)$ and $(X,\L)$ a sheaf morphism $H:(Y,\G)\to (X,\L)$ is a pair $(\lambda,\mu)$ where $\lambda:Y\to X$ is a continous map
and $\mu$ is a continous map $Y+_{\lambda} \L\to \G$ such that $\mu_y=\mu(y,-)$ is a homomorphism of $\L_{\lambda(y)}$ into $\G_y$.
We consider $Y+_{\lambda} \L=\{(y,t)\in Y\times \L:\lambda(y)=\pi(t)\}$ as a subspace of $Y\times \L$.
That is, it inherits its topology from the product topology on $Y\times \L$.

A sheaf morphism $(\lambda,\mu)=H:(Y,\G)\to (X,\L)$ produces a homomorphism of lattices
$\Gamma(H):\Gamma(X,\L)\to \Gamma(Y,\G)$ the natural way:
for $\sigma\in \Gamma(X,\L)$ define $\Gamma(H)\sigma$ by $(\Gamma(H)\sigma)(y)=\mu(y, \sigma(\lambda y))$ for all $y\in Y$.
A sheaf morphism $h^d:\B^d\to \A^d$ can also be asociated with a homomorphism $h:\A\to \B$. 
Define $h^d=(h^*, h^o)$ where for $y\in X(\B)$, $h^*(y)=h^{-1}\cap Zd\A$ and for $y\in X(\B)$ and $a\in A$
$$h^0(h, a/\Ig^{\A}h^*(y))=h(a)/\Ig^{\B}y.$$

%Because prime and maximal ideals are different we have:

\begin{example} Let $\A=\prod_{i\in I}\B_i$, where $\B_i$ are directly indecomposable $BAO$s. Then
$\Zd\A= {}^I2$ and $X(\A)$ is the Stone space of this algebra.
The stalk $\delta_{M}(\A)$ of $\A^{\delta}$ over $M\in X(\A)$ is the ultraproduct 
$\prod_{i\in I}\B_i/F$  where $F$ is the ultrafilter on $\wp(I)$  corresponding to $M$.
\end{example}

\begin{definition} Let $\A\in \CA_{\omega}$ and $x\in A$. The dimension set of $x$, in symbols $\Delta x$, 
is the set $\{i\in \omega: c_ix\neq x\}.$ Let $n\in \omega$. 
Then the $n$ neat reduct of $\A$ is the cylindric algebra of dimension $n$ 
consisting only of $n$ dimensional elements (those elements such that $\Delta x\subseteq n)$,
 and with operations indexed up to $n$.
\end{definition}
\begin{example}

\begin{enumarab}

\item Let $\A\in \Nr_n\CA_{\omega}$. Then there is a sheaf ${\bf X} =(X, \delta, \pi)$ such that $\A$ is isomorphic to continous sections
$\Gamma(X;\delta)$ of $\bold X.$  Indeed, let $X(\A)$ be the Stone space of $\Zd\A$. Then for any maximal ideal 
$x$ in $\Zd\A$, $\Ig^{\A}(x)$ is maximal in $\Nr_n\A$.
Let $\delta(A)=\bigcup \G_x$, where $\G_x=\A/\Ig^{\A}x$. 
The projection $\pi:\delta(\A)\to X(\A)$ is defined 
for $s\in \G_x$ by $\pi(s)=x$. For $a\in A$, 
we define a function $\sigma_a: X(\A)\to \delta(\A)$ by $\sigma_a(x)=a/\Ig^{\A}x\in \G_x$. 
Then $\pi\circ  \sigma$ is the identity and $\delta(\A)$ has  the smallest topology such that these maps are continuous.
Then $\eta: \A\to \Gamma(X(\A)), \delta)$ defined by
$\eta(a)=\sigma_a$ is the desired isomorphism.

\item Let $\A\in \Nr_n\CA_{\omega}$. For any ultrafilters 
$\mu$ and $\Gamma$ in $\Zd\A$, the map $\lambda:\A/\mu\to \A/\Gamma$ defined via, 
$a/\mu\mapsto a/\Gamma$
maps $\Zd\A$ into $\Zd\A$. (The latter is the set of zero-dimensional elements). 
The dual morphism is $\lambda^d=(\lambda, \lambda^0) :(X_{\Gamma}, \delta(\Gamma))\to (X_{\mu}, \delta(\mu))$, 
is defined by $\lambda(\Delta)=\Delta$ and $\lambda^0(\Delta, (\Delta), a/{\Delta}))=(\Delta, a/\Delta)$.
Thus it is an isompphism from $(X_{\Gamma}, \delta(\Gamma)$ onto the restriction of 
$(X_{\mu}, \delta(\mu))$ to the closed set $ X_{\Gamma}$.
Conversley, every restriction of $(X_{\mu}, \delta(\mu))$ to a closed subset $Y$ of $X_{\mu}$ is up to isomorphism the dual space of 
$\Nr_n\A/F$ for a filter $F$ of $\Zd\A$. For if $\Gamma=\bigcap Y$, then $Y=X_{\Gamma}$ 
since $Y$ is closed and the dual space of $\Nr_n\A/{\Gamma}$ is isomorphic to 
$(Y, \delta(\mu)\upharpoonright Y)$.
\end{enumarab}
\end{example}

For an algebra $\A$ and $X\subseteq \A$, $\Ig^{\A}X$ is the ideal generated by $X$. We write briefly lattice for a $BLO$; 
hopefully no confusion is likely to ensue.
\begin{definition}
\begin{enumarab}
\item  A lattice $L$ is regular if whenever $x$ is a prime ideal in $\Zd L$, then $\Ig^{\A}x$ is a prime ideal in
$\A$.

\item A lattice $L$ is strongly regular, if whenever $x$ is a prime idea in $\Zd \L$, then $\Ig^{\A}x$ is a maximal ideal in $\A$.

\item A lattice $L$ is congruence strongly regular, if whenever $x$ is a prime ideal in $\Zd\L$, then $\Co^{\A}x$ is a maximal congruence of $\A$.
%\item A lattice $L$ is 
\end{enumarab}
\end{definition}
%\begin{theorem} if $\A$ is a $BLO$ that is strongly regular, then the lattice of 
%\end{theorem}
%\end{definition}
If $L$ is not relatively complemented, then (2) and (3) above are not equivalent; but if it is relatively complemented then they are equivalent. 
A lattice with the property that every interval is complemented is called a relatively complemented lattice. 
In other words, a relatively complemented lattice is characterized by the property that for every element $a$ in an interval $[c,d]=\{x: c\leq x\leq d\}$
there is an element $b$, such that $a\lor b=d$ and $a\land b=c$. 
Such an element is called a complement; it may not be unique, but if the lattice is bounded then relative complements in $[a, 1]$ are just 
complements, and in case of distributivity such complements are 
unique.
In arbitrary lattices the lattice of ideas may not be isomorphic to the lattice of congruences, the 
following theorem gives a sufficient and necessary condition for this to hold.
The theorem is  a classic due to Gratzer and Schmidt.

\begin{theorem} For the correspondence between congruences and ideals to be an 
isomorphism it is necessary and sufficient that $L$ is distributive, relatively complemented with a minimum $0$.
\end{theorem}
\begin{proof} {\bf Sketch} Clearly the ideal corresponding to the identity relation is the $0$ ideal. 
Since every ideal of $L$ is a congruence class under some homomorphism, we obtain distributivity.
To show relative complementedness,  it suffices to show that if $b<a$, then $b$ has a complement in the interval $[0,a]$. 
Let $I_{a,b}$ be the ideal which consists of all $u$ with 
$u\equiv 0(Theta_{a,b})$. $V_{a,b}$ is 
a congruence class under precisely one relation, hence $a\equiv b mod(\Theta[V_{a,b}])$.  Hence for some $v\in I_{a,b}$ we have 
$b\lor v=a$ and $b\land v=0$. 
Conversely, we have every ideal is a congruence class under at most one congruence relation, and of course under at least one.
\end{proof}

In case or relative complementation, we have
\begin{theorem} the following are equivalent
\begin{enumarab}
\item $L$ is strongly regular
\item Every principal ideal of $L$ is generated by a an elemnt in $\Zd L$
\item $\delta(\A)$ is semisimple
\end{enumarab}
\end{theorem}
\begin{proof} Easy
\end{proof}

We push the duality a step futher esatablishing a correspondence between open (closed) sets of $BLO$s and open subsets of its dual.
An ideal $I$ in $\A$ is regular if $\Ig^{\A}(I\cap \Zd\A)=I$.
\begin{theorem} There is an isomomorphism between the set of all regular ideals in $\Gamma(X, \delta)$ 
onto the lattice of open subsets of $X.$
\end{theorem}
\begin{proof} For $\sigma\in \Gamma(X,\delta),$ let $[\sigma]=\{x\in X: \sigma(x)\neq 0_x\}$. For $U\subseteq X$, let 
$J[U]=\{\sigma\in \Gamma(X, \delta): [\sigma]\subseteq U\}.$ 
Then $J\mapsto U[J]$ is an isomorphism, its inverse is $U[J]=\bigcup\{[\sigma]: \sigma\in J\}.$
\end{proof}

%It is easy to see that locally finite algebras are nice. 
%For a class of algebras $K$ we say that $K$ has $ES$ if epimorphisms (in the categorial sense) are surjective.
Note that a simple lattice is necessarily strongly regular (and hence regular), but the converse is not true, even in the case of strong regularity.
There are easy examples.
As an application to our duality theorem established above, we show that certain properties can extend from simple structures to 
strongly regular ones. The natural question that bears an answer is how far are strongly regular algebras from simple algebras; 
and the answer is: pretty 
far. 
For example in cylindric algebras any non-complete theory $T$ in a first order language gives 
rise to a strongly regular $\omega$-dimensional algebra, namely, $\Fm_T$, that is not simple.
%We will show that $ES$ fails in the class of simple infinite dimensional cylindric algebras. 
%To prove our result we use Stephen Comer's sheaf theoretical construction \cite{Co1}, \cite{Co2} and Judit's 
%Madar\'asz construction \cite{M} for the semisimple case.

%\section{Applications}

$ES$ abreviates that epimorphisms (in the categorial sense) are surjective. Such abstract property 
is equivalent to the well-known Beth definability property 
for many abstract logics, including fragments of first order logic, and multi-modal logics. 

In fact, it applies to any algebraisable logic (corresponding to a quasi-variety) regarded
as a concrete category.  This connection was established by N\'emeti. 
As an application, to our hitherto established duality, we have: 

\begin{theorem} Let $V$ be a class of distributive bounded lattices such that the simple lattices in $V$ 
have the amalgamation property $(AP)$.
Assume that there exist strongly regular lattices $\A,\B\in V$ and an epimorphism $f:\A\to \B$ that is not onto.
Then $ES$ fails in the class of simple lattices
\end{theorem}

\begin{demo}{Proof} Suppose, to the contrary that $ES$ holds for simple algebras.
Let $f^*:\A\to \B$ be the given epimorphism that is not onto. We assume that $\A^d=(X,\L)$ and $\B^d=(Y,\G)$ 
are the corresponding dual sheaves over the Priestly  spaces $X$ and $Y$ and by  duality that 
$(h,k)=H:(Y,\G)\to (X,\L)$ is a monomorphism. Recall that $X$ is the set of prime ideals in $Zd\A$, and similarly for $Y$.
We shall first prove
\begin{enumroman}
\item $h$ is one to one
\item for each $y$ a maximal ideal in $\Zd\B$, $k(y,-)$ is a surjection of the stalk over $h(y)$ onto the stalk over $y$.
\end{enumroman}
%This part of the proof is identical to \cite{Co1} Theorem 5.3; it is not a generalization. 
Suppose that $h(x)=h(y)$ for some $x,y\in Y$. Then $\G_x$, $\G_y$ and $\L_{hx}$ are simple algebra, 
so there exists a simple $\D\in V$ and monomorphism $f_x:\G_x\to \D$ and $f_y:\G_y\to \D$ such that
$$f_x\circ k_x=f_y\circ k_y.$$
Here we are using that the algebras considered are strongly regular, and that the simple algebras have $AP$.
Consider the sheaf $(1,D)$ over the one point space $\{0\}=1$ and sheaf morphisms 
$H_x:(\lambda_x,\mu):(1,D)\to (Y,\G)$ and $H_y=(\lambda_y, v):(1,D)\to (Y,\G)$
where $\lambda_x(0)=x$ $\lambda_y(0)=y $ $\mu_0=f_x$ and $v_0=f_y$. The sheaf $(1,\D)$ is the space dual to $\D\in V$
and we have $H\circ H_x=H\circ H_y$. Since $H$ is a monomorphism $H_x=H_y$ that is $x=y$.
We have shown that $h$ is one to one.
Fix $x\in Y$. Since, we are assuming that  $ES$ holds for simple algebras of $V,$ in order to show that 
$k_x:\L_{hx}\to \G_x$ is onto, it suffices to show that $k_x$ is an epimorphism.  
Hence suppose that $f_0:\G_x\to \D$ and $f_1:\G_x\to \D$ for some simple $\D$ such that $f_0\circ k_x=f_1\circ k_x$.
Introduce sheaf morphisms 
$H_0:(\lambda,\mu):(1,\D)\to (Y,\G)$ and $H_1=(\lambda,v):(1,\D)\to (Y,\G)$
where $\lambda(0)=x$, $\mu_0=f_0$ and $v_0=f_1$. Then $H\circ H_0=H\circ H_1$, 
but $H$ is a monomorphism, so we have $H_0=H_1$ from which 
we infer that $f_0=f_1$. 

We now show that (i) and (ii) implies that $f^*$ is onto, which is a contradiction.
%In this part we follow Comer \cite{Co1} Lemma 5.1.
Let $\A^d=(X,\L)$ and $\B^d=(Y,\G)$. It suffices to show that $\Gamma((f^*)^d)$ is onto (Here we are taking a double dual) . 
So suppose $\sigma\in \Gamma(Y,\G)$. For each $x\in Y$, 
$k(x,-)$ is onto so $k(x,t)=\sigma(x)$ for some $t\in \L_{h(x)}$. That is $t=\tau_x(h(x))$ for some 
$\tau_x\in \Gamma(X,\G)$. Hence there is a clopen neighborhood $N_x$ of $x$ such that 
$\Gamma(f^*)^d)(\tau_x)(y)=\sigma(y)$ for all $y\in N_x$. 
Since $h$ is one to one and $X,Y$ are Boolean spaces, we get that $h(N_x)$ is clopen in $h(Y)$ and there is a 
clopen set $M_x$ in $X$ such that $h(N_x)=M_x\cap h(Y)$. Using compactness, there exists a partition of
$X$ into clopen subsets $M_0\ldots M_{k-1}$ and sections $\tau_i\in \Gamma(M_i,L)$ such that
$$k(y,\tau_i(h(y))=\sigma(y)$$ 
wherever $h(x)\in M_i$ for $i<k$. Defining 
$\tau$ by $\tau(z)=\tau_i(z)$ whenever $z\in M_i$ $i<k$, it follows that $\tau\in \Gamma(X,\L)$ and $\Gamma((f^*)^d)\tau=\sigma$.
Thus $\Gamma((f^*)^d)$ is onto $\Gamma(\B^d)$, and we are done. 
\end{demo}

\section{Omitting types in non classical logics, topologically}

From now on $\L$ is a core fuzzy logic. We set up the context where omitting types apply. 
\begin{definition}
\begin{enumroman}
\item An $n$ type, or simply a type,  is a set $\Gamma$ 
whose formulas have free variables among the first $n$ variables. Fix a theory $T$ and an $n$
type $\Gamma$.
$\M=(M,\bold L)$ realizes $\Gamma$ if there is $s\in {}^nM$
such that $||\phi(s)||_{\M}^{\bold L}=1$ for all $\phi\in \Gamma$. $\M$ omits $\Gamma$ if $\M$ does not realize $\Gamma$. 
\item $\Gamma$ is a principal type in $T$ if there is a formula
$\phi(\bar{x})$ such that for all $\M\models T$ for all $v\in {}^nM$, $||\phi(v)\implies \psi(v)||_{\M}^{\bold L}=1$ for all $\psi\in \Gamma$.
Otherwise $\Gamma$ is non-principal.
\item Call a formula $\exists x\phi$ containing free variables $y_1\ldots y_n$ witnessed in $(\M, \bold L)$ if for each $a_1\ldots a_n\in M$, 
there is an element $b\in M$ such that 
$$||\exists x\phi(x,a_1\ldots a_n||_{\M}^{\bold L}=||\phi(b, a_1\ldots a_n)||_{\M}^{\bold L};$$ similarly for $\forall x\phi$. Call $(\M,\L)$ 
witnessed if each formula beginning  
with a quantifier is witnessed. 
\end{enumroman}
\end{definition}

Let $\kappa$ be a cardinal. Consider the following statement:

\begin{athm}{$OTT(\kappa)$} If $\Sigma$ is a countable theory and $\Gamma_i$, $i\in \kappa,$ are non-principal types, 
then there is a witnessed safe model of $T$
omitting these types.
\end{athm}

\begin{theorem}
\begin{enumroman}
\item  The statement $(\forall \kappa\leq \omega)$$OTT(\kappa)$ is provable in $ZFC$.
\item The statement $(\forall \kappa<{}^{\omega}2)OTT(\kappa)$ is independent of $ZFC.$
\end{enumroman}
\end{theorem}
 
\begin{demo}{Proof}
We only prove (i). A similar statement is proved in \cite{fuzzy} but our proof is completely different.
(ii) will be proved later. Let $\Sigma$ be a countable theory, and $\{\Gamma_i:$ $i<\omega\}$ be a family 
of non-principal types.
Add infinitely many countably many constants, let the added constants be $C$. 
Now $\Sigma$ can be expanded to Henkin complete theory $T$ which is  a countable union $\bigcup_{m\in \omega} T_m$
where $T_0=T\subseteq T_1\subseteq T_2\subseteq\ldots T_m\ldots$
such that 
\begin{enumarab}
\item each $T_m$ is consistent and is obtained from $T$ by adding finitely many axioms using only finitey many constants,
\item for every pair of sentences in the expanded language we have
$T_m\vdash \phi\to \psi$ or $T_m\vdash \psi\to \phi,$
\item If $T_m\nvdash \forall y \phi(y)$, then $T_m\nvdash \phi(c),$
\end{enumarab} 
Consider the algebra of sentences $\B=\Fm_T$. We write $[\phi]$ for $[\phi]_T$ the equivalence clas containing $\phi$.
Let $X=Max(\B)$, be the compact Hausdorff space with countable basis $D_M(a)$ for $a\in B$.
Then in $\B$, since the extension $T$ is Henkin, we have
\begin{equation}\label{e1}
\begin{split}
[\exists x\alpha(x)]=\bigvee_{c\in C} [\alpha(c)]
\end{split}
\end{equation}
\begin{equation}\label{e2}
\begin{split}
[\forall x\alpha(x)]= \bigwedge_{c\in C}[\alpha(c)]
\end{split}
\end{equation}
Now since the types considered are non-principal, they remain non principal in the expanded language.
For if not, then $T\vdash \phi\to \Gamma_i(\bar{c})$ for some $i$, then $T_m\vdash \phi\to \Gamma_i(\bar{c})$.
Now $T_m$ and $\phi$ contain only finitely many constants, replacing those by new variables (distinct variables for distinct contants) 
such that substitutions are free,
avoiding collisions, we obtain that $\Sigma\vdash \phi'\to \Gamma_i$ which is a contradiction.
It therefore follows that 
\begin{equation}\label{e3}
\begin{split}
\forall  c_1,\ldots c_n,  \forall i\in \omega, \bigwedge_{\phi\in \Gamma_i} [\phi(c_1,\ldots c_n)]=0
\end{split}
\end{equation}
Then for every variable $x$ and formula with one free variable  $\alpha$ we have
\begin{equation}\label{e4}
\begin{split}
H_{\alpha,x}=V_M({\exists x\alpha(x))}\sim \bigcup_{c\in C} V_M({\alpha(c)}),
\end{split}
\end{equation}
\begin{equation}\label{e5}
\begin{split}
J_{\alpha,x}=\bigcap_{c\in C}V_M({\alpha(c)})\sim V_M({\forall x\alpha x}),
\end{split}
\end{equation}
and for every $i$, 
\begin{equation}\label{e6}
\begin{split}
K_{(c_1,\ldots c_n, \Gamma_i)}=\bigcap_{\phi\in \Gamma_i} V_M({\phi(c_1\ldots c_n))}
\end{split}
\end{equation}
are, by theorem \ref{d},  nowhere dense sets. 
Let $$H=\bigcup_{\alpha} \bigcup_{x} H_{\alpha,x},$$ 
$$J=\bigcup_{\alpha}\bigcup_{x} J_{\alpha,x},$$
and
$$K=\bigcup_{\bar{c}}\bigcup_{i\in \omega}K_{(\bar{c},\Gamma_i)}.$$
Then each of these sets is a countable union of nowhere dense sets. 
Given any $a=[\psi]\in \B$, let $F$ be a maximal filter in the complement of the union of these sets and in 
$D_{M}a$. this is possible since by theorem \ref{d}, $Max(\B)$ is a Polish space, and the Baire category theorem holds, 
so that the complement of a countable collection of nowhere dense sets is dense.
Let $T_1=\cup F=\{\phi: [\phi]_T\in F\}$. 
Then we have the following:
\begin{enumerate}
\item $F\notin \bigcup_{\alpha,x}H_{\alpha,x}$ implies that for any $\alpha$, for any $x$, if $\exists x\alpha(x)\in T_1$ then $\alpha(c)$ is in 
$T$ for some $c$.
\item $F\notin \bigcup_{\alpha,x}J_{\alpha,x}$, implies that for  $\alpha$ for any $x$, if not $T_1\vdash \forall x \alpha(x)$ then not $T_1\vdash \alpha(c)$. 
This means that $T_1$ is Henkin. 
\item Finally, $F\notin K$ means that for all $i\in \omega$ for all $\phi\in \Gamma_i$ there exists $\bar{c}$ such that
$\phi(\bar{c)}\notin T_1$. 
\end{enumerate}
Then $T_1$ is a Henkin theory, $T_1\nvdash \psi$ and its canonical model is as desired.
That is the canonical model of $T_1$ is safe, witnessed and omits the given types.
\end{demo}
Theories considered remain countable. However, we now ask for the omission of possibly uncountably many non-principal types.
We shall prove (ii). We will see that we are actually touching upon somewhat deep issues in set theory here.
We need two lemmas. The first is a known consequence of Martin's axiom (henceforth $MA$). 
\begin{lemma}\label{m} 
The statement $(\forall \kappa<{}^{\omega}2)$$OTT(\kappa)$ is provable in $ZFC$ +$MA$
\end{lemma}
\begin{demo}{Proof} By Martin's axiom the union of $<{}^{\omega}2$ nowhere dense sets is equal to a countable union.
\end{demo}
\begin{theorem} The statement $(\forall k<{}^{\omega}2)(OTT(\kappa))$ is independent of $ZFC+\neg CH$.
\end{theorem}
\begin{demo}{Proof} We have proved consistency since $MA$ implies the required statement.  Now we prove independence. Let $covK$ 
be the least cardinal $\kappa$ such that the real line can be covered by $\kappa$ many closed disjoint 
nowhere dense sets. It is known that 
$\omega<covK\leq 2^{\omega}$. In any Polish space the intersection of $< covK$ dense sets is dense \cite{CF}. 
But then if $\kappa<covK$, then $OTT(\kappa)$ is true.
The independence can be  proved using standard iterated forcing to show that it is consistent that $covK$ could be literally anything greater 
than  $\omega$ and $\leq {}^{\omega}2$.
\end{demo}
\begin{theorem} 
The statement $(\forall \kappa<covK)OTT(\kappa)$ is provable in $ZFC$.
\end{theorem}
Note that  Martin's axiom implies that $covK =2^{\omega}$ which reproves \ref{m}. 
This is mentioned in \cite{CF}. The connection between omitting types and 
combinatorics of the real line was first 
discovered by Newelski \cite{N}.   

\subsection{Some model theoretic consequences of the omitting types theorem}

In classical first order logic, the omitting types theorem is used to construct what is known as atomic and prime models, which are minimal models.
In this section we define atomic, prime and saturated models in the fuzzy context and work out some connections between them. 
All languages considered are 
countable.

Now, having an omitting types theorem at our disposal for fuzzy 
logic, such investigations can be carried out in this more general context.
We give a sample by studying the so called atomic models. 
We define atomic models. Let $\L$ be a core fuzzy predicate logic.

\begin{definition} 
\begin{enumroman} Let $T$ be a theory.
\item  A formula $\phi(x_1\ldots x_n)$ is said to be complete in $T$ if $\phi$ is consistent with $T$ and whenever 
$\psi$ is consistent with $T$ and
$T\models \psi\implies \phi$ then $T\models \phi\implies  \psi$. 
\item A formula $\theta$ is completable if there is a complete formula $\psi$ such that $T\models \psi\implies \theta$.
\item $T$ is atomic if every formula consistent with $T$ is completable. 
\item For a model $\M$, let $Th(\M)$ be the set of all sentences in the language of $\M$ that are valid. That is $\phi\in Th(\M)$ iff $||\phi||_{\M}=1$.
A model $\M=(M,\bold L)$ is an atomic model, if for 
every $a_1,\ldots a_n\in M$, there exists a 
complete formula $\phi(x_1\ldots x_n)$, with respect to $Th(\M)$, such that $||\phi(a_1\ldots a_n)||_{\M}^{\bold L}=1.$
\end{enumroman}
\end{definition}
\begin{lemma}\label{sup} Let $\B$ be a $BL$ algebra. Let $X\subseteq B$, be such that for all non-zero $b\in \B$, there exists a non-zero $x\in X$, such that $x\leq a$.
Then $\bigvee X=1$.
\end{lemma}
\begin{demo}{Proof} Let $\bigvee X=b$. Assume that $b<1$. Then $1\cap -b\neq 0$. Let $x\in X$ be non zero below $1\cap -b$. 
Then $x\leq -b$ and $x\leq b$, hence $x\leq b\cap -b=0$ 
which is a contradiction. 
\end{demo}
\begin{theorem}\label{atomic} Let $\Sigma$ be a theory. Then $\Sigma$ has a countable safe witnessed atomic model
if and only if $\Sigma$ is atomic.
\end{theorem}

\begin{demo}{Proof} Assume that $\Sigma$ has an atomic model $\M$.
Let $\phi(x_1,\dots x_n)$ be a consistent with $\Sigma$. Then $\psi=\exists x_1\ldots \exists x_n\phi(x_1,\dots x_n)\in \Sigma$. 
For if not, then since $\Sigma$ is complete, then
$\Sigma\cup \psi\vdash \bot$ and this is impossible since $\phi$ is consistent with $\Sigma.$
Let $a_1,\ldots a_n\in M$ be satisfied by $\phi$. Let $\theta$ be a complete formula satisfiable by $a_1, \ldots a_n$.
Then we have by completeness of $\Sigma$ that $\Sigma\vdash \theta\implies \phi$ or 
$\Sigma\vdash \phi\implies \theta$, but in the second case we will also have 
$\Sigma\vdash \theta\implies \phi$ since $\theta$ is complete.

Now for the converse. In the classical case, the omitting types theorem is used, but the proof depends on negation. Here
we give a direct proof, which is very similar to that of the omitting types theorem. 
Let $\Sigma$ be the given atomic theory. Like in the proof of the omiting types theorem 
add a set $C$ of countably many constants, form a Henkin complete extension $T$ of $\Sigma$,  and let $\B=\Fm_T$.
Now in this present case, we consider the following meets and joins:
Then in $\B$ we have
\begin{equation}\label{e1}
\begin{split}
[\exists x\alpha(x)]=\bigvee_{c\in C} [\alpha(c)]
\end{split}
\end{equation}
\begin{equation}\label{e2}
\begin{split}
[\forall x\alpha(x)]= \bigwedge_{c\in C}[\alpha(c)]
\end{split}
\end{equation}
Let $Fm_n$ denote the set of formulas where at most the $n$ first variables can be free.
Now since the theory is atomic,  if we let $\Gamma_n=\{\phi\in Fm_n: \phi \text { is complete }\}$, then we  have from lemma \ref{sup}
\begin{equation}\label{e3}
\begin{split}
\forall  c_1,\ldots c_n,  \forall i\in \omega, \bigvee_{\phi\in \Gamma_i} [\phi(c_1,\ldots c_n)]=1
\end{split}
\end{equation}

Then for every variable $x$ and formula with one free variable  $\alpha$ we have
$$H_{\alpha,x}=V_M(\exists x\alpha(x))\sim \bigcup_{c\in C} V_M(\alpha(c)$$
$$J_{\alpha,x}=\bigcap_{c\in C} V_M(\alpha(c))\sim V_M({\forall x\alpha x})$$
and for every $i$, 
$$K_{(c_1,\ldots c_n, \Gamma_i)}=Max(\B)\sim \bigcup_{\phi\in \Gamma_i} V_M(\phi(c_1\ldots c_n))$$ 
are nowhere dense sets. As in the omitting types theorem, define $H$, $J$ and $K$.
Then each of these sets is a countable union of nowhere dense sets in $Max(\B)$. 
Let $F$ be a maximal filter in the complement of these sets. Let $T=\bigcup F$.
Then $T$ is a Henkin theory and its canonical model is as desired. Lets check this.  The first two conditions imply that $T$ is a Henkin theory.
Finally, $F\notin K$ means that for all $i\in \omega$ for all $\phi\in \Gamma_i$ there exists $\bar{c}$ such that
$\phi(\bar{c)}\in T$.

\end{demo}

\begin{definition} Let $S$ be a signature. 
\begin{enumarab}
\item Two $S$ structures $(\M_1,{\bold L_1})$ and $(\M_2,{\bold L_2})$ are elementary equivalent, if for each sentence $\phi$, we have
$\M\models \phi$ iff 
$\M\models \phi$.

\item An elementary embedding of $(\M_1,\bold L_1)$ into $(\M_2, \bold L_2)$ is a pair $(f,g)$ such that

(i) $f$ is an injection from $M_1$ into $M_2,$

(ii) $g$ is an embedding of $\bold L_1$ into $\bold L_2,$

(ii) $g(||\phi(a_1\ldots a_n)||_{\M_1}^{\bold L_1}=||\phi(f(a_1)\ldots f(a_n)||_{\M_2}^{\bold L_2}.$

\item $(\M_1,\bold L_1)$ and $(\M_2,\bold L_2)$ are isomorphic if there is an elementary embedding $(f,g)$ such that
$f$ is a bijection and $g$ is an isomorphism.

\end{enumarab}
\end{definition}

From now on, we fix one algebra $\bold L$, so that all models are of the form $(\M, \bold L),$ which we sometimes write as $\M$. 
Furthermore we assume that all models are witnessed.

\begin{definition} Let $T$ be a theory.
\begin{enumerate}
\item  
A set $\Gamma$ in the variables $x_1\ldots x_n$ is consistent
with $T$ if there exists a model $(\M,\bold L)$ of $T$ and $s\in {}^nM$ such that $||\phi(s)||_{M}^{\bold L}=1$ for all $\phi\in \Gamma$.
\item $\Gamma$ is a complete type if it a maximal consistent set, that is $\alpha\notin \Gamma$, then $T\cup \alpha\vdash \bot$.
$S_n(T)$ denotes the set of complete $n$ -types, that is types using only $n$ variables.
\item For a model $(\M,\bold L)$ and $a_1\ldots a_n\in M$, $\tp^{\M}(a_1\ldots a_n)=\{\phi\in Fm_n:||\phi(a_1\ldots a_n)||=1\}.$
We may write $\tp^{\M}(\bar{a})$. We omit the superscript $\M$ when clear from context.
\item For a model $\M=(M,L)$ and $Y\subseteq M$, $\M_Y$ is the model obtained by adding a constant $c_a$ for each $a\in Y$
and interpreting $c_a$ as $a$.
\item A model $(\M,\bold L)$ 
is $\omega$ or countably saturated if for each finite set $Y\subseteq M$, every set $\Gamma(x)$ consistent with $Th(\M_Y)$
is realized in $\M_Y$.
\end{enumerate}
\end{definition}

Notice that $S_n(T)$ is a compact Hausdorff space, since it is the set of maximal filters in the algebra
$\Fm_n/T$ where $\Fm_n$ is the set of formulas which has only at most $n$ free variables.

\begin{theorem} Let $T$ be a complete theory. 
Then $T$ has a countably saturated model if and only if for each $n<\omega$, $T$ has countably many complete $n$ types 
in $n$ variables
\end{theorem}
\begin{demo}{Proof} We prove the harder direction. Add a countable list $\{c_1, c_2\ldots\}$ of new constants forming $\L$. 
For each finite subset $Y\subseteq C$, the types $\Gamma(x)$
of in $L_Y$ are countable.
Let
$$\Gamma_1(x)\ldots, \Gamma_n(x),\ldots $$
be an enumeration of all types of $T$ in all expansions $L_Y$, $Y$ a finite subset of $C$.
Let
$$\phi_1\ldots, \phi_n, \dots$$ be an enumeration of all sentences of $\L$.
Define inductively an increasing sequence
$$T=T_0\subseteq T_1\subseteq T_2\subseteq \ldots $$
of theories of $\L$ such that for each $m<\omega$:
\begin{enumarab}
\item each $T_m$ is consistent and is obtained from $T$ by adding finitely many axioms using only finitey many constants,
\item If $\phi_m=\alpha\to \beta$, then either $\alpha\to \beta\in T_m$  or $\beta\to \alpha\in T_m,$
\item If $\phi_m=\forall x\psi$ is not in $T_{m+1}$ then $\psi(c)$ is not in $T_{m+1},$
\item if $\Gamma_m(x)$ is consistent with $T_{m+1}$ then $\Gamma_m(d)\subseteq T_{m+1}$ for some $d\in C.$
\end{enumarab}
The first three items are like the proof of completeness forming a complete Henkin extension. 
It is clear that the last task can be 
implemented without interfering wth the first three.
For assume inductively that $T'_{m+1}$ have been constructed satisfying (1), (2) and (3). Then if 
$\Gamma_m(x)$ is consistent with $T'_{m+1}$ then one chooses a constant $d$ not occuring in $T_{m+1}$ 
nor $\Gamma_m(x)$, this is possible, since only finitely many constants are in use and puts $T_{m+1}=T'_{m+1}\cup \Gamma_m(d).$ Else he puts
$T_{m+1}=T'_{m+1}.$
The union $T_{\omega}$ is a Henkin complete theory, and its canonical model is as required.
Let $Y\subseteq M$ be finite and let $\Sigma(x)$ be consistent with $Th(\M_Y)$. Then extend $\Sigma(x)$ to a type $\Gamma(x)$ in $Th(\M_Y)$.
Then for some $m$, $\Gamma(x)=\Gamma_m(x) $, and the later is consistent with $T_{m+1}$. Hence 
$\Gamma_m(c)\subseteq T_{m+1}$, then $c$ realizes $\Gamma(x)$ in $\M_Y$..
\end{demo}
Now saturated models are the large models. Now we investigate their dual, the small models.
\begin{definition} 
\begin{enumarab}
\item A model $\M\models T$ is a prime model, if it is elementary embeddable  in every model of $T.$
\item A model $\M\models T$ is atomic, if for every $n\in \omega$, for every consistent formula $\psi$ using $n$ free varibales,
there exists a minimal 
formula $\psi$, also using only $n$ free variables such that $T\models \psi\to \phi$.
Here minimal means that for any formula $\xi$ with only $n$ free variables, 
whenever $T\models \xi\to \psi$,  then $\xi\to \bot$ or $T\models \psi\to \xi$.
\end{enumarab}
\end{definition}

It is easy to see that countable atomic models are prime.  The proof goes like the classical case.

\begin{theorem} Let $\L$ be a countable language and let $T$ be a complete theory. Then the following are equivalent
\begin{enumroman}
\item $T$ has a prime model
\item $T$ has an atomic model
\item The principal types in $S_n(T)$ are dense for all $n$
\end{enumroman}
\end{theorem}
\begin{demo}{Proof} The proof is like the classical case, see \cite{CK}, so we will be sketchy:
\begin{enumerate}
\item $(i)\to (ii)$ Here we use the omiting types theorem. Assume that $\M$ is countably prime. Let $a_1\ldots a_n\in M$ and let
$\Gamma(x_1\ldots x_n)$ be the set of formulas $\mu(x_1\ldots x_n)$ such that 
$$||\mu(a_1\ldots a_n)||_{\M}=1.$$
For any countable model $\B$ of $T$ we have an elementary embedding
$f:\M\to \B$ whence $f(a_1),\ldots f(a_n)$ satisfy $\Gamma$. Therefore $\Gamma$ is realized in every model of $T$.
By the omitting types theorem,  there is a formula $\phi$ that isolates $\Gamma$. 
Then $\phi$ is complete and is satisfied by $a_1\ldots a_n$.

\item $(ii)\to (iii)$ Let $\phi$ be an $\L$ formula such that $[\phi]$ 
is a non empty open set in $S_n(T)$. Let $M\models T$ be atomic. Then, as above since $T$ is complete, we have 
$T\models \exists\bar{v}\phi(\bar{v})$. 
There is an $\bar{a}\in M^n$ such that $\M\models \phi[\bar{a}]$.
Then $\tp^{\M}(\bar{a})\in [\phi]$ and is isolated since $M$ is atomic.
 
\item $(iii)\to (i)$ Suppose that the isolated types in $T$ are dense. Add a countable set of constants to form a Henkin extension of $T$. 
Then $\bigvee_{\phi\in \Gamma_i} [\phi(\bar{c})]=1$
where $\Gamma_n=\{\phi\in Fm_n: \phi\text { is complete }\}$. This follows from the fact that every $[\phi]$ contains a principal type, 
and principal types are generated by complete formulas \ref{max}, since they are 
maximal. In other words every formula is completable.
Next proceed as in the proof of Theorem \ref{atomic} constructing  an atomic hence prime model.
\end{enumerate}
\end{demo}
\begin{corollary} The following are equivalent for a theory $T$.
\begin{enumroman}
\item Every formula is completable.
\item The isolated types are dense in $S_n(T)$ for every $n$.
\end{enumroman}
\end{corollary}
\begin{definition} A complete theory $T$ is $\omega$ categorial iff it has up to isomorphism only one countable model.
\end{definition} 
\begin{theorem}Let $T$ be a complete theory.
Then the following are equivalent:
\begin{enumroman}
\item $T$ is $\omega$ categorial.
\item  $T$ has a model which is both atomic and saturated.
\item  Every type $\Gamma(x_1\ldots x_n)$ is principal.
\item  All countable models of $T$ are atomic.
\end{enumroman}
\end{theorem}
\begin{demo}{Proof} 
\begin{enumarab}
\item $(i)\to (ii)$. Let $\M$ be the unique countable model of $T$. Then $\M$ is countable prime and so is atomic.
Since $T$ has only one countable model, it has a countably saturated model.
Hence $\M$ is countably saturated.

\item $(ii)\to (iii)$ Since $\M$ is $\omega$ saturated, the type $\Gamma$ is realized by some $n$ tuple $a_1\ldots a_n$. 
Since $\M$ is atomic, $a_1\ldots a_n$ satisfies an atomic formula
$\phi$. Clearly $\phi\in \Gamma$.

\item $(v)\to (vi)$ Direct

\item $(vi)\to (i)$ We show that any two models that are atomic and elementary equivalent are isomorphic. But this follows from a back and forth 
argument as in \cite{CK}.

\end{enumarab}
\end{demo}
In the classical case the two more equivalences can be added. That the number of types in $Fm_n/T$ is finite, and that there are finitely 
many formulas modulo $T$ for each $n$.
This follows from the algebraic property of Boolean algebras that if in an algebra all maximal filters are principal, then both the algebra 
and hence the set of maximal filters are finite.
The above does not work for any $BL$ algebra. Consider for example the Heyting algebra, which is an infinite linear order. Then
the algebra has one maximal filter but it is not finite.
However these statements imply the other 4 formulas, but are not equivalent to any of them.

The following results follows like the classical case:

\begin{theorem} Any complete theory $T$ which has a countably saturated model, has a countable atomic model.
\end{theorem}
\begin{demo}{Proof} Using a binary tree argument and lemma \ref{max}. Assume that $T$ has no atomic model. Then $T$ has a 
consistent formula that is not completable. For each incompletable formula we can choose two formulas below it that are incompatible. 
This can be done
infinitely many times giving a tree of incompletable formulas. Each branch gives a consistent set of formulas and there are $^{\omega}2$ 
branches, which can be extended to obtain ${\omega}2$ types contrary to the assumption that it has a saturated model.
\end{demo}


\begin{thebibliography}{100}

\bibitem{AUamal} Sayed Ahmed, T. {\it On Amalgamation of Reducts of Polyadic Algebras.}
Algebra Universalis  {\bf 51} (2004), p.301-359.

\bibitem{Bulletin} Sayed Ahmed, T. {\it Algebraic Logic, where does it stand today?}
Bulletin of Symbolic Logic, {\bf 11}(4) (2005), p. 465--516

\bibitem{neet} Sayed Ahmed T. {\it The amalgmation property, and a problem of Henkin Monk and Tarski}
Journal of Algebra, number theory, advances and applications {\bf 1}(2)(2009) p. 127-141

 \bibitem{n} Sayed Ahmed, T. {\it On neat embeddings of cylindric algebras}
Mathematical Logic Quarterly {\bf 55}(6)(2009)p.666-668




\bibitem{AUneat}  Sayed Ahmed T, {\it Some results about neat reducts} Algebra Universalis, {\bf 1}(2010) p. 17-36.

\bibitem{MLQ} Sayed Ahmed , T. {\it The class of polyadic algebras has the superamalgamation property}
Mathematical Logic Quarterly {\bf 56}(1)(2010)p.103-112 

%\bibitem{s1} Sayed Ahmed T {\it Some results about neat reducts} Algebra Universalis, 17-36 (1) 2010

\bibitem{amal} Sayed Ahmed {\it Classes of algebras without the amalgamation property} Logic Journal of IGPL.
{\bf 1} (2011) 87-104. 


\bibitem{Hung} Sayed Ahmed T. {\it Amalgamation of polyadic Heyting algebras} Studia Math Hungarica, in press.

\bibitem{CK}Chang C.C and Kiesler {\it Model Theory} Studies in Logic and the Foundations of Mathematics, volume 73, 1994.

\bibitem{D} Daigneault,  A., {\it Freedom in polyadic algebras and two theorems 
of Beth and Craig}. Michigan Math. J.{\bf 11}  (1963), p. 129-135.

\bibitem{DM} Daigneault, A., and Monk,J.D., 
{\it Representation Theory for Polyadic algebras}. 
Fund. Math. {\bf 52}(1963) p.151-176.

\bibitem{E} Esteva F, Godo L, {\it Monodial $t$-norm based logic: towards a logic for left continous $t$-norms.} Fuzzy sets and systems
124 (2001) 271-288.

\bibitem{H} Hajek P. {\it Metamathematics of Fuzzy Logic} Trends in Logic. Studia Logica Library, Kluwer academic publishers (1998).

\bibitem{Ha} Hajek P, Cintula P, {\it On theories and models in fuzzy predicate logic}
Journal of Symbolic Logic 71 2006 p.863-880 

\bibitem{F} Fremlin D, , {\it Consequences of $MA$}. Cambridge University press. 
(1984)

\bibitem{b} Gabbay M.D., Maksimova L. {\it Interpolation and Definability: Modal and Intuitionistic Logic}
Oxford Science Publications (2005)


\bibitem{G} Gentzen, G., 1934-5, 
"Untersuchungen Über das logische Schliessen," Math. Zeitschrift 39: 176-210, 405-431.

\bibitem{G} Georgescu G. {\it A representation theorem for polyadic Heyting algebras} 
Algebra Universalis {\bf 14} (1982), 197-209.

\bibitem{Godel} Gödel, K., 1933, "Zur intuitionistischen Arithmetik und Zahlentheorie," Ergebnisse eines mathematischen Kolloquiums 4: 34-38. 

\bibitem{Halmos} Halmos, P., {\it Algebraic Logic.} 
Chelsea Publishing Co., New York, (1962.)

\bibitem{Henkin} Henkin, L., {\it An extension of the Craig-Lyndon interpolation theorem} Journal of Symbolic Logic 28(3) (1963) p.201-216

\bibitem{HMT1}Henkin, L., Monk, J.D., and Tarski, A., {\it Cylindric Algebras Part I}. 
North Holland, 1971.

\bibitem{N} Newelski, L. {\it Omitting types and the real line.}
Jornal of Symbolic Logic, {\bf 52}(1987),  p.1020-1026.


\bibitem{spec} Leustean L. {\it The prime and maximal spectra and the reticulation of $BL$ algebras} Central European Science Journals (2003) 
382-397

\bibitem{fuzzy} Novak, V., Murinova P., {\it Omitting types in Fuzzy Predicate Logics} to appear in Fuzzy sets and systems



\bibitem{HMT2} Henkin, L., Monk, J.D., and Tarski, A., {\it Cylindric Algebras Part II}.
North Holland, 1985.
\bibitem{cat}Herrlich H, Strecker G. {\it Category theory} Allyn and Bacon, Inc, Boston (1973)


\bibitem{H} Heyting, A., {\it Die formalen Regeln der intuitionistischen Logik, in three parts}, Sitzungsber. preuss. Akad. Wiss.: 42-71,(1930) 
158-169.  English translation of Part I in Mancosu 1998: 311-327. 

\bibitem{Hy} Heyting, A., 1956, Intuitionism: An Introduction, North-Holland Publishing, Amsterdam. Third Revised Edition (1971). 

\bibitem{Ho97}  Hodges, W. {\it A shorter Model Theory}. Cambridge. University Press. 1997.
 
\bibitem{J70} Johnson, J.S. {\it Amalgamation of Polyadic Algebras}. 
Transactions of the  American Mathematical Society, {\bf 149}(1970) p.627-652

\bibitem{AUU} Mad\'arasz, J. and Sayed Ahmed T.,
{\it Amalgamation, interpolation and epimorphisms.}
Algebra Universalis {\bf 56} (2) (2007), p. 179 - 210.


\bibitem{P} Pigozzi,D. 
{\it Amalgamation, congruence extension, and interpolation properties in algebras.} 
Algebra Universalis. 
{\bf 1}(1971), p.269-349. 



\end{thebibliography}
\end{document}